\documentclass[11pt]{amsart}
\usepackage{fullpage}
\usepackage{amssymb,amsmath,amsthm,amscd,mathrsfs,graphicx, color}
\usepackage[cmtip,all,matrix,arrow,tips,curve]{xy}

\numberwithin{equation}{section}
\newtheorem{teo}{Theorem}[section]
\newtheorem{pro}[teo]{Proposition}
\newtheorem{lem}[teo]{Lemma}

\newtheorem{ques}[teo]{Question}

\newtheorem{fact}[teo]{Fact}

\theoremstyle{definition}

\newtheorem{que}[teo]{Question}
\newtheorem{notation}[teo]{Notation}

\theoremstyle{remark}
\newtheorem{rem}[teo]{Remark}

\newcommand{\calA}{\mathcal{A}}
\newcommand{\calM}{\mathcal{M}}
\newcommand{\calR}{\mathcal{R}}
\newcommand{\bZ}{\mathbb{Z}}

\newcommand{\KlausA}[1]{}
\newcommand{\SamA}[1]{}
\newcommand{\YanoA}[1]{}
\newcommand{\RaduA}[1]{}

\newcommand{\Sp}{\operatorname{Sp}}
\newcommand{\GL}{\operatorname{GL}}
\newcommand{\Nm}{\operatorname{Nm}}
\newcommand{\Id}{\operatorname{Id}}
\newcommand{\Pic}{\operatorname{Pic}}
\newcommand{\Ext}{\operatorname{Ext}}
\newcommand{\Hom}{\operatorname{Hom}}
\newcommand{\Sym}{\operatorname{Sym}}

\newcommand{\Aut}{\operatorname{Aut}}
\renewcommand{\Im}{\operatorname{Im}}
\newcommand{\End}{\operatorname{End}}
\newcommand{\Gr}{\operatorname{Gr}}
\newcommand{\PD}{\operatorname{PD}}

\newcommand{\fH}{\mathfrak{H}}

\newcommand{\ZZ}{\mathbb{Z}}
\newcommand{\QQ}{\mathbb{Q}}
\newcommand{\RR}{\mathbb{R}}
\newcommand{\CC}{\mathbb{C}}
\newcommand{\TT}{\mathbb{T}}
\newcommand{\PP}{\mathbb{P}}

\newcommand{\AV}{\overline{A}_g^V}
\newcommand{\AP}{\overline{A}_g^P}
\newcommand{\AC}{\overline{A}_g^C}
\newcommand{\oM}{\overline{M}_g}
\newcommand{\oR}{\overline{R}_{g+1}}

\newcommand{\coR}{\overline{\calR}_{g+1}}

\makeatletter
\let\@wraptoccontribs\wraptoccontribs
\makeatother

\begin{document}
\bibliographystyle{amsalpha}

\title[Extending the Prym map]{Extending the Prym map to toroidal compactifications of the moduli space of abelian varieties}

\author[S. Casalaina-Martin]{Sebastian Casalaina-Martin}
\address{University of Colorado, Department of Mathematics, Boulder, CO 80309, USA}
\email{casa@math.colorado.edu}

\author[S. Grushevsky]{Samuel Grushevsky}
\address{Stony Brook University, Department of Mathematics, Stony Brook, NY 11794, USA}
\email{sam@math.sunysb.edu}

\author[K. Hulek]{Klaus Hulek}
\address{Institut f\"ur Algebraische Geometrie, Leibniz Universit\"at Hannover,  30060 Hannover, Germany \newline Current address:
Institute for Advanced Study, 
School of Mathematics, 
1 Einstein Drive, 
Princeton,  NJ 08540,
USA}
\email{hulek@math.uni-hannover.de}
\email{hulek@ias.edu}

\author[R. Laza]{Radu Laza}
\address{Stony Brook University, Department of Mathematics, Stony Brook, NY 11794, USA  \newline Current address:
Institute for Advanced Study, 
School of Mathematics, 
1 Einstein Drive, 
Princeton,  NJ 08540,
USA}
\email{rlaza@math.sunysb.edu}
\email{rlaza@math.ias.edu}

\contrib[with an appendix by]{Mathieu Dutour Sikiri\'c}
\address{Rudjer Boskovi\'c Institute, Bijenicka 54, 10000 Zagreb, Croatia}
\email{mathieu.dutour@gmail.com}

\thanks{Research of the first author was supported in part by NSF grant DMS-11-01333 and Simons Foundation Collaboration Grant for Mathematicians (317572). Research of the second author is supported in part by NSF grant DMS-12-01369. Research of the third author is supported in part by DFG grant Hu-337/6-2. Research of the fourth author is supported in part by NSF grants DMS-12-00875, DMS-12-54812, and by a Sloan Fellowship.}

\begin{abstract}
The main purpose of this paper is to present a conceptual approach to understanding the extension of the Prym map from the space of admissible double covers of stable curves to different toroidal compactifications of the moduli space of principally polarized abelian varieties.  By separating the combinatorial problems from the geometric aspects we can reduce this to the computation of certain monodromy cones. In this way we  not only shed new light on the extension results of Alexeev, Birkenhake, Hulek, and Vologodsky for the second Voronoi toroidal compactification, but we also apply this to other toroidal compactifications, in particular the perfect cone compactification, for which we obtain a combinatorial characterization of the indeterminacy locus, as well as a geometric description up to codimension six, and an explicit toroidal resolution of the Prym map up to codimension four.
\end{abstract}

\date{\today}

\maketitle

\section*{Introduction}
A fundamental tool in the  study of  algebraic curves is the  theory of   Jacobians. Assigning to a curve its principally polarized Jacobian  defines the Torelli period map $M_g\to A_g$ from the coarse moduli space of curves of genus $g$ to the coarse moduli space of principally polarized abelian varieties (ppav) of dimension $g$. It is a well-known  fact, due to Mumford and Namikawa \cite{nam}, that the  Torelli map extends to a morphism  $\oM\to \AV$  from the  Deligne--Mumford compactification  to the second Voronoi toroidal compactification. More recently, Alexeev and Brunyate \cite{ab} have studied extensions of the Torelli map  to other toroidal compactifications and have shown that the period map extends to a morphism to the perfect cone compactification $\AP$, but not to a morphism to the central cone compactification $\overline{A}^{C}_g$ for $g\ge 9$,  disproving a conjecture of Namikawa.

While the Torelli map is injective for all $g$, for $g\ge 4$ it is not dominant. One geometric approach to understanding higher-dimensional ppav is via Prym varieties, which are ppav associated to connected \'etale double covers of curves.  Associating to a cover its principally polarized Prym variety defines the Prym period map $R_{g+1}\to A_g$, where $R_{g+1}$ is the coarse moduli space of connected \'etale double covers of curves of genus $g+1$. The Prym period map is dominant for $g\le 5$, and has been used to  provide a  geometric approach to the Schottky problem for $g=4,5$, to study the rationality of threefolds, and to give a better understanding of the geometry of $A_4$ and $A_5$.

In contrast to the case of Jacobians, it has been known since the work of Friedman and Smith \cite{fs} that the Prym  period map does not extend to a morphism from Beauville's moduli space of admissible double covers $\oR$ to any of the standard toroidal compactifications. Subsequent work of Alexeev, Birkenhake, and Hulek \cite{abh} and Vologodsky \cite{vologodsky}  identifies the indeterminacy locus of the rational map  $\overline{R}_{g+1}\dashrightarrow \AV$; it is the closure of the locus of so-called Friedman--Smith covers with at least $4$ nodes  (see \S \ref{secFSexamples}).  In this paper, we investigate the problem of extending the Prym map to other toroidal compactifications.   Our main results are:

\begin{itemize}
\item  A complete combinatorial characterization of the indeterminacy locus of the Prym map to the perfect and central cone compactifications (Theorem \ref{teoprymext}).   The techniques also give a complete combinatorial characterization of the indeterminacy locus of the Prym map to the second Voronoi compactification, providing another proof of \cite[Thm.~3.2]{abh}.

\item  A  geometric characterization of the indeterminacy locus of the Prym map $\overline R_{g+1}\dashrightarrow \bar A_g^P$ to the perfect cone compactification up to codimension $6$ in $\overline R_{g+1}$  in terms of Friedman--Smith covers  (Theorem \ref{teoindPM}).

\item An explicit resolution of the Prym map $\overline R_{g+1}\dashrightarrow \AP$ up to codimension $4$  (Theorem \ref{teoresPM}).  This also resolves the Prym map to $\AV$ and $\AC$ up to codimension $4$.
\end{itemize}

In  Appendix \ref{secappMDS}, Mathieu Dutour Sikiri\'c also proves an extension result to the central cone compactification (Theorem \ref{teoindC}).

\smallskip

In this paper, we approach the extension problem for the Prym  map in terms of  the Hodge theoretic framework of  a general period map $\calM\to \mathcal D/\Gamma$ from a moduli space to a classical period domain.  This allows us to determine the conditions for extensions of period maps to moduli spaces that are compactified so that the monodromy transformations are of Picard--Lefschetz type (i.e.~given by rank 1 forms).
In this way we  separate the geometric aspects of the problem from the combinatorial issues involved in dealing with various admissible cone decompositions.

In particular, the approach  unifies the arguments for Jacobians and Pryms, and we discuss the Torelli map throughout for motivation.
As a result, we also get a  new proof of   the extension results
of \cite{abh} for $\oR\dashrightarrow\AV$.  In \cite{abh}, the authors have the additional goal of determining compactified Pryms as  stable semiabelic pairs; focusing here on the extension condition allows us to make a more direct, Hodge theoretic  argument.  With the work in \cite{abh},  translating from our results to the language of stable semiabelic pairs is straightforward (\S \ref{secModStAbVar}, \S \ref{secsubFibers}).
In addition, one of our original motivations for this work was investigating the extension of the period map for cubic threefolds  to a morphism from a suitable GIT compactification of the moduli space of threefolds to a suitable compactification of $A_5$, stemming from our work \cite{cml} and \cite{cml2}, and using some of the results of our work \cite{gh}. The  methods we use in this paper apply  in that setting also, and we will return to the study of the period map for cubic threefolds in  subsequent work.

\smallskip

A few words about the structure of the paper. We start in Section \ref{secttoroidal} by reviewing some basic facts about the toroidal compactifications (second Voronoi, perfect, central) that we consider in our paper. We then discuss (Section \ref{sectHT}) the general  framework of degenerations of Hodge structures and the connection to toroidal compactifications. This is mostly standard (see eg.~\cite{cattani} for an exposition), but we find it convenient to include a short discussion of this adapted to our needs. In Section \ref{sectPrym}, we briefly review the standard compactification of the moduli of Prym varieties by admissible covers (\cite{b}) and the associated combinatorial data (graphs with an involution, etc.).
In Section  \ref{sectMon}, we specialize the discussion of Section \ref{sectHT} to curves and Prym varieties and discuss the computation of the monodromy cones in terms of the dual graph.
The monodromy cone for  Jacobians is classical (eg.~\cite{nam76II}) and that of Pryms is essentially contained in \cite{fs} and \cite{abh}.  Nonetheless, we believe that our presentation unifies, simplifies, and clarifies some of the arguments in the literature.  Our goal  will be to apply similar techniques to the study of other moduli spaces via Hodge theory in the future.

With these preliminaries, new results start in Section \ref{sectExt}, where we recast the extension criteria for the Torelli map, and then prove combinatorial criteria, in terms of the dual graph, for the extension of the Prym map to various toroidal compactifications of $A_g$, obtaining Theorem \ref{teoprymext} and  thus giving in addition  a new  proof of \cite[Thm.~3.2]{abh}.
We then proceed to relate these combinatorial conditions to geometric conditions on the admissible covers.  The so-called Friedman--Smith covers are central to this discussion and we describe in Section \ref{secFSexamples} their monodromy in detail: in Subsection  \ref{secFSMonCone} we compute the monodromy cones, and in Theorem \ref{teoFSMCP}
we discuss their properties with respect to the fans defining different toroidal compactifications.
In Section \ref{sectindeterm}, we use these computations to describe
the indeterminacy locus of the Prym map geometrically,
and it is interesting to note that this behavior for the perfect cone
compactification is  quite different from that for 
the second Voronoi compactifictaion.  We are
able to give a complete geometric characterization of the indeterminacy locus of the Prym map to
the perfect cone compactification $\bar A_g^P$  up  to codimension $6$ (Theorem \ref{teoindPM}), utilizing the recent results of Melo and Viviani \cite{MV12}.

The computations also allow us to describe
the resolution of the period map in terms of explicit, toroidal
modifications of the moduli space of admissible covers.  In Section \ref{sectRes} we describe
the resolution of the period map to the perfect cone
compactification completely up to codimension $4$ (Theorem
\ref{teoresPM}).
In Section \ref{secsubFibers} we start a discussion on the fibers of the Prym map. More precisely, we discuss which types of admissible covers are mapped to which
strata. This also provides another link to  \cite{abh} since we discuss the relationship  between the monodromy cones and the degeneration data of $1$-parameter families, which in turn
determine semiabelic varieties which are limits of Pryms.

Many of the arguments in the paper regarding the Prym map in low codimension rely on working through a
number of examples, and explicit computations of monodromy cones. These are somewhat
lengthy and technical, and to maintain the structural unity of the
argument we collect these explicit computations in the appendices.
Appendix \ref{seccombinatorics} treats the combinatorics of the Friedman--Smith cones and relates these to the various cone decompositions.
In Appendix \ref{secexamples} we discuss some examples where the Prym map extends; this comes down to proving that certain monodromy cones
belong to either the  second Voronoi, perfect cone or central cone decomposition.
Appendix \ref{secDegen} contains some lengthy calculations where we discuss further degenerations of Friedman--Smith examples. In particular we
compute their monodromy cones and discuss  to which, if any, cone decompositions these belong.
Finally, in Appendix \ref{secSimp} we discuss a method which allows us to simplify certain monodromy cones and thus to reduce to previous
calculations.

\subsection*{Acknowledgements} We are very grateful to Mathieu Dutour Sikiri\'c who was always willing to answer our questions on
cone decompositions and who helped us check some of our guesses with his powerful computer programs.

\subsection*{Notation}
We will use calligraphic letters to refer to moduli stacks (e.g.~$\mathcal A_g$, $\mathcal R_{g+1}$, etc.), and Roman  letters for the associated coarse moduli spaces (e.g.~$A_g$, $R_{g+1}$, etc.). Since all the spaces occurring here (with the exception of Alexeev's stack of stable semiabelic pairs) are Deligne--Mumford stacks, all the period maps are assumed to be locally liftable, and the extensions are insensitive to finite covers, there is  essentially no difference between using stacks or the associated coarse moduli space. In fact, we will typically stick to  the coarse moduli space, except for the situations where we want to emphasize the modular meaning.


\section{Brief review of toroidal compactifications}\label{secttoroidal}
In this section, we briefly review the theory of toroidal compactifications of $A_g$ (see \cite{AMRT}, \cite{nam} and \cite{FC90} for more details), focusing on the  three classically known
toroidal compactifications (up to refinement of the fans, i.e.~blow-ups), that is the perfect cone (also known as first Voronoi), second Voronoi, and central cone compactification. Primarily the purpose here is to fix the notation and terminology needed later.

\begin{notation}
As is customary, when necessary, we will use subscripts (eg.~$H_\ZZ$) to indicate the coefficients for modules and algebraic groups. Unless specified, the coefficients are either $\QQ$ or $\RR$.
\end{notation}

\subsection{The Satake--Baily--Borel Compactification} Fix a free abelian group $H$ of rank $2g$, and a non-degenerate, skew-symmetric, bilinear form $Q$ on $H$.  We let $D$ be  the classifying space of polarized weight $1$ Hodge structures on $H$:
$$
  D:=\{F\in \operatorname{Grass}(g,H_\CC): Q(F,F)=0, \ \ iQ(F,\overline F)>0\}\cong G_\RR/K,
$$
where $G_\RR\cong \Sp(2g,\RR)$ and $K=U(r)$ is the maximal compact subgroup. Taking $Q$ to be the standard symplectic form, $D$ can be  (canonically) identified with the Siegel upper half-space $\fH_g$, the space of symmetric $g\times g$ complex matrices with positive definite imaginary part.   The fractional linear transformations give an action of $G_\bZ=\Sp(2g,\ZZ)$ on $D\cong \fH_g$, and we set
$$A_g:=\fH_g/\Sp(2g,\ZZ).$$
The Satake--Baily--Borel (SBB) compactification $A_g^*$
is a normal, projective compactification of $A_g$  that admits a stratification:
$$A_g^*=A_g\sqcup A_{g-1}\sqcup\ldots\sqcup A_{0}.$$

We recall that $A_g^*$ and the above stratification are obtained (set-theoretically) by adding to $D$ the so called rational  boundary components $F_{W_0}$, and then taking the quotient  with respect to the natural $G_\bZ=\Sp(2g,\bZ)$  action.
Namely, the rational  boundary components $F_{W_0}$ of $D$ correspond to the choice of rational maximal parabolic subgroups $P_{W_0}\subset\Sp(Q,H_\QQ)$, which in turn correspond to the choice of a totally isotropic subspace $W_0\subseteq H_\QQ$ (of which $P_{W_0}$ is then the stabilizer).  Note that since  $\Sp(2g,\ZZ)$ acts transitively on the set of isotropic subspaces $W_0$ of $H_\QQ$ of fixed dimension, the set of rational boundary components is essentially indexed by the $\nu(=\dim W_0)\in\{0,\dots, g\}$. Furthermore, the choice of $W_0$ defines a weight filtration on $H_{\QQ}$:
\begin{equation}\label{eqfiltration}
  W_{-1}:=\{0\}\subseteq W_0\subseteq W_1:=(W_0)_Q^\perp\subseteq W_2:=H_{\QQ}.
 \end{equation}
The polarization $Q$ induces a polarization (non-degenerate symplectic form) $\bar Q$ on $\Gr_1^W=W_1/W_0$. It is then standard (eg.~\cite[p.84]{cattani}) that the boundary component $F_{W_0}$ is the classifying space $D_{g'}$ with ($g'=g-\nu$) of $\bar Q$-polarized Hodge structures on $\Gr_1^W$, giving the component $A_{g'}=F_{W_0}/G_{\ZZ}$ of $A_g^*$. (N.B. $F_{\{0\}}=D$, and after the identification $F_{W_0}=D_{g'}=\mathfrak H_{g'}$, the action of $G_{\mathbb Z}$ restricts to the action of $\operatorname{Sp}(2g',\mathbb Z)$.)

\subsection{Toroidal compactifications}
The toroidal compactifications are certain refinements
of the SBB compactification    $A_g^*$, depending  on a choice of a compatible collection of admissible cone  decompositions,  $\Sigma$.  Each such choice gives a compactification $\overline{A}_g^{\Sigma}$ with a  canonical map $ \overline{A}_g^{\Sigma}\to A_g^*$.
Here we review a few points about the construction from the perspective of Hodge theory (essentially following \cite{cattani}).

The construction is relative over $A_g^*$, and one starts by considering a totally isotropic subspace $W_0\subseteq H_{\QQ}$ of dimension $\nu\le g$ and the corresponding boundary component of $A_g^*$. Consider then the real Lie subalgebra of $\mathfrak{sp}(Q,H_{\RR})$ preserving $W_0$:
$$
  \mathfrak n(W_0):=\{N\in  \mathfrak{sp}(Q,H_{\RR})\mid \mathrm{Im}(N)\subseteq W_0\}.
$$
Then for any $N\in \mathfrak n(W_0)$ we have $N^2=0$, and thus $N$ defines a weight filtration compatible with that induced by $W_0$, see \eqref{eqfiltration}. In other words, we have
$$
 \Im(N)=W_0(N)\subseteq W_0\subseteq W_1=W_0^\perp \subseteq W_1(N)=\ker (N)=\Im(N)^\perp,
$$
and, in particular, a natural surjection
\begin{equation}\label{eqnsurjection}
\Gr_2^W(:=W_2/W_1)\twoheadrightarrow \Gr_2(N)(:=W_2(N)/W_1(N)).
\end{equation}
Furthermore, since $N$ is a nilpotent symplectic endomorphism, we get a natural isomorphism
\begin{equation}\label{eqncomp}
\begin{CD}
\Gr_2(N) @>N>> \Gr_0(N) @>Q(N(\cdot),\cdot)>> \Gr_2(N)^\vee\\
v@>>> N(v)@>>>Q(N(\cdot),v),
\end{CD}
\end{equation}
which can be interpreted as giving a non-degenerate bilinear form $Q_N$ on $\Gr_2(N)$. The form $Q_N$ turns out to be symmetric, and by pullback can be viewed as a form on  $\Gr_2^W$; thus there is a natural map (defined over $\QQ$)
\begin{equation}\label{eqnW}
\mathfrak n(W_0)\stackrel{\sim}{\longrightarrow} \Hom(\Sym^2 \Gr_2^W,\RR),
\end{equation}
which (it is not hard to see) is an isomorphism.

As described above, $\mathfrak n(W_0)$ is canonically identified with the Lie algebra of symmetric bilinear forms (or equivalently symmetric $g'\times g'$ matrices, with $g'=g-\nu$) on $\Gr_2^W$. With this identification, we consider the cone of positive definite $g'\times g'$ symmetric matrices
$$
  \mathfrak n(W_0)^+:=\{N\in \mathfrak n(W_0)\mid  Q_N \textrm{ is positive definite}\}.
$$
Let $\Sigma$ be a  compatible collection of admissible cone decompositions (see \S \ref{secAdCD}).
Now for each cone $\sigma_{W_0}\in \Sigma_{W_0}$, there is an associated space $B(\sigma_{W_0})$ together with a map $B(\sigma_{W_0})\to F_{W_0}$, where $F_{W_0}$ is the rational boundary component associated to $W_0$ (see eg.~\cite[p.91]{cattani}).    These maps are compatible in the sense that if $\tau_{W_0}\le \sigma_{W_0}$ is a face, then there is a commutative diagram
$$
\xymatrix@C=.5cm@R=.5cm{
B(\tau_{W_0}) \ar@{->}[rr] \ar@{->}[rd]& & B(\sigma_{W_0}) \ar@{->}[ld]\\
&F_{W_0}&
}
$$
One then sets $D^{\Sigma}=\bigcup_{W_0}\bigcup_{\sigma_{W_0}\in \Sigma_{W_0}}B(\sigma_{W_0})$.  The action of $G_{\ZZ}=\Sp(2g,\ZZ)$ extends to an action on $D^{\Sigma}$, and then (set-theoretically) $\bar A^\Sigma_g=D^{\Sigma}/G_{\ZZ}$, inducing also a natural map $\bar A^\Sigma_g\to A^*_g$.

\subsection{Admissible cone decompositions for quadratic forms}\label{secAdCD}
We now review some basic terminology and results about cone decompositions.
Let $\Lambda$ be a  free $\ZZ$-module of rank $g$.   The space of quadratic forms on $\Lambda$ is $(\Sym^2 \Lambda)^\vee$, which comes equipped with a natural diagonal action of $\GL(\Lambda)=\Aut_{\ZZ}(\Lambda)$. One considers the open cone of positive definite quadratic forms
$$C\subset (\Sym^2 \Lambda)^\vee\otimes_\ZZ \RR,$$
and then lets $\overline{C}^\QQ$ be the rational closure.
Obviously, $C$ and $\overline{C}^\QQ$ are  $\GL(\Lambda)$-invariant. For any subgroup $\Gamma\subseteq \GL(\Lambda)$ (typically we will be interested $\Gamma= \GL(\Lambda)$), a $\Gamma$-admissible rational polyhedral decomposition $\Sigma$ (in short {\it admissible decomposition}) of $C$  is a $\Gamma$-invariant collection of (rational, convex, polyhedral)  subcones covering $\overline{C}^\QQ$ which satisfies certain natural axioms (see \cite{nam} or \cite[Ch.~IV, Def.~2.2, p.96]{FC90} for details), most notably the requirement that there are only finitely many orbits of cones of $\Sigma$ modulo the action of $\Gamma$.

For the construction of the toroidal compactifications $\overline{A}_g^\Sigma$ one requires an admissible decomposition for the space of quadratic forms associated to each isotropic subspace $W_0$ (see \eqref{eqnW}). As discussed, all isotropic subspaces $W_0$ of fixed dimension are conjugate, and thus what one needs is an admissible decomposition for each lattice $\Lambda'$ of rank  $0\le g'\le g$, compatible in the following sense. We say that $\Sigma'$ and  $\Sigma$ are compatible if there exists a surjection $\Lambda \twoheadrightarrow \Lambda'$ so that $\Sigma'$ is obtained from $\Sigma$ via pull-back by the natural inclusion $\overline C^{\QQ}(\Lambda')\subseteq \overline C^{\QQ}(\Lambda)$.  If this is the case for one surjection $\Lambda\twoheadrightarrow \Lambda'$, it will be true for all surjections.    In particular, specifying an admissible decomposition for $\Lambda$ then specifies uniquely compatible admissible decompositions for all lattices $\Lambda'$ of smaller rank. In short,  all we need to define a toroidal compactification $\bar A^\Sigma_g$ is  an admissible cone decomposition for the rank $g$ lattice.

Three admissible decompositions are classically  known for $A_g$, namely the so called second Voronoi, the perfect cone (or first Voronoi), and the central cone
decomposition (these can, of course, be further subdivided). These decompositions
are discussed in \cite[\S8, \S9]{nam}. We shall address in this paper all three decompositions and the associated toroidal compactifications.
Though we will not review their definitions (the interested reader should see \cite{nam}), we will discuss the relevant facts about them in the following subsection.
There is also another admissible decomposition known, namely that into $C$-types \cite{RB}, which is less known to algebraic geometers. This coincides with the second Voronoi
decomposition for $g\leq 4$, but for $g\geq 5$ second Voronoi is a proper refinement of the $C$-type decomposition. To our knowledge no geometric interpretation
of the corresponding toroidal compactification is known.

\smallskip

Finally, we recall some terminology.  A cone $\sigma\subseteq \overline C^{\mathbb Q}$  is called \emph{basic} if the integral generators of its $1$-dimensional faces can be completed to a $\mathbb Z$-basis of $(\operatorname{Sym}^2\Lambda)^\vee$. It is called \emph{simplicial} if these generators can be completed to a $\mathbb Q$-basis; i.e.~if the generators are linearly independent.

\subsection{Admissible cone decompositions and rank $1$ quadrics}
In the geometric context of our paper, we will only be interested in cones spanned by rank $1$ quadrics (i.e.~squares of linear forms), since  our (log of) monodromy operators will be rank one. For such cones it is essentially a combinatorial problem to decide if they belong to the second Voronoi, perfect, or central cone decompositions. These results are well known and we will refer the reader to \cite{ab} and \cite{MV12} for further details.

For  $\ell_1,\ldots,\ell_n\in \Lambda_{\mathbb R}^\vee\setminus\lbrace 0\rbrace$, let $\sigma:=\mathbb R_{\ge 0}\langle \ell_i^2\rangle_{i=1}^n$ be the corresponding cone generated by rank $1$ quadrics in $\operatorname{Sym}^2\left(\Lambda_{\mathbb R}^\vee\right)$.    Given a basis for $\Lambda$, we will often refer to the cone $\sigma$ by writing the matrix whose $i$-th row is the expression for $\ell_i$ in terms of the dual basis to the given basis,
and to any such matrix will associate such a cone.

In this setup, we then have the following combinatorial results that determine whether a set of linear forms in $\Lambda^\vee$ generate a cone contained in a cone of one of the three standard admissible decompositions.
\begin{lem}[Second Voronoi] \label{lemsecvor}
Let $\Lambda$ be free $\ZZ$-module of  rank $g$.  Suppose $\ell_1,\ldots,\ell_n\in \Lambda^\vee$ are primitive, non-zero, linear forms.
The following are equivalent:
\begin{enumerate}
\item   $\{\ell_1^{2},\ldots, \ell_n^{2}\}$ lie in a common cone of the second Voronoi decomposition.
\item   $\RR_{\ge 0}\langle \ell_1^{2},\ldots, \ell_n^{2}\rangle$ is a cone in the second Voronoi decomposition.
\item Any $\RR$-linearly independent subset $\{\ell_j\}_{j\in J}\subseteq \{\ell_1,\ldots,\ell_n\}$, is a $\ZZ$-basis of the $\mathbb Z$-module $ \RR\langle \ell_j\rangle_{j\in J}\cap \Lambda^\vee$.
\item  Any $\RR$-linearly independent subset $\{\ell_j\}_{j\in J}\subseteq \{\ell_1,\ldots,\ell_n\}$ of maximal rank, is a $\ZZ$-basis of the $\mathbb Z$-module  $ \RR\langle \ell_j\rangle_{j\in J}\cap \Lambda^\vee $.
\end{enumerate}
\end{lem}
\begin{proof}
This is well known.  We direct the reader to \cite[Lem.~4.5]{ab} and the references therein.
\end{proof}

One may take as a definition that a \emph{matroidal} cone is a second Voronoi cone  generated by rank $1$ quadrics (this is essentially the content of Lemma \ref{lemsecvor}).   It follows from the lemma that a face of a matroidal cone is matroidal, and moreover, that matroidal cones are simplicial.    We denote by  $\Sigma_{\text{mat}}\subseteq\Sigma_V$ the collection of matroidal cones.

To connect the discussion with that of \cite{abh}, we recall the notion of a dicing.  Fix a collection of codimension-$1$ affine spaces  $\{H_i\}_{i\in I}$ in $\Lambda_{\RR}$.  Let $\mathscr H=\bigcup_{i\in I}H_i$ be the associated arrangement of affine spaces.  The arrangement $\mathscr H$ is stratified by the intersections of the $H_i$.  We say that  $\mathscr H$ defines a \emph{dicing} of $\Lambda$ if the union of the $0$-dimensional strata of $\mathscr H$ is exactly the lattice $\Lambda$.

\begin{lem}\label{lemdice}
Let $\Lambda$ be a free $\ZZ$-module of rank $g$.  Suppose that  $\ell_1,\ldots,\ell_n\in \Lambda^\vee$ are $\RR$-linearly independent.  Then $\ell_1,\ldots,\ell_n$ form a $\ZZ$-basis for $\Lambda^\vee$ if and only if they determine a dicing of $\Lambda _{\RR}$.  More precisely, this means that  the collection of hyperplanes
$$
  H_{i,m}:=\{x\in \Lambda_{\RR} : \ell_i(x)=m\}
$$
with $i=1,\ldots,n$ and $m\in \ZZ$ defines  a dicing of  $\Lambda$.
\end{lem}
\begin{proof}
This follows from the definitions and is left to the reader.
\end{proof}

\begin{rem}\label{remDelaunay}
Associated to a  quadratic form $q\in C$ is a so-called Delaunay decomposition of  $\Lambda\otimes _{\mathbb Z}\mathbb R$.  The second Voronoi decomposition is defined so that the Delaunay decomposition of a quadric remains unchanged for all quadrics in a given (open) second Voronoi cone.  We will only be interested in quadratic forms that lie in second Voronoi cones generated by rank $1$ quadrics.  In this case, the Delaunay decomposition has a well-known, and simple description (see \cite[Theorem 3.2]{ER2} or the proof of \cite[Lem.~3.1]{abh}):
 \emph{If  $\ell_1,\ldots,\ell_n \in \Lambda^\vee$, span  $\Lambda^\vee_{\mathbb R}$,  and $\sigma =\RR_{\ge 0}\langle \ell_1^{2},\ldots, \ell_n^{2}\rangle$ is a second Voronoi cone, then the Delaunay decomposition for any (positive definite) quadric $q\in \sigma^\circ$ is given by the (dicing) hyperplane arrangement  associated to $\ell_1,\ldots,\ell_n$.}
\end{rem}

\begin{lem}[Perfect cone] \label{lempc}
Let $\Lambda$ be a free $\ZZ$-module of rank $g$.    Suppose $\ell_1,\ldots,\ell_n\in \Lambda^\vee$ are primitive, non-zero, linear forms.  The following are equivalent.
\begin{enumerate}
\item $\{\ell_1^{2},\ldots, \ell_n^{2}\}$ lie in the same  cone of the perfect cone decomposition.
\item There exists a quadratic form $Q$ on $\Lambda^\vee_{\RR}$ such that
\begin{enumerate}
\item $Q(\ell)>0$ for all $\ell \in \Lambda^\vee_{\RR}\setminus \{0\}$; i.e.~$Q$ is positive definite.
\item $Q(\ell)\ge 1$ for all $\ell \in \Lambda^\vee \setminus \{0\}$.
\item $Q(\ell_i)=1$, $i=1,\ldots,n$.
\end{enumerate}
\end{enumerate}
\end{lem}

\begin{proof}
This follows from the definition of the perfect cone decomposition in  \cite{nam}.  (See also the \emph{proof} of \cite[Thm.~4.7]{ab}.)
\end{proof}

\begin{rem}\label{remMV}
Since cones in the perfect cone decomposition are generated by rank $1$ quadrics, a cone in the perfect cone decomposition is a second Voronoi cone if and only if it is matroidal (i.e.~$\Sigma_P\cap \Sigma_V\subseteq \Sigma_{\text{mat}}$).  Recently Melo and Viviani  \cite[Thm.~A]{MV12} showed that matroidal cones are in the perfect cone decomposition  (i.e.~$\Sigma_{\text{mat}}\subseteq \Sigma_P$), establishing that
$\Sigma_V\cap\Sigma_P=\Sigma_{\text{mat}}$. Note in particular that the following  special  case of \cite[Thm.~A]{MV12} follows directly from the definitions and Lemma \ref{lempc}:  \emph{if $\sigma\in \Sigma_{\text{mat}}$ is  generated by at most $g$ rank 1 quadratic forms, then $\sigma\in \Sigma_P$.  In particular,  if $q\in \sigma\in \Sigma_P$ is a rank $1$ quadric, then $\RR_{\ge 0}\langle q\rangle$ is a face  of $\sigma$.}
\end{rem}

\begin{lem}[Central cone]\label{lemcc}
Let $\Lambda$ be a free $\ZZ$-module of rank $g$.    Suppose  $\ell_1,\ldots,\ell_n\in \Lambda^\vee$ are primitive, non-zero, linear forms.  The following are equivalent.
\begin{enumerate}
\item $\{\ell_1^{2},\ldots, \ell_n^{2}\}$ lie in the same  cone of the central  cone decomposition.
\item There exists a quadratic form $Q$ on $\Lambda^\vee_{\RR}$ such that
\begin{enumerate}
\item $Q(\ell)>0$ for all $\ell\in \Lambda^\vee_{\RR}\setminus \{0\}$; i.e.~$Q$ is positive definite.
\item $Q(\ell)\ge 1$ for all $\ell \in \Lambda^\vee\setminus \{0\}$.
\item $Q(\ell_i)=1$, $i=1,\ldots,n$.
\item $Q(\ell)\in \ZZ$ for all $\ell\in \Lambda^\vee$.
\end{enumerate}
\end{enumerate}
\end{lem}

\begin{proof}
This follows from the definition of the central cone decomposition in  \cite{nam}.  (See also the \emph{proof} of \cite[Thm.~4.8]{ab}.)
\end{proof}

\begin{rem}
We note that all but the last condition above are the same as for the perfect cone compactification, and thus it turns out that {\em if a collection of \emph{rank 1} quadratic forms lies in a central cone, they also lie in a perfect cone}, but not vice versa
(see also Remarks \ref{remTorPC} and \ref{remTorCC} below).
 \end{rem}

Given an admissible cone decomposition $\Sigma$, we will denote by  $\Sigma^{(1)}$ the collection of cones that are generated by rank $1$ quadrics.
Note that if $\sigma\in \Sigma^{(1)}$ and $\tau$ is a face of $\sigma$, then $\tau\in \Sigma^{(1)}$.     Note also that by definition $\Sigma_P=\Sigma_P^{(1)}$.  We can summarize the discussion above as follows.
$$
 \sigma \in \Sigma^{(1)}_V\   (=\Sigma_{\text{mat}})\   \text { or }
\sigma \in \Sigma^{(1)}_C\   \implies \sigma \in \Sigma_P\ \ (=\Sigma^{(1)}_P)
$$

\begin{rem} \label{remlowdim}
The metrics
\begin{equation}\label{Acone}
  Q_A(\underline x):=\sum_{1\le i\le j \le n} x_ix_j,
\ \ \ \ \ \ \ Q_D(\underline x):=\sum_{1\le i\le j \le n, (i,j)\ne (1,2)} x_ix_j
\end{equation}
define cones of type $A$ and $D$ respectively in the perfect cone decomposition  (in fact also in the central cone) decomposition.  Cones of type $A$ are matroidal, whereas for $n\ge 4$, the type $D$ cones are not (and also fail to be simplicial).
\end{rem}

\begin{rem} \label{relationsconedec}
At this point we would like to recall the relation between the three known admissible decompositions. For $g= \operatorname{rank}\Lambda$,  $g\leq 3$, all three decompositions (namely the
second Voronoi, perfect cone and central cone) coincide. For $g=4$ it is still true that the perfect cone and the central cone decomposition coincide and the second
Voronoi decomposition is a refinement of these. More precisely the only non-basic cone of the perfect cone decomposition, namely the $D_4$ cone, is
subdivided into basic cones in the second Voronoi decomposition (see \cite{ER} for details). For $g=5$  the second Voronoi decomposition is still a refinement of the  perfect cone decomposition (\cite{RB}), but this is
no longer the case for $g\geq 6$ (\cite{EB}). In general all three decompositions are different in the sense that none is a refinement of the other.
\end{rem}

\section{Monodromy cones and extensions to toroidal compactifications}\label{sectHT}
The central question addressed in this paper is the question of extending the period map for Prym varieties to toroidal compactifications.  The basic set-up for such a problem is that of a locally liftable map $\mathcal P:B^\circ\to D/\Gamma$ from a smooth base $B^\circ$ to a locally symmetric variety (eg.~maps arising from weight $1$ VHS associated to families of varieties $\mathfrak X^\circ / B^\circ$). We then consider a partial simple normal crossing smooth compactification $B^\circ\subset B$ and we are asking for extensions of the map $\mathcal P$ from $B$ to a given (fixed) toroidal compactification  $\overline{D/\Gamma}^\Sigma$. Since the problem is essentially local, we may assume without loss of generality  that $B^\circ$ is a polycylinder (i.e.~$B^0=(S^\circ)^k\times S^{n-k}\subset B=S^n$, where $S^\circ=S\setminus\{0\}$ and $S$ is the unit disk), and that the monodromy operators  around the boundary divisors are unipotent.

With this set-up the extension question has an elegant answer. Namely, one defines a monodromy cone associated to the period map $\mathcal P$, and then $\mathcal P$ extends if and only if the monodromy cone is compatible with the cones of the admissible decomposition $\Sigma$. We review this below, following Cattani \cite{cattani},  with a focus on weight $1$ Variation of Hodge Structure (although some of the considerations apply more generally).

\subsection{Degenerations of weight $1$ Hodge structures}
The monodromy cone for a variation of Hodge structures is a basic tool in understanding extensions of period maps.  Here we review the definition of the log of monodromy, the monodromy cone, and the connection with quadratic forms.

\subsubsection{The log of monodromy}
We focus on the case of weight $1$ Hodge structures for simplicity. Let  $\pi^\circ:\mathfrak X^\circ \to S^\circ$, be a smooth,  projective morphism over the punctured disk $S^\circ$. Fix a  base-point $\ast \in S^\circ$, with fiber $X_\ast=(\pi^\circ)^{-1}(\ast)$,   and let $T$ be the associated monodromy operator on $H^1(X_\ast,\QQ)$. It is well known that $T$ is quasi-unipotent;
in fact since we are in weight $1$, we have  $(T^n-Id)^2=0$. For simplicity, we will assume further that $T$ is in fact unipotent; i.e.~$(T-Id)^2=0$.
Since unipotent monodromy can be obtained after a finite base change, this assumption will not affect extension questions (see Remark \ref{remFBC}). Thus
\begin{equation}\label{eqnlm1}
N=\log T=T-Id\in \End H^1(X_\ast,\QQ)
\end{equation}
is the log of monodromy operator. Note that $N\in \mathfrak{sp}(H,Q)$, where $H=H^1(X_\ast,\QQ)$ and $Q$ is intersection pairing on $H$, and $N^2=0$.

To relate with the discussion of Section \ref{secttoroidal}, we would like to view $N$ as a quadratic form.  To this end we recall that there is a
 limit polarized, mixed Hodge structure $H^1_{\lim}=H^1_{\lim}(N)$ on the torsion free quotient   $H^1(X_\ast,\ZZ)_\tau$, where the weight filtration $W_\bullet=W_\bullet(N) $ is defined (using $N^2=0$) by
 \begin{equation}\label{eqnNwf}
W_{-1}=\{0\}\subset W_0=\Im(N)\subseteq W_1:=\ker (N) \subseteq W_2:=H^1(X_\ast,\QQ).
\end{equation}
As in \eqref{eqncomp} and \eqref{eqnW} (which are essentially linear algebra statements about nilpotent symplectic endomorphisms), we can view
 the log of monodromy as a map
\begin{eqnarray}\label{eqnlm3}
Q(N(\cdot),\cdot ):\Gr_2(N) &\to& (\Gr_2(N))^\vee \in \Hom(\Sym^2 \Gr_2(N),\QQ)\\
 \bar z &\mapsto& Q(N(\cdot),z),\nonumber
\end{eqnarray}
or equivalently as a symmetric bilinear form on $\Gr_2(N)$.

\begin{rem} Since we will need to explicitly compute monodromy in several cases, we note that with respect to a suitable symplectic basis on $H^1(X_\ast)$, we can write (eg.~\cite[Prop.~4.8]{nam})
\begin{equation*}
T=\left(
\begin{smallmatrix}
1_{g'}&0&0&0\\
0&1_{\nu}&0&b\\
0&0&1_{g'}&0\\
0&0&0&1_{\nu}\\
\end{smallmatrix}
\right)  , \ \ N=\log T=\left(
\begin{smallmatrix}
0&0&0&0\\
0&0&0&b\\
0&0&0&0\\
0&0&0&0\\
\end{smallmatrix}
\right),
\end{equation*}
with $b$ a symmetric non-degenerate $\nu\times \nu$ matrix, $\nu=\dim W_0=\Im(N)$ and $g'=g-\nu$. The identification of $N$ with a quadratic form is simply obtained by considering the matrix $b$. The salient point of the discussion above is that $b$ should be viewed a quadric form on $\Gr_2(N)$ which is essential for compatibility issues as discussed below.
\end{rem}

\begin{rem}
To a $1$-parameter unipotent degeneration of weight $1$ Hodge structures, one can associate either  a limit Mixed Hodge structure (from the point of view of degenerations of Hodge structures following Schmid \cite{schmid} and Steenbrink \cite{steenbrink})
or a semiabelian variety (see \S \ref{secModStAbVar}).
The two limit objects are canonically identified via the functorial equivalence of categories between semiabelian varieties and polarized weight $1$ MHS (e.g.~Deligne \cite[\S10]{deligne3}).
From the perspective of the monodromy matrices discussed above, the $g'\times g'$ blocks correspond to the compact part of the limit semiabelian variety and are essentially irrelevant to the extension question. On the other hand, the $\nu\times \nu$ matrix $b$ defining the quadratic form is a key ingredient for extension questions.
\end{rem}

\subsubsection{Monodromy cones}\label{smoncone}
 We now consider families over higher dimensional bases. Let $\pi^\circ:\mathfrak X^\circ \to (S^\circ)^k\times S^{n-k}$ be a smooth, projective morphism.   Fix a  base-point $\ast \in (S^\circ)^n$, with fiber $X_\ast=(\pi^\circ)^{-1}(\ast)$,   and let $T_i$ ($i=1,\ldots,k$)  be the associated monodromy operators on $H^1(X_\ast,\QQ)$; i.e.~generators for the induced homomorphism $\ZZ^k\cong \pi_1((S^\circ)^k,\ast)\to \Aut H^1(X_\ast,\QQ)$.     For simplicity, as before, we will assume further that the $T_i$ are in  fact unipotent, and   let $N_i=\log T_i=T_i-Id$ ($i=1,\ldots,k$) be the log of monodromy operators.   Again, since this can be obtained after finite base change, this will not affect extension questions.    We can now define the monodromy cone:
\begin{equation}\label{eqnmc1}
\sigma(\pi^\circ):=\RR^+\langle N_1\ldots, N_k\rangle\subseteq \mathfrak{sp}(H_\RR,Q),
\end{equation}
with $H=H^1(X_\ast,\QQ)$ and $Q$ the intersection pairing on $H$.

As before, we would like to identify this cone with a cone of quadratic forms on a \emph{fixed} vector space.
The point is that for  each $\lambda_1,\ldots,\lambda_k>0$, we obtain a limit mixed Hodge structure $H^1_{\lim}(\underline \lambda)=H^1_{\lim}(\sum_{i=1}^k\lambda_iN_i)$ on $H^1(X_\ast,\QQ)$, with monodromy weight filtration $W_\bullet(\underline \lambda)=W_\bullet(\sum_{i=1}^k\lambda_iN_i)$ given by \eqref{eqnNwf} with
$$N=\lambda_1N_1+\ldots+\lambda_kN_k.$$
 It is well known (see eg.~\cite{cattani}) that for $\lambda_1,\ldots,\lambda_k>0$,
\begin{equation}\label{eqnkerN}
\ker (\lambda_1N_1+\ldots+\lambda_kN_k)=\bigcap_{i=1}^k\ker (N_i).
\end{equation}
Thus $W_1(\underline \lambda)$, and hence $\Gr_2(\underline \lambda)$, is independent of the $\lambda_i>0$.   Consequently, for $$\lambda_1,\ldots,\lambda_k>0,$$ the $\lambda_1N_1+\ldots+\lambda_kN_k$ can all be viewed as quadratic forms on
$$\Gr_2:=H^1(X_\ast,\QQ)/\bigcap \ker (N_i).$$
In conclusion, the monodromy cone can be identified with a cone of symmetric matrices on the vector space $\Gr_2$
\begin{equation}\label{eqnmc3}
\sigma(\pi^\circ):=\RR^+\langle N_1\ldots, N_k\rangle\subseteq \Hom (\Sym^2\Gr_2,\RR).
\end{equation}

\begin{rem}
Using $N^2=0$ and the symplectic form $Q$, there is  a natural identification $W_0=W_1^\perp$, which then gives an identification
$$
W_0(N)=\Im(N)=\sum \Im N_i.
$$
\end{rem}

\subsubsection{Closures of monodromy cones}\label{sectclosure}
We now discuss the closure of the monodromy cone.
Clearly in regards to the description \eqref{eqnmc1}, we have
\begin{equation}\label{eqncmc1}
\overline{\sigma(\pi^\circ)}=\RR_{\ge 0}\langle N_1\ldots, N_k\rangle\subseteq \mathfrak{sp}(H_\RR,Q).
\end{equation}
However, in regards to \eqref{eqnmc3}, the description is not as obvious.
The issue is that in setting
$$
  \Gr_2=\Gr_2(\underline\lambda)=H^1(X_\ast,\QQ)/\bigcap \ker (N_i),
$$
the $N_i$ \emph{individually} are not naturally identified as quadratic forms in $\Hom (\Sym^2 \Gr_2,\RR)$; they are quadratic forms in $\Hom (\Sym^2 \Gr_2(W_\bullet(N_i)),\RR)$, respectively.   To remedy this, set $\overline N_i$ ($i=1,\ldots,k$) to be the composition:
\begin{footnotesize}
\begin{equation}\label{eqnCLM}
\begin{CD}
\Gr_2(W_\bullet(\underline \lambda))@>\rho_i >>  \Gr_2(W_\bullet(N_i)) @>N_i>> \Gr_0(W_\bullet(N_i)) @>\rho_i^\vee>> \Gr_0(W_\bullet(\underline \lambda).\\
@| @| @| @| \\
H^1/\bigcap \ker (N_i) @>\rho_i >> H^1/\ker (N_i) @>Q(\cdot, N_i(\cdot))>> (H^1/\ker (N_i))^\vee @>\rho_i ^\vee>>  (H^1/\bigcap \ker (N_i) )^\vee.
\end{CD}
\end{equation}
\end{footnotesize}
Then unwinding the definitions, we obtain
\begin{equation}\label{eqncmc3}
\overline{\sigma(\pi^\circ)}=\RR_{\ge 0}\langle \overline N_1\ldots, \overline  N_k\rangle\subseteq \Hom (\Sym^2 \Gr_2,\RR).
\end{equation}

\subsection{Monodromy cones in the geometric context} \label{secMGC} In the previous subsection we have discussed the abstract Hodge theoretic aspects associated to a degeneration. This discussion allows us to tie-in with the theory of toroidal compactifications discussed in Section \ref{secttoroidal}. Further, via the discussion of \S\ref{sectclosure}, it reduces the computation of monodromy cones to the case of $1$-parameter bases. Here we assume that this $1$-parameter VHS is arising from a $1$-parameter geometric family. In this situation, we would like to interpret the monodromy cones in terms of the geometry and combinatorics of the central fiber (i.e.~the limit geometric object).

Namely, we assume here that there is a smooth family $\pi^\circ:\mathfrak X^\circ\to S^\circ$ which has an extension  $\pi:\mathfrak X\to S$ to a projective morphism, with central fiber $X_0=\pi^{-1}(0)$ a simple normal crossing divisor in the family (this can be obtained after a finite base change by the semi-stable reduction theorem, and will not affect extension questions). Note this will also imply that the monodromy is unipotent. As is well known, the central fiber $X_0$ carries a canonical Mixed Hodge Structure (MHS).  Furthermore, the Clemens--Schmid exact sequence (eg.~\cite[p.109]{csseq}) relates the  limit mixed Hodge structure on $X_\ast$ to the MHS on  $X_0$.   To recall the sequence we will let $i:X_\ast \to X_0$ be the Clemens collapsing map (the composition of the inclusion $X_\ast\subseteq \mathfrak X$ with the contraction $\mathfrak X\to X_0$), and we will denote by $\PD$ any of the Poincar\'e duality isomorphisms.   In the weight $1$ case, the Clemens--Schmid exact sequence is
$$
\xymatrix{
0 \ar@{->}[r]& H^1(X_0) \ar@{->}[r]^{i^*}& H^1_{\lim}(X_\ast) \ar@{->}[r]^N&  H^1_{\lim}(X_\ast) \ar@{->}[r]^{\beta=\iota_*\PD}& H_{1}(X_0)  \ar@{->}[r]^\alpha & H^3(X_0) \ar@{->}[r]^{i^*} &\ldots
}
$$
Since we will not use the definition of $\alpha$, we refer the reader to \cite[p.108]{csseq}.  The maps $\alpha$, $i^*$, $N$, and $\beta$ are morphisms of mixed Hodge structures of types $(2,2)$, $(0,0)$, $(-1,-1)$, and $(-1,-1)$ respectively.   It follows that there  are  isomorphisms
$$
H^1(\Gamma,\QQ) =\Gr_0^{X_0}(H^1(X_0))\stackrel{i^*}{\longrightarrow} \Gr_0
$$
(the first identification being given by the Mayer--Vietoris spectral sequence for $X_0$) and
$$
\Gr_2\stackrel{\beta=i_*\PD}{\longrightarrow} \Gr_{0}^{X_0}(H_{1}(X_0))=\frac{(W_{-1}(H^{1}(X_0)))^\perp}{(W_{0}(H^{1}(X_0)))^\perp}=(\Gr_{0}^{X_0}(H^{1}(X_0)))^\vee.
$$
Thus composing, we may view the log of monodromy as a map
\begin{equation}\label{eqnCSLM}
(\Gr_{0}^{X_0}(H^{1}(X_0)))^\vee \stackrel{\beta^{-1}}{\longrightarrow} \Gr_2 \stackrel{N}{\longrightarrow} \Gr_0 \stackrel{(\iota^*)^{-1}}{\longrightarrow} H^1(\Gamma,\QQ).
\end{equation}
Again using the identification
$
\Gr_{0}^{X_0}(H^{1}(X_0))=H^1(\Gamma,\QQ)
$,
we may  identify the spaces  $H^1(\Gamma,\QQ)^\vee=H_1(\Gamma,\QQ)$ by the universal coefficients theorem, and
the composition \eqref{eqnCSLM} allows us to view the  log of monodromy as a map
\begin{equation}\label{eqnlm4}
N:H_1(\Gamma,\QQ)\to H^1(\Gamma,\QQ) \in \Hom (\Sym^2 H_1(\Gamma,\QQ),\QQ)
\end{equation}
Consequently, for the case of a family of stable curves $\pi:\mathfrak X \to S^n$, smooth over $(S^\circ)^k\times S^{n-k}$, the monodromy cone is given by
\begin{equation}\label{eqnmc4}
\sigma(\pi^\circ):=\RR^+\langle N_1\ldots, N_k\rangle\subseteq \Hom (\Sym^2H_1(\Gamma,\QQ),\QQ)_\RR
\end{equation}
where $\Gamma$ is the dual graph to the curve $X_0=\pi^{-1}(0)$, and $N_i$ is the log of monodromy   around  the hyperplane $\{x_i=0\}$.

In order to describe the closure of the monodromy cone, we introduce some further notation.
Let $0\ne 0_i\in S^n$ be a point in the hyperplane $\{x_i=0\}$, sufficiently close to $0$.  Let $X_{0_i}$ be the fiber over $0_i$, and let $\Gamma_i$ be the dual graph of $X_{0_i}$.   The issue with describing the closure of the monodromy cone is that $N_i$ is not a quadratic form on $H_1(\Gamma,\QQ)$, rather it is a quadratic form on $H_1(\Gamma_i,\QQ)$.    We resolve this using \eqref{eqnCLM}, and the identification $\beta:\Gr_2 \to H_1(\Gamma,\QQ)$.    Thus, there exist morphisms
\begin{equation}\label{eqnRHO}
\begin{CD}
H_1(\Gamma,\QQ)@>\rho_i >> H_1(\Gamma_i,\QQ)
\end{CD}
\end{equation}
so that setting $\overline N_i$ to be the composition
\begin{equation}\label{eqnMonComp}
\begin{CD}
H_1(\Gamma,\QQ)@>\rho_i >> H_1(\Gamma_i,\QQ)@> N_i>>  H^1(\Gamma_i,\QQ)@>\rho_i^\vee >> H^1(\Gamma,\QQ),
\end{CD}
\end{equation}
then the closure of the monodromy cone is given by
\begin{equation}\label{eqncmc4}
\overline{\sigma(\pi^\circ)}=\RR_{\ge 0}\langle \overline N_1\ldots, \overline  N_k\rangle\subseteq \Hom (\Sym^2 H_1(\Gamma,\QQ),\QQ)_\RR.
\end{equation}
Finally, the map $\rho_i$ in  \eqref{eqnRHO} can be described combinatorially.  For each $j=1,\ldots,k$ there is a natural map of chain complexes $C_\bullet(\Gamma,\ZZ)\to C_\bullet (\Gamma_j,\ZZ)$ (see \S\ref{secSimp}, where $\Gamma_j$ is denoted $\Gamma/S^c$), inducing surjective maps
\begin{equation}\label{eqnCM}
H_1(\Gamma,\QQ)\to H_1(\Gamma_j,\QQ).
\end{equation}
We claim that this  map agrees with the map $\rho_j$ above.    This follows from the definitions, and we sketch the argument here.  The key point is the identification
$$
\beta=i_*\PD:\Gr_2:=H^1_{\lim}/W_1\to \Gr_0^{X_0}(H_1):=H_1(X_0)/W_{-1}^{X_0}(H_1).
$$
We define $W_{-1}^{X_0}(H_1)=(W_0^{X_0}(H^1))^\perp$.      Dualizing the exact sequence (obtained from the Meyer--Vietoris spectral sequence)
$$
0\to W_0^{X_0}(H^1)\to H^1(X_0)\to H^1(\widehat X_0)\to 0
$$
we obtain that $(W_0^{X_0}(H^1))^\perp=H_1(\widehat X_0)$ using the universal coefficients theorem, where $\widehat X_0$ denotes the normalization of $X_0$.    In short, we have
$$
\beta=i_*\PD:\Gr_2\to H_1(X_0)/H_1(\widehat X_0).
$$
Thus there is a commutative diagram
\begin{footnotesize}
$$
\begin{CD}
 \Gr_2(\sum \lambda_jN_j)=H^1(X_\ast)/\bigcap \ker (N_j)@>\rho >> \Gr_2(N_j) =H^1(X_\ast)/\ker (N_j)@.\\
 @V\beta=i_*PD VV @V\beta=i_*PD VV\\
H_1(\Gamma)=H^1(\Gamma)^\vee=H_1(X_0)/H_1(\widehat X_0)@>(i_* \PD) (\PD^{-1} i_*^{-1})>> H_1(X_{0_j})/H_1(\widehat X_{0_j})=H^1(\Gamma_j)^\vee=H_1(\Gamma_j).
\end{CD}
$$
\end{footnotesize}
The bottom row  is easily seen to be the combinatorial map \eqref{eqnCM} above (note that $i_*^{-1}$ is only defined up to vanishing cycles, but nevertheless, $i_* i_*^{-1}$ is well defined as a map on the quotients).

\begin{rem}
One can also identify $\rho_j^\vee$ directly as well.  Following the definitions, one finds that it is given by
$$
H^1(\Gamma_j)=\Gr_0(H^1(X_0)) \stackrel{i^*}{\cong }\Gr_0(N_j)=\Im(N_j)\hookrightarrow \Im(\sum N_j)
$$
$$
=\Im(N)=\Gr_0(N)\
\stackrel{(i^*)^{-1}}{\cong } H^1(\Gamma).
$$
This can also be identified with the dual of the combinatorial map given above
\end{rem}

\subsection{Monodromy cones and extensions of period maps}\label{secextpermap}
Now that we have defined the pertinent terms, we can discuss the standard extension results for period maps to toroidal compactifications.  We begin by making one remark.
\begin{rem}\label{remFBC}
Let $B$ be a smooth variety.  It is a basic fact  that if a rational map from $B$ to any of the compactifications of the moduli of abelian varieties extends after a finite base change, then the rational map itself extends.  For this reason, we will be free in what follows to make finite base changes when considering extensions of period maps.
\end{rem}

\subsubsection{Extension via Hodge theory}
Fix a compatible collection of admissible cone decompositions $\Sigma$, and let $\bar A^\Sigma_g$ be the associated toroidal compactification.   Let $$f^\circ:(S^\circ)^k\times S^{n-k}\to A_g$$ be a locally liftable morphism (i.e.~one induced by a family of abelian varieties).    After a finite base change, we may assume that the monodromy operators $T_i$ around the boundary divisors $\{x_i=0\}$ are unipotent (see Remark \ref{remFBC}).   Then setting $N_i=\log T_i$, we have seen that for any $\lambda_1,\ldots,\lambda_k>0$, there is a fixed $Q$-isotropic subspace $W_0=\ker (\sum_i \lambda_iN_i)$.  The Borel extension theorem (\cite{borel}) implies that   $f^\circ$ extends to a morphism $f:S^n\to A^*_g$.  The isotropic subspace $W_0$ determines a boundary component $F_{W_0}$, and in turn, a boundary component $A^*_{g'}$. The point $f(0)$ is the point of $A^*_{g'}$ associated to the pure weight $1$ Hodge structure determined by the first graded piece of the limit mixed Hodge structure of any semi-stable reduction of the restriction of the (induced) family (of say abelian varieties) to a one-parameter base.    The following extension theorem is well known (see eg.~\cite[Thm.~7.29, Rem.~7.30]{nam}, \cite[Thm.~7.2]{AMRT}, \cite[Thm.~5.7, p.116]{FC90}):

\begin{fact}\label{fctext} The map $f^\circ$ extends to a morphism $S^n\to \bar A^\Sigma_g$ if and only if the monodromy cone (as defined in \eqref{eqnmc3}) is contained in a cone in $\Sigma$ (more specifically,  a cone in $\Sigma_{W_0}$).
\end{fact}

\subsubsection{Resolving period maps} \label{secRPM}
We now recall how one can resolve the period map to a toroidal compactification $\bar A_g^\Sigma$ of the moduli of abelian varieties.
 Let us assume that we have a semiabelian scheme $\mathscr X\to B$ where $B$ is smooth, and that there is a simple normal crossing divisor $\Delta\subseteq B$ so that the family is abelian over $B^\circ=B-\Delta$.    Fix a base point $0\in \Delta$.  The variety $(B,\Delta)$ is toroidal at $0$, corresponding to the semi-group ring $\mathbb C[\mathbb N^k]$, where $k$ is the number of components of $\Delta$ meeting at $0$ (technically we mean that locally, the miniversal space  is smooth over this toric variety).   Fix a basis $e_1,\ldots,e_{k}$ for $\mathbb R^k$. We say that $\mathbb R_{\ge 0}\langle e_i\rangle$ is the cone associated to the toric data at $0$.     Now let $\sigma(\mathscr X/B)$ be the  monodromy cone associated to the semiabelian family at $0$.     The period map defines a map of cones
$$
\mu:\mathbb R_{\ge 0}\langle e_i\rangle \to \sigma(\mathscr X/B)
$$
$$
e_i\mapsto \log T_{e_i}
$$
where $T_{e_i}$ is the monodromy around the hyperplane associated to $e_i$.    Now the admissible decomposition $\Sigma$ decomposes $\sigma(\mathscr X/B)$ into a fan $\mathscr F_0^\Sigma$ of cones.

  Fact \ref{fctext} states that the period map extends at $0$ if and only if $\mathscr F_0^\Sigma$ has just one cone of maximal dimension.  Otherwise, there is an induced fan $\mu^{-1}\mathscr F_0^\Sigma$ decomposing the cone $\mathbb R_{\ge 0}\langle e_i\rangle$.    Any fan $\mathscr F$ decomposing $\mathbb
  R_{\ge 0}\langle e_i\rangle$ determines a birational
  morphism to $\mathbb A^k_{\mathbb C}$.  Any fan $\mathscr F$ that refines $\mu^{-1}\mathscr F_0^\Sigma$ determines a birational modification of $B$ (in an \'etale
  neighborhood of $0$) that resolves the period map in a (\'etale) neighborhood of $0$.  In particular, the fan $\mu^{-1}\mathscr F_0^\Sigma$ determines the minimal, toric birational modification that will resolve the period map.   This minimal, toric birational modification is canonical and glues to give  a birational modification for a period map on a moduli space.

\subsection{Moduli stacks and abelian varieties} \label{secModStAbVar}

We may also view the toroidal compactifications as compactifications of the moduli of abelian varieties.
We begin by recalling  the basic structure
of toroidal compactifications from this perspective. Every toroidal compactification  $A_g^{\Sigma}$ has a canonical map
$$
\varphi^{\Sigma}: A_g^{\Sigma} \to A_g^{*} =A_g \sqcup A_{g-1} \sqcup \ldots \sqcup  A_0
$$
which defines a stratification  $\beta^{\Sigma}_i= (\varphi^{\Sigma})^{-1}(A_{g-i})$. The strata  $\beta^{\Sigma}_i$ are themselves
stratified $\beta^{\Sigma}_i=\sqcup \beta(\sigma)$ where $\sigma$ runs through all $GL(i,\ZZ)$ orbits of cones in the decomposition $\Sigma$ of
$\Sym^2(\ZZ^i)$ containing rank $i$ matrices. The strata $\beta(\sigma)$ themselves are of the form $\beta(\sigma)= {\mathcal T}(\sigma) /\mathcal G(\sigma) $ where
 ${\mathcal T}(\sigma)$ is a torus bundle over the $i$-fold fiber product of the universal family $p: {\mathcal X}_{g-i} \to {\mathcal A}_{g-i}$. Indeed
$\pi(\sigma)=p^{\times i}\circ q(\sigma) : {\mathcal T}(\sigma) \to {\mathcal X}_{g-i}^{\times i} \to {\mathcal A}_{g-i}$ where $q(\sigma) $ is  a torus bundle whose fibers have dimension $i(i+1)/2 - \dim(\sigma)$.
More precisely
${\mathcal T}(\sigma)= {\mathcal T}_i / {\mathcal T}_{\sigma}$ where the fibers of ${\mathcal T}_i$ and  ${\mathcal T}_{\sigma}$ are $\Sym^2(\ZZ^i) \otimes \CC^*$
and $(\langle \sigma \rangle \cap \Sym^2(\ZZ^i)) \otimes \CC^*$ respectively.
The group $G(\sigma)$ is the stabilizer of the cone $\sigma$ in $GL(i,\ZZ)$ and acts naturally on  ${\mathcal T}(\sigma)$.
The codimension of   ${\mathcal T}(\sigma)$ in $A_g^{\Sigma}$ is $\dim(\sigma)$.

\subsubsection{The Faltings--Chai stacks}
Faltings--Chai have given a stack theoretic interpretation of the toroidal compactifications.
For each toroidal compactification $\bar A_g^\Sigma$,   there is an  irreducible, normal, proper,  Deligne--Mumford $\CC$-stack $\overline{\calA}_g^\Sigma$ with coarse moduli space $\bar A_g^\Sigma$, and a semiabelian scheme $\mathcal X_g^\Sigma\to \bar{\mathcal A}_g^\Sigma$ extending the universal abelian variety $\mathcal X_g\to \mathcal A_g$ (\cite[Thm.~5.7 (5), p.117]{FC90}).
 While the stack  does not represent a moduli functor of semiabelian varieties, we do have the following. A  semiabelian scheme  $X\to S$ over the disk, such that the restriction  $X^\circ\to S^\circ$ to the punctured disk is abelian  (i.e.~a morphism $S\to \bar{\mathcal A}_g^\Sigma$ with $S^\circ \to \mathcal A_g$)
 is determined by a set of degeneration data
(see also \cite{alexeev02}, \cite{abh}), namely:
(D0) A principally polarized abelian variety $(A,M)$ (with $M$ an ample line bundle) inducing an isomorphism  $\lambda_M:A\to \widehat A$, where $\widehat A=\operatorname{Pic}^0(A)$.  (D1a) A semiabelian variety
$
0\to \mathbb T\to G\to A\to 0
$
with split torus part, determined by a homomorphism $c:\Lambda\to \widehat A$, where $\Lambda$ is the character lattice of $\mathbb T$.  (D1b) A second semiabelian variety
$
0\to \widehat {\mathbb T}\to \widehat G\to \widehat A\to 0
$
induced by a homomorphism $\hat c: \hat \Lambda\to A$, where $\hat \Lambda$ is the character lattice of $\widehat{\mathbb T}$.  (D2) An isomorphism of lattices $\phi:\widehat \Lambda\to \Lambda$ so that $c\circ \phi=\lambda_M\circ \hat c$.
(D3) A bihomomorphism $\tau:\hat \Lambda\times \Lambda \to (\hat c\times c)^*(\mathscr P^\circ )^{-1}
$,
where $\mathscr P^\circ$ is the rigidified Poincar\'e bundle with zero section removed.   (D4) A cubical morphism
$\psi:\widehat \Lambda \to \hat c^*(M^\circ)^{-1}$,  where $M^\circ$ is the principal bundle obtained from removing the zero section of $M$.
(D5) For each $\lambda\in \Lambda $ a section $\theta_\lambda \in \Gamma(A,M\otimes c(\lambda))$  satisfying some further compatibility conditions with the data above.
This data, more precisely the bihomomorphism $\tau$, defines a quadratic form $B:\Lambda\times  \Lambda\to \mathbb Q$, which in fact agrees
with the log of monodromy for the $1$-parameter family (at least up to $\GL$-conjugation and scaling, which is irrelevant from the extension of period maps perspective).
(D6) The Delaunay decomposition of $\Lambda_{\mathbb R}$ determined by $B$.

Given this data, we can describe the image of the central point under the map $f: S\to \bar { A}_g^\Sigma$ as follows.  The data (D0) determines the image of $0$ under the composition $S\to \bar {A}_g^\Sigma\to A_g^*$; in particular it determines the stratum $\beta_i$ described above, in which $f(0)$ lies.  The quadratic form $B$ lies in a unique cone $\sigma\in \Sigma$ of minimal dimension.  The point $f(0)$ then lies in the stratum $\beta(\sigma)\subseteq \beta_i$, described above.    The remaining degeneration data determines the specific point $f(0)$ within $\beta(\sigma)$.
More precisely,   using the description of ${\mathcal T}_i$ given in \cite[Prop.~7.2]{ght} the biholomorphism $\tau$ defines a point in ${\mathcal T}_i$ and  $f(0)$ is its image in  $\beta^{\operatorname{Vor}}(\sigma)= ({\mathcal T}_i/{\mathcal T}_{\sigma})/G(\sigma)$.

\begin{rem}\label{remstack}  Let  $\overline {\mathcal M}$ be  a smooth,  Deligne--Mumford $\CC$-stack containing an open substack $\mathcal M$  with normal crossing boundary divisor $\Delta=\overline{\mathcal M}\setminus \mathcal M$.  Let $\overline M$ and $M$ be the respective coarse moduli spaces.    Suppose there exists a morphism $\mathcal M\to \mathcal A_g$, inducing a morphism $M\to A_g$. Using the  Abramovich--Vistoli purity lemma: {\it
The  morphism $\mathcal M\to \mathcal A_g$ extends to a  morphism $\overline{\mathcal M}\to \bar {\mathcal A}^\Sigma_g$  if and only if the morphism $M\to A_g$ extends to a morphism $\overline M \to \bar A^\Sigma_g$.}
\end{rem}

\begin{rem}
The singularities of the stack $\bar{\mathcal A}_g^\Sigma$ can be read off from the cones in $\Sigma$. Basic cones give rise to smooth points of the stack.
Simplicial but non-basic cones correspond to quotient
singularities by finite abelian groups. More precisely
these groups  are identified with  the quotient of the
lattice $(\operatorname{Sym}^2(\Lambda))^\vee \cap \langle \sigma \rangle$ by the sub-lattice generated by the integral generators of $\sigma$.   Non-simplicial cones give rise to more general (toric) singularities of the
moduli stack. Singularities of the  varieties $\bar A_g^\Sigma$ can also occur if the cones are basic, in fact they already occur on ${A}_g$ itself.
The singularities depend on the finite stabilizer of a point in the toric construction.
The codimension of the singular locus of the stack $\bar{\mathcal A}_g^{P}$  is $10$, whereas it is $3$ in the case of
$\bar{\mathcal A}_g^V$ \cite{DHS}.
\end{rem}

\subsubsection{Alexeev's stack of stable semiabelic pairs}   \label{secSSAP}
Alexeev has constructed a moduli space  $\bar {\mathcal A}_g^A$ of complex stable semiabelic pairs  that contains $\mathcal A_g$  as
an open substack \cite{alexeev02}.  The stack  $\bar {\mathcal A}_g^A$  is a proper,
algebraic (Artin) $\CC$-stack with finite diagonal \cite[Thm 5.10.1]{alexeev02}. Moreover, the stack
admits a coarse moduli space, with a component that has normalization
isomorphic to the second Voronoi compactification $\AV$ \cite[Thm. 5.11.6, p.701]{alexeev02} (see also Olsson \cite{olsson}). The main point is that the degeneration
data described above define locally relatively complete models which admit an action of the universal semiabelian variety. Gluing these to obtain a
universal family over a compactifcation of ${\mathcal A}_g$ one is
naturally led  to the second Voronoi decomposition.

We also recall how the degeneration data above determine a stable semi-abelic pair.    The Delaunay decomposition defines a  fan on $\mathbb R\oplus \Lambda_{\mathbb R}$ with cones determined by the Delaunay decomposition shifted by $(1,0)\in \mathbb N\oplus \Lambda$.
We use this to define an $\mathscr O_A$-algebra $\mathscr R$.  As a module, $\mathscr R$ is freely generated by $M_\chi:=M^{\otimes d}\otimes c(\lambda)$ for each $\chi=(d,\lambda)\in \mathbb N\oplus \Lambda$.  We define multiplication by the natural identification $M_{\chi_1}\otimes M_{\chi_2}=M_{\chi_1+\chi_2}$ when $\chi_1,\chi_2\in \delta$ lie in a common  cone in the fan over the Delaunay decomposition.   When $\chi_1,\chi_2$ do not lie in a common cone, sections multiply to $0$.
The morphisms $\tau$ and $\phi$ define a natural action of $\hat\Lambda$ on $\mathscr R$, and the action is properly discontinuous in the Zariski topology on the relative Proj, so that $X=(\operatorname{Proj}_A\mathscr R)/\hat\Lambda$ is a well-defined, polarized scheme.    This is the stable semiabelic pair.   Note that the fiber of $X$ over $A$ is a (possibly reducible)  projective toric variety, obtained by gluing the toric varieties determined by the tiling
in the Delaunay decomposition.


\section{Prym varieties and admissible covers}\label{sectPrym}
It is well known that  $\mathcal M_g$ has a normal crossing compactification $\overline{\mathcal M}_g$ obtained by allowing stable curves, whose limiting Hodge theoretic behavior is controlled by a combinatorial object, the dual graph $\Gamma$. Prym varieties are abelian varieties  obtained from connected \'etale double covers of curves. A normal crossing compactification for the moduli of connected \'etale double covers $\mathcal R_g$ (compatible with $\overline {\mathcal M}_g$) was constructed  by Beauville by considering admissible double covers of stable curves. The associated combinatorial object governing the limiting Hodge theoretic behavior is a dual graph with involution. We briefly review this below.

\subsection{Admissible covers}
Let $C$ be a stable curve of  genus $g+1\ge 2$.   Recall that an \emph{admissible double cover of $C$} is a finite, surjective morphism $\pi:\widetilde C \to C$ of stable curves such that:
\begin{enumerate}
\item  The arithmetic genus of $\widetilde C$ is equal to $\tilde g=2g+1$.
\item For each irreducible component $C'$ of $C$, the restriction $\pi: \pi^{-1}(C')\to C'$ has degree two (but $\pi^{-1}(C')$ may be  reducible or disconnected).
\item If $\iota: \widetilde C\to \widetilde C$ is the sheet interchange involution associated to $\pi$, the fixed points of $\iota$ are a subset of the nodes of $\widetilde C$, and at a fixed node the local branches of $\widetilde C$ are not exchanged.
\end{enumerate}

There exists a smooth, irreducible, proper Deligne--Mumford $\CC$-stack $\overline{\mathcal R}_{g+1}$ parameterizing admissible double covers of stable curves of genus $g+1\ge 2$ \cite{b}. We denote by $\mathcal R_{g+1}$ the open sub-stack of connected, \'etale double covers of smooth curves.  The forgetful functor $\overline{\mathcal R}_{g+1}\to \overline {\mathcal M}_{g+1}$ to the moduli of stable curves
defines a degree  $2^{2(g+1)}-1$ cover,  ramified along an irreducible boundary divisor $\delta_0^{\operatorname{ram}}\subset\overline{\calR}_{g+1}$ (see \S \ref{secBD}).
 The full boundary $\delta_{\overline{\mathcal R}_{g+1}}$ of $\overline {\mathcal R}_{g+1}$ is a simple normal crossing divisor with the property that  \'etale locally at a point $\pi:\widetilde C\to C$, its irreducible components
 correspond to the nodes of $C$.  We discuss the irreducible components of the boundary divisor $\delta_{\overline{\calR}_{g+1}}$  in \S \ref{secBD} below.    We denote by $\overline R_{g+1}$ and $R_{g+1}$ the coarse moduli spaces of the respective stacks.

\subsection{Involutions of graphs}
Here we fix our conventions on graphs. A graph $\Gamma$ is a set of {\em vertices} $V=V(\Gamma)$ and a set of {\em oriented edges} $\overrightarrow E=\overrightarrow E(\Gamma)$ together with  maps $(\overrightarrow E  \xymatrix{ \ar @{->}^s @< 2pt> [r] \ar@{->}_t @<-2pt> [r] &  } V, \overrightarrow E \stackrel{\tau}{\to} \overrightarrow E)$, where $\tau$ is a fixed-point free involution, and $s$ and $t$ are maps satisfying $s(\overrightarrow e)=t(\tau(\overrightarrow e))$ for all $\overrightarrow e \in \overrightarrow E$.  The maps $s$ and $t$ are called the \emph{source} and \emph{target} maps respectively.

We define the set of \emph{(unoriented) edges} to be  $E(\Gamma)=E:=\overrightarrow E/\tau$.   Given an oriented edge $\overrightarrow e \in \overrightarrow E $ we will denote by $\underline {\overrightarrow e}$ the class of $\overrightarrow e$ in  $E$.   An \emph{orientation of an edge} $e\in E$ is a representative for $e$ in $\overrightarrow E$; we use the notation $\overrightarrow e$ and $\overleftarrow e$ for the two possible orientations of $e$. An \emph{orientation of a graph $\Gamma$} is a section $\phi:E\to \overrightarrow E$ of the quotient map.  An \emph{oriented graph}  consists of a pair $(\Gamma,\phi)$ where $\Gamma$ is a graph and $\phi$ is an orientation.

A \emph{morphism of graphs} $\Gamma_1\to \Gamma_2$ consists of a pair of maps
$$
  V(\Gamma_1)\to  V(\Gamma_2) \ \ \text{ and } \ \  \overrightarrow E(\Gamma_1)\to \overrightarrow E(\Gamma_2)
$$
so that all of the associated diagrams commute.
An \emph{involution $\iota$ of a graph} is an endomorphism  of the graph such that $\iota^2$ is the identity.   We can define morphisms of oriented graphs as well; an involution of an oriented graph is defined in the obvious way.

Associated to a graph $\Gamma$ is a chain complex
$$
 (C_\bullet(\Gamma,\ZZ),\partial_\bullet),
$$
where $C_0(\Gamma,\ZZ)$ is the free $\ZZ$-module with basis $V(\Gamma)$, and $C_1(\Gamma,\ZZ)$ is the quotient of the free $\ZZ$-module with basis $\overrightarrow E(\Gamma)$ by the relation $\overleftarrow e=-\overrightarrow e$ for every $e\in E(\Gamma)$.
We denote by $[\overrightarrow e]$ the class of $\overrightarrow e$ in $C_1(\Gamma,\ZZ)$.
Note that while  $\underline {\overrightarrow e}=\underline {\overleftarrow e}$ in $E$ (they correspond to the same unoriented edge), in $C_1(\Gamma,\ZZ)$  we have $[\overrightarrow e]=-[\overleftarrow e]$.   An orientation $\phi$ determines a basis $\{[\phi(e)]\}_{e\in E}$ for $C_1(\Gamma,\ZZ)$, which identifies $C_1(\Gamma,\ZZ)$ with the usual chain group of $1$-chains for the associated simplicial complex.
The  boundary map is defined by:
\begin{equation*}
  \partial: C_1(\Gamma,\ZZ) \to C_0(\Gamma,\ZZ)
\end{equation*}
\begin{equation*}
  [\overrightarrow e]  \mapsto t(\overrightarrow e)-s(\overrightarrow e).
\end{equation*}
We will denote by  $H_\bullet (\Gamma,\ZZ)$ the groups obtained from the homology of  $C_\bullet (\Gamma,\ZZ)$; the group $H_\bullet (\Gamma,\ZZ)$ is isomorphic to the homology of the underlying topological space of $\Gamma$.

From the definitions one can check immediately that an involution $\iota$  of a graph $\widetilde \Gamma$ induces an involution $\iota$ of the chain complex $C_\bullet(\widetilde \Gamma,\ZZ)$.   This in turn induces an involution $\iota$ on $H_\bullet (\widetilde \Gamma,\ZZ)$; we denote by $H_1(\widetilde \Gamma, \ZZ)^\pm$ the eigenspaces of the action of $\iota$.  We define
$$
  H_1(\widetilde \Gamma,\ZZ)^{[+]}=H_1(\widetilde \Gamma,\ZZ)/H_1(\widetilde \Gamma,\ZZ)^{-} \ \ \text { and } \ \ H_1(\widetilde \Gamma,\ZZ)^{[-]}=H_1(\widetilde \Gamma,\ZZ)/H_1(\widetilde \Gamma,\ZZ)^+.
$$
Note that
$$
  H_1(\widetilde \Gamma,\ZZ)^{[-]}\cong \Im\left( \frac{1}{2}(\Id -\iota) \right) \subseteq \frac{1}{2}H_1(\widetilde \Gamma,\ZZ)
  $$
  $$
  H_1(\widetilde \Gamma,\ZZ)^{[+]}\cong \Im\left( \frac{1}{2}(\Id +\iota) \right) \subseteq \frac{1}{2}H_1(\widetilde \Gamma,\ZZ).
$$
As usual, we construct a cochain complex $C^\bullet(\widetilde \Gamma,\ZZ)=\Hom (C_\bullet(\widetilde \Gamma,\ZZ),\ZZ)$, and define the cohomology groups $H^\bullet(\widetilde \Gamma,\ZZ)$ to be the homology of this complex.  We have $H^i(\widetilde \Gamma,\ZZ)=H_i(\widetilde \Gamma,\ZZ)^\vee$ ($i=0,1$).  However, in contrast we have
$$
  H^i(\widetilde \Gamma,\ZZ)^\pm=\left(H_i(\widetilde \Gamma,\ZZ)^{[\pm]}\right)^\vee \ \ (i=0,1).
$$
\begin{rem}\label{remcoedge}
Here we make an observation that will be useful for later computations.  To simplify the discussion, fix an orientation of the graph $\Gamma$.  Then, by definition, $C^1(\Gamma,\ZZ)=C_1(\Gamma,\ZZ)^\vee=(\bigoplus_{e\in E}\ZZ\cdot e)^\vee=\bigoplus_{e\in E}\ZZ\cdot e^\vee$.  We call the elements $e^\vee$ co-edges.    There is by definition a surjection $C^1(\Gamma,\ZZ)\twoheadrightarrow H^1(\Gamma,\ZZ)$.   Denote temporarily by $\widehat {e^\vee}$ the image of a co-edge $e^\vee$ in $H^1(\Gamma,\ZZ)$.    Now note that if $\widehat {e_1^\vee}=\widehat {e_2^\vee}$, then for any $z\in H_1(\Gamma,\ZZ)\subseteq C_1(\Gamma,\ZZ)$, we have $e_1^\vee(z)=e_2^\vee(z)$.
\end{rem}

An \emph{admissible involution} of a graph $\widetilde \Gamma$ is  an involution $\iota$  such that  for all $\overrightarrow e\in \overrightarrow E$,   $\iota(\overrightarrow e)\ne \overleftarrow e$.
In other words, an involution is admissible if whenever an unoriented  edge of the graph is fixed by the involution, the vertices at the endpoints of the edge are not interchanged by the involution (unless the edge is a loop, in which case the condition requires the associated oriented edges not to be interchanged).

If $\pi:\widetilde C\to C$ is an admissible cover, then the associated involution
$\iota$ of $\widetilde C$ induces a well defined involution $\iota$ of the vertices $V(\Gamma_{\widetilde C})$ and of the unoriented edges $E(\Gamma_{\widetilde C})$.  There is also a well defined induced involution on the set of oriented edges.   We will call this   the  \emph{induced admissible involution of the dual graph of $\widetilde C$}.


\subsection{Prym varieties}
Let $C$ be a stable curve of genus $g\ge 2$.    The Jacobian $JC$ is defined to be the connected component of the identity in $\Pic(C)$.  The Jacobian is a semiabelian variety of dimension $g$, which can be described explicitly as follows.  Let $\nu:N\to C$ be the normalization and let $\Gamma=\Gamma_C$ be the dual graph of $C$.   Then there is an exact sequence
\begin{equation}\label{eqnExtJac}
\begin{CD}
0@>>> H^1(\Gamma,\ZZ)\otimes_{\ZZ}\CC^*@>>>  JC @>\nu^*>>  JN @>>> 0.
\end{CD}
\end{equation}
The extension determines a class in
$$
\Ext^1(JN,H^1(\Gamma,\ZZ)\otimes_{\ZZ}\CC^*)=\Hom (H_1(\Gamma,\ZZ),\widehat {JN}),
$$
where $\widehat {JN}=\Pic^0(JN)$ is the dual abelian variety.   We refer the reader to \cite[p.76]{abh} for an explicit description of the extension class (see also \S \ref{secsubFibers}).      For later reference we note that $JC$ is an extension of the torus $\TT_C:=H^1(\Gamma,\ZZ)\otimes_{\ZZ}\CC^*$, which has character lattice
\begin{equation}\label{eqnJClattice}
\Lambda_C:=\Hom (\TT_C,\CC^*)=H_1(\Gamma,\ZZ).
\end{equation}

Now let $\pi:\widetilde C\to C$ be an admissible double cover of a stable curve $C$ of genus $g+1\ge 2$.
We define the Prym variety
$$
P:=P(\widetilde C/C)=\ker \left(\Nm:J\widetilde C\to JC\right)_0
$$
to be the connected component of the identity in the kernel of the norm map.
The Prym variety is a semiabelian variety of dimension $g$, which can be explicitly described as follows.   Let $\tilde \nu:\widetilde N\to \widetilde C$ and  $\nu:N\to C$ be the normalizations,  and let $\tilde \Gamma=\Gamma_{\widetilde C}$ and  $\Gamma=\Gamma_C$ be the dual graphs of $\widetilde C$ and  $C$ respectively.   Then there is an exact sequence
\begin{equation}\label{eqnExtPrym}
\begin{CD}
0@>>> H^1(\widetilde \Gamma,\ZZ)^-\otimes_{\ZZ}\CC^*@>>>  P @>>>  A @>>> 0,
\end{CD}
\end{equation}
where $A$ is a finite cover of $P_N:=P(\widetilde N/N)=\ker \left(\Nm:J\widetilde N\to JN\right)_0$, the Prym variety of the normalization.
The extension determines a class in
$$
\Ext^1(A,H^1(\widetilde \Gamma,\ZZ)^-\otimes_{\ZZ}\CC^*)=\Hom (H_1(\Gamma,\ZZ)^{[-]},\widehat A).
$$
We direct the reader to \cite[\S 1, Prop.~1.5]{abh} for more details on the relationship between $A$ and $P_N$, as well as for an explicit description of the extension class (see also \S \ref{secsubFibers}).    For later reference we note that $P$ is an extension of the torus $\TT_P:=H^1(\widetilde \Gamma,\ZZ)^-\otimes_{\ZZ}\CC^*$, which has character lattice
\begin{equation}\label{eqnPrymlattice}
\Lambda_{\widetilde C/C}:=\Hom (\TT_P,\CC^*)=H_1(\Gamma,\ZZ)^{[-]}.
\end{equation}


\subsection{The  boundary divisors in  $\coR$}\label{secBD}
In some arguments in what follows we will want to enumerate certain types of admissible covers. We thus review the enumeration of the irreducible boundary components of $\coR$ following \cite{farkassurvey} and the  references therein (see especially \cite{bernstein}, \cite{FL10}), noting also the corresponding descriptions in terms of vanishing cycles (see also the preprint version of \cite{fs}).

Recall that for a smooth curve $C$ of genus $g$, there are natural  identifications of the following sets:
$$\text{  \{Conn.~\'et.~dbl.~cov.}\ \pi: \widetilde C\to C\}=H^1(C,\ZZ/2\ZZ)-\{0\}$$
$$ =\{ \eta\in \Pic^0(C): \eta \ncong \mathscr O_C, \ \eta^{\otimes 2}\cong \mathscr O_C\}.$$
For a stable curve $C_0$, with a unique node, we will denote by $C_\ast$ a nearby smooth curve, and $\gamma \in H^1(C_\ast,\ZZ)$ the associated vanishing co-cycle. The irreducible boundary components of $\coR$ are as follows.
The preimage in $\coR$ of the locus of irreducible stable curves $\delta_0\subset\overline{\calM}_{g+1}$ has three irreducible components $\delta_0'$,
$\delta_0''$, and $\delta_0^{ram}$ defined as follows:
$$
\begin{array}{lcl}
\delta_0'&=&\{(C_0,a): C_0\in \delta_0^\circ, \ a\in H^1(C_\ast,\ZZ/2\ZZ)-0,\  a \cdot \gamma =0, \text { but } a \notin\langle \gamma \rangle \}^-\\
\delta_0''&=& \{(C_0,a): C_0\in \delta_0^\circ,\  a\in H^1(C_\ast,\ZZ/2\ZZ)-0,\ a \cdot \gamma =0, \text { and } a \in\langle \gamma \rangle \}^-\\
\delta_0^{\operatorname{ram}} &=& \{(C_0,a): C_0\in \delta_0^\circ, \ a\in H^1(C_\ast,\ZZ/2\ZZ)-0, \  a \cdot \gamma \ne 0 \}^-
\end{array}
$$
In the above, and in what follows, we will denote by $\delta_i^\circ\subseteq \delta_i$ the locus of curves with a single node.
The bar after the sets above denotes taking the closure.

\begin{figure}[htb]
\begin{equation*}
\xymatrix{
\widetilde \Gamma \ \ \ \ & *{\bullet} \ar@{-}@(lu,ld)|-{\SelectTips{cm}{}\object@{>}}_{\tilde e^+}  \ar@{-}@(ru,rd)|-{\SelectTips{cm}{}\object@{>}}^{\tilde e^-} & & \Gamma \ \ \ \ &  *{\bullet} \ar@{-}@(lu,ld)|-{\SelectTips{cm}{}\object@{>}}_{e}  &
}
\end{equation*}

\caption{Dual graph of  a generic admissible cover in $\delta_0'$.}\label{Fig:d0'}
\end{figure}

\begin{figure}[htb]
\begin{equation*}
\xymatrix{
\widetilde \Gamma & *{\bullet} \ar @{-}@/_1pc/[rr]|-{\SelectTips{cm}{}\object@{<}}_{\tilde e^-}
 \ar@{-} @/^1pc/[rr]|-{\SelectTips{cm}{}\object@{>}}^{\tilde e^+}  ^<{\tilde v^-}^>{\tilde v^+}  &&*{\bullet} & & \Gamma  \ \ \ \ &  *{\bullet} \ar@{-}@(lu,ld)|-{\SelectTips{cm}{}\object@{>}}_{e}_<{v}  &
}
\end{equation*}

\caption{Dual graph of  a generic  admissible cover in $\delta_0''$.}\label{Fig:d0''}
\end{figure}

\begin{figure}[htb]
\begin{equation*}
\xymatrix{
\widetilde \Gamma \ \ \ \ & *{\bullet} \ar@{-}@(lu,ld)|-{\SelectTips{cm}{}\object@{>}}_{\tilde e} & & \Gamma \ \ \ \ &  *{\bullet} \ar@{-}@(lu,ld)|-{\SelectTips{cm}{}\object@{>}}_{e}  &
}
\end{equation*}

\caption{  Dual graph of  a generic admissible cover in $\delta_0^{\operatorname{ram}}$.}\label{Fig:d0ram}
\end{figure}

The preimage in $\coR$ of the boundary divisor $\delta_i\subset\overline{\calM}_{g+1}$ has three irreducible components (only
two for $i=(g+1)/2$, in which case the first two are the same) described as follows

$$
\begin{array}{lcl}
\delta_i&=&\{(C,\eta): C=C_i\cup C_{g+1-i}\in \delta_i^\circ,\  \eta |_{C_{g+1-i}}\cong \mathscr O_{C_{g+1-i}} \}^-   \\
\delta_{g+1-i}&=&\{(C,\eta): C=C_i\cup C_{g+1-i}\in \delta_i^\circ,\  \eta |_{C_i}\cong \mathscr O_{C_i} \}^-  \\
\delta_{i;g+1-i} &=& \{(C,\eta): C=C_i\cup C_{g+1-i}\in \delta_i^\circ,\  \eta |_{C_i}\ncong \mathscr O_{C_i}, \ \eta |_{C_{g+1-i}}\ncong \mathscr O_{C_{g+1-i}}  \}^-
\end{array}
$$

\begin{figure}[htb]
\begin{equation*}
\xymatrix@R=.3cm{
& *{\bullet} \ar @{-}@/_0pc/[rd]|-{\SelectTips{cm}{}\object@{>}}^{\tilde e^+} _<{\tilde v^+}&&&&&&   & &\\
\widetilde \Gamma \ \ \ \ && *{\bullet}&&\Gamma &*{\bullet} \ar @{-}@/_0pc/[r]|-{\SelectTips{cm}{}\object@{>}}^{e} _<{ } &*{\bullet}\\
 &*{\bullet}\ar@{-} @/^0pc/[ru]|-{\SelectTips{cm}{}\object@{>}}_{\tilde e^-} ^<{\tilde v^-} &&& &&& \\
}
\end{equation*}

\caption{  Dual graph of  a generic  admissible cover in $\delta_i$ (or $\delta_{g+1-i}$).}\label{Fig:di}
\end{figure}

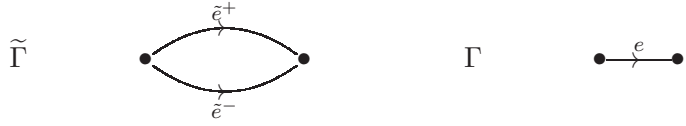
\begin{figure}[htb]
\begin{equation*}
\xymatrix{
\widetilde \Gamma \ \ \ \ & *{\bullet} \ar @{-}@/_1pc/[rr]|-{\SelectTips{cm}{}\object@{>}}_{\tilde e^-}
 \ar@{-} @/^1pc/[rr]|-{\SelectTips{cm}{}\object@{>}}^{\tilde e^+}   &&*{\bullet} & & \Gamma \ \ \ \ & *{\bullet}
 \ar@{-} @/^0pc/[r]|-{\SelectTips{cm}{}\object@{>}}^{e}  & *{\bullet}
}
\end{equation*}

\caption{  Dual graph of  a generic admissible cover in $\delta_{i,g+1-i}$}\label{Fig:digi}
\end{figure}


\section{Monodromy cones for Prym varieties}\label{sectMon}
In this section, we compute the monodromy cones (in the terminology of \S\ref{sectHT}) for a boundary point $(\widetilde C,C)$ of $\oR$ in terms of the combinatorics of the dual graph of $(\widetilde C, C)$ (discussed in \S\ref{sectPrym}). As a warm-up, we first review the classical case of Jacobians. The case of Pryms then naturally follows. While essentially equivalent computations can be found in \cite{fs} and \cite{abh}, our presentation for Prym varieties seems to be somewhat new.

\subsection{Monodromy cones for stable curves}
Let $C$ be  a stable curve of genus $g\ge 2$.  Let $\mathscr C\to B$ be  a miniversal deformation  of $C$, with discriminant $\Delta$, and set $B^\circ=B_C^\circ=B-\Delta$.  Denote by $0\in B$ the point corresponding to $C$.
Let  $\Gamma$ be the dual graph  of $C$.  We have  $\dim B_C=3g-3$ and $\Delta$ is a collection of simple normal crossing hyperplanes, indexed by the nodes of $C$ (which in turn are indexed by the edges of $\Gamma$).
Recall that the Jacobian of $C$ is a semiabelian variety obtained as an extension of the torus $\TT_C= H^1(\Gamma,\ZZ) \otimes_{\ZZ} \CC^*$, which has character lattice
$$
\Lambda_C=H_1(\Gamma,\ZZ).
$$

Since $B^\circ$ is locally (near $0$) a polycylinder and the associated mondoromies are unipotent, we can apply the considerations of \S\ref{smoncone} and \S\ref{secMGC} and define the monodromy cone
$$\sigma(C)\subseteq \overline C^{\QQ}(\Lambda_C)$$
to be the cone spanned by the log of mondromies around the branches of $\Delta$ (see \eqref{eqnmc1}). More precisely, we recall that for each irreducible component $\Delta_e\subseteq \Delta$, corresponding to an edge $e\in \Gamma$, there is an associated quadratic form obtained from  the log of monodromy around $\Delta_e$ (cf. \eqref{eqnlm4}).  The closure  $\overline  \sigma(C)$ is the cone generated by these quadratic forms (see \eqref{eqncmc4}):
$$
\overline{\sigma}(C)=\RR_{\ge 0}\langle \overline N_e\rangle_{e\in \Gamma}\subseteq \Hom (\Sym^2 H_1(\Gamma,\QQ),\QQ)_\RR.
$$

We now state the following well-known description of the monodromy cone, and provide a sketch of the proof.

\begin{pro}\label{projmoncone}
Suppose that $C$ is a stable curve of genus $g\ge 2$. Let $e$ be an edge of the dual graph $\Gamma$ of $C$.   Then $(e^\vee)^{2}$ is  the quadratic form obtained as the log of  monodromy around the corresponding  component  $\Delta_e$ of the discriminant.
Consequently,
the closure of the monodromy cone for $C$ is
$$
\overline{\sigma}(C)=\RR_{\ge 0}\langle  (e^\vee)^2 \rangle_{e\in E(\Gamma)}.
$$
\end{pro}

\begin{proof}
We start by describing the monodromy operators as in \eqref{eqnlm4}.
Let us first consider  the special  case where  $C$ has a single node.
There are two possibilities:
\begin{enumerate}
\item[(0)]  $e^\vee=0\in H^1(\Gamma,\QQ)$  (equivalently, $C\in \delta_i$, $i>0$).
\item [(1)]  $e^\vee\ne 0\in H^1(\Gamma,\QQ)$ (equivalently, $C\in \delta_0$).
\end{enumerate}
In case (0), $H_1(\Gamma,\QQ)=0$, so $JC$ is an abelian variety, the monodromy is trivial, and there is nothing to show.  In  case (1), $H_1(\Gamma,\QQ)= \QQ\langle e\rangle$, where $e$ is the unique edge of $\Gamma$.  As before, we view the log of monodromy as a map
$$
\begin{CD}
H_1(\Gamma,\mathbb Q)@> N_e>> H^1(\Gamma,\mathbb Q)=\left(H_1(\Gamma,\mathbb Q)\right)^\vee.
\end{CD}
$$
Since the  monodromy operator $T_e$ is given by the well-known Picard--Lefschetz transformation, it follows that $ N_e(e)=e^\vee$.
The associated quadratic  form is then  $(e^\vee)^2$.

The general case  follows by the arguments in \S \ref{secttoroidal},  \S \ref{sectHT}   (esp.~\S\ref{secMGC} and \eqref{eqnMonComp}), which establish that $\overline  N_e$ is given by the composition:
$$
\begin{CD}
H_1(\Gamma,\QQ)@>>> H_1(\Gamma_e,\QQ) @> N_e >> H^1(\Gamma_e,\QQ)@>>> H^1(\Gamma,\QQ)
\end{CD}
$$
where $\Gamma_e$ is the dual graph of the curve obtained from $C$ by smoothing all of the nodes except for the one corresponding to $e$.
\end{proof}

\begin{rem}
Let $\psi:S\to B$ be a morphism   from  the unit disc to a mini-versal deformation space of the curve $C$,  induced from a $1$-parameter deformation of $C$.   For each edge $e\in \Gamma$, the dual graph of $C$, let $z_e$ be a local parameter defining the hyperplane $H_e\subseteq B$ parameterizing curves with all nodes smoothed except the one corresponding to $e$.  Then the log of monodromy for the family is given by $\sum_{e\in \Gamma}\operatorname{ord}\psi^*(z_e) (e^\vee)^2$.
\end{rem}

\subsection{Monodromy cones for admissible double covers}
Let $\pi:\widetilde C\to C$ be an admissible cover of  a stable curve $C$ of arithmetic genus $g+1\ge 2$.  Let $B$ be the base of a miniversal deformation of the cover, with discriminant $\Delta$, and set $B^\circ=B-\Delta$.  Denote by $0\in B$ the point corresponding to $\pi:\widetilde C\to C$.
Let  $\widetilde \Gamma$ (resp.~$\Gamma$) be the dual graph  of $\widetilde C$ (resp.~$C$).  We have that $\dim B=3g$ and $\Delta$ is a collection of simple normal crossing hyperplanes, indexed by the nodes of $C$ (which are in turn indexed by the edges of $\Gamma$).
Recall that the Prym variety of $\pi:\widetilde C\to C$  is a semiabelian variety obtained as an extension of the torus $\TT_P= H^1(\widetilde \Gamma,\ZZ)^- \otimes_{\ZZ} \CC^*$, which has character lattice
$$
\Lambda_{\widetilde C/C}=H_1(\Gamma,\ZZ)^{[-]}.
$$
We wish to describe the associated (log of)  monodromy cone $$\sigma(\widetilde C/C)\subseteq \overline C^{\QQ}(\Lambda_{\widetilde C/C}).$$   Recall that for each irreducible component $\Delta_e\subseteq \Delta$, corresponding to an edge $e\in \Gamma$, there is an associated quadratic form obtained as the log of monodromy around $\Delta_e$.  The closure  $\overline \sigma(\widetilde C/C)$ is the cone generated by these quadratic forms.

From the perspective of Hodge theory,
 for each $1$-parameter family $f:S\to B$, with $f(0)=0\in B$, there is an associated $1$-parameter variation of Hodge structures determined by the family of Prym varieties.  We will denote by $N^-$ the log of monodromy for this VHS.  At the same time, there is a $1$-parameter VHS determined by the  Jacobians of the covering curves, with log of monodromy $\widetilde N$.  One can identify $\operatorname{Gr}_0(N^-)=(\operatorname{Gr}_0\widetilde N)^-=H^1(\widetilde \Gamma,\mathbb Q)^-$.
Then as in \eqref{eqnlm3}, we can view the log of monodromy as a map
\begin{equation}\label{eqnPrymLM}
\begin{CD}
H_1(\widetilde \Gamma,\mathbb Q)^-@>N^->>  H^1(\widetilde \Gamma,\mathbb Q)^-,
\end{CD}
\end{equation}
or dually as a positive definite quadratic form on $H_1(\widetilde \Gamma,\mathbb Q)^-$.
 Now to describe the closure of the monodromy cone, for each edge $e\in \Gamma$, one obtains a log of monodromy operator
\begin{equation}\label{eqnPrecall2}
\begin{CD}
H_1(\widetilde \Gamma,\mathbb Q)^-@>\overline N_e^->>  H^1(\widetilde \Gamma,\mathbb Q)^-,
\end{CD}
\end{equation}
so that the associated bilinear form is symmetric and positive semi-definite.    The closure of the monodromy cone is then given as:
$$
\overline{\sigma}(C)=\RR_{\ge 0}\langle \overline N_e^-\rangle_{e\in \Gamma}\subseteq \Hom (\Sym^2 H_1(\widetilde \Gamma,\QQ)^-,\QQ)_\RR.
$$

\begin{pro}\label{propmoncone}
Suppose that $\pi:\widetilde C\to C$ is an admissible cover of a stable curve of genus $g+1\ge 2$. Let $e$ be an edge of the dual graph $\Gamma$ of $C$ and let  $\tilde e$ be an edge of $\widetilde \Gamma$ lying above $e$.
Then (up to a scalar factor) $(\tilde e^\vee-\iota\tilde e^\vee)^{2}$ is  the quadratic form obtained as the log of  monodromy around the corresponding  component  $\Delta_e$ of the discriminant.
Consequently,
the closure of the  monodromy cone is
$$
\overline \sigma(\widetilde C/C)=\RR_{\ge 0}\langle  (\tilde e^\vee-\iota \tilde e^\vee)^2 \rangle_{e\in E(\Gamma)}.
$$
\end{pro}

\begin{proof}  Let $\pi:\widetilde C\to C$ be an admissible cover of a stable curve of genus $g+1\ge 2$.
We will start by describing   the closure of the monodromy cone via the  log of monodromy as described in  \eqref{eqnPrecall2}.

First we will consider  the special case where $C$ has a unique node.  It is convenient to break this down further into three sub-cases.  To do this, let us fix some notation.  Let $e$ be the unique edge of $\Gamma$, the dual graph of $C$.   Let $\tilde e$ be an edge of $\widetilde \Gamma$ lying over $e$.
Then exactly one of the following holds:

\begin{enumerate}
\item[(0)]   $\widetilde e^\vee-\iota\tilde e^\vee=0\in H^1(\widetilde \Gamma,\QQ)$  (equivalently, $\widetilde C\to C\in \delta_{\overline{\mathcal R}_{g+1}}\setminus(\delta_{i,g+1-i}\cup \delta_0'$)).
\item [(1)]  $\widetilde e^\vee-\iota\tilde e^\vee = 2\tilde e^\vee\ne  0\in H^1(\widetilde \Gamma,\QQ)$  (equivalently, $\widetilde C\to C\in \delta_{i,g+1-i}$).
\item [(2)] $\widetilde e^\vee-\iota\tilde e^\vee \ne 2\tilde e^\vee, 0\in H^1(\widetilde \Gamma,\QQ)$  (equivalently, $\widetilde C\to C\in \delta_0'$).
\end{enumerate}
In case (0),  $H^1(\widetilde \Gamma,\QQ)^-=0$, the Prym  is an abelian variety, the monodromy is trivial, and there is nothing to show.
In cases (1) and (2), $H_1(\widetilde \Gamma,\QQ)^{-}=\QQ\langle \tilde e-\iota \tilde e\rangle$, where $e$ is the unique edge of $\Gamma$ and $\tilde e$ and $\iota \tilde e$ are the edges of $\widetilde \Gamma$ over $e$ interchanged by the involution.

Now we will describe the log of monodromy as a map
$$
\begin{CD}
H_1(\widetilde \Gamma,\mathbb Q)^{-}@>N_e^->>  H^1(\widetilde \Gamma,\mathbb Q)^-=\left(H_1(\widetilde \Gamma,\mathbb Q)^{-}\right)^\vee,
\end{CD}
$$
To do this, let us fix $\ast\in B$ an appropriate base point with $\widetilde C_\ast \to C_\ast$ an \'etale double cover.  In suitable coordinates on $H^1(\widetilde C_\ast,\CC)^-$, the monodromy operator is given by
$$
T_e=\left(
\begin{smallmatrix}
1&2&0&\cdots&0\\
0&1&0&\cdots & 0\\
\vdots && \ddots  &&\vdots \\
0&\cdots&0&1&0\\
0&\cdots &0&0&1\\
\end{smallmatrix}
\right)
\ \ \ \text{(resp. } \ \
T_e=\left(
\begin{smallmatrix}
1&1&0&\cdots&0\\
0&1&0&\cdots & 0\\
\vdots && \ddots  &&\vdots \\
0&\cdots&0&1&0\\
0&\cdots &0&0&1\\
\end{smallmatrix}
\right)
\text{)}
$$
This follows from the classification of admissible covers in \S \ref{secBD} (see also \cite{fs}).
Since the log of monodromy is given by $N_e=T_e-\Id$, it follows that  (up to a scalar multiple) $N_e^-(\tilde e -\iota \tilde e)=\tilde e^\vee -\iota \tilde e^\vee $.
Thus the associated quadratic form is $\left(\tilde e^\vee -\iota \tilde e^\vee\right)^2$.

 The general case then follows by the arguments in \S \ref{secttoroidal},  \S \ref{sectHT}   (esp.~\S\ref{secMGC} and \eqref{eqnMonComp}) by considering the composition
$$
\begin{CD}
H_1(\widetilde \Gamma,\QQ)^{-}@>>> H_1(\widetilde \Gamma_e,\QQ)^{-} @> N_e^- >> H^1(\widetilde \Gamma_e,\QQ)^-@>>> H^1(\widetilde \Gamma,\QQ)^-,
\end{CD}
$$
where $\widetilde \Gamma_e$ is the dual graph of the curve obtained from $\widetilde C$ by smoothing all of the nodes except those lying above the node of $C$ corresponding to $e$.
\end{proof}

\begin{rem}\label{remPrymMonCone}
Let $\psi:S\to B$ be a morphism from the disk to a mini-versal deformation space of the admissible cover $\widetilde C\to C$  induced from a $1$-parameter deformation of $\widetilde C\to C$.    For each edge $e\in \Gamma$, the dual graph of $C$, let $z_e$ be a local parameter defining the hyperplane $\Delta_e\subseteq B$ parameterizing covers with all nodes smoothed except those corresponding to $e$.  Then the log of monodromy for the family is  given by $\sum_{e\in \Gamma} \operatorname{ord}\psi^*(z_e)(\tilde e^\vee-\iota\tilde e^\vee)^{2}$.
\end{rem}

\section{Extension criteria for the Torelli and Prym map}\label{sectExt}
After the preliminaries of the previous sections, we can state our results regarding the extension of the period maps to various toroidal compactifications.
In general given a period map $\mathcal M\to \mathcal A_g$ and a normal crossing compactification $\mathcal M\subset \overline{\mathcal M}$, the question of extending to
the boundary essentially boils down to two steps: the computation of monodromy cones (which we did in the previous section) and then a check that a monodromy cone is contained in one of the cones of the fan defining a toroidal compactification (a combinatorial statement). Of course, this general process is well known and occurs in various guises in the literature,
but its systematic application in the case of admissible covers gives a good and uniform
understanding of the extensions of Prym maps to toroidal compactifications.

\subsection{Extension of the Torelli map}

To motivate the arguments for the Prym map,  we first review in this section the results of Alexeev and Brunyate \cite{ab} for the Torelli map.
Throughout  this subsection, we will use the following notation.   Fix $g\ge 2$ and $ C$ a stable curve in $\overline {\mathcal M}_g$.  Let   $\Gamma$ be the dual graph.
Recall from Proposition \ref{projmoncone} that the closure of the  monodromy cone for the admissible cover is the cone
$$
\overline \sigma(C):=\RR_{\ge 0}\langle (e^\vee)^2\rangle_{e\in E(\Gamma)}\subseteq \left(\Sym^2H_1(\Gamma,\ZZ)\right)^\vee_{\RR}.
$$

Fix a free $\ZZ$-module $\Lambda$ of rank $g$, and a $GL(\Lambda)$-admissible cone decomposition $\Sigma$ of $\overline C^{\QQ}(\Lambda)$.  Let $\bar A_{g}^{\Sigma}$ be the associated toroidal compactification.     Fix a surjection $\Lambda \to \Lambda_C =H_1( \Gamma,\ZZ)$, and denote by $\Sigma_C$ the  $GL(\Lambda_C)$-admissible cone decomposition of $\overline C^{\QQ}(\Lambda_C)$ (induced by the inclusion  $\overline C^{\QQ}(\Lambda_C)\hookrightarrow \overline C^{\QQ}(\Lambda)$).  Recall  that $\Sigma_C$ does not depend on the surjection $\Lambda\to \Lambda_C$.
We now compile results from the literature due to Mumford, Namikawa \cite{nam76II}, and Alexeev and Brunyate \cite{ab}.

\begin{teo}\label{teotorelliext}  Fix $g\ge 2$.
The Torelli map
$$
J^\Sigma:\overline {M}_{g}\dashrightarrow \bar {A}_{g}^{\Sigma}
$$
extends to a morphism in a neighborhood of a stable curve $C$ if and only if there exists a cone $\sigma\in \Sigma_C$ of the admissible decomposition containing the monodromy cone $\sigma(C)$.
\begin{enumerate}
\item \emph{(Mumford--Namikawa {\cite[Cor.~18.9]{nam76II}})} The Torelli map extends to a morphism to the second Voronoi compactification:
$$
J^V:\overline M_g\to \AV.
$$

\item \emph{(Alexeev--Brunyate {\cite[Thm.~4.7, Thm.~6.7]{ab}})} The Torelli  map extends to a morphism to the perfect cone compactification:
$$
  J^P:\overline M_g\to \AP.
$$

\item \emph{(Alexeev--Brunyate {\cite[Thm.~4.8]{ab}})}
The Torelli map extends to a morphism to  $\AC$  in a neighborhood of $C\in \overline{M}_{g}$ if and only if  there exists a quadratic form $Q$ on $H^1( \Gamma,\RR)$ such that
\begin{enumerate}
\item $Q(r)>0$ for all $r\in H^1(\Gamma,\RR)\setminus \{0\}$; i.e.~$Q$ is positive definite.

\item $Q(\ell)\ge 1$ for all $\ell \in H^1( \Gamma,\ZZ) \setminus \{0\}$.

\item $Q(e^\vee)=1$, for all $e\in E(\Gamma)$ such that $e^\vee\ne 0$.
\item $Q(\ell)\in \ZZ$ for all $\ell \in H^1( \Gamma,\ZZ)$.

\end{enumerate}

\end{enumerate}

\end{teo}

\begin{proof}
The first statement of the theorem  follows from the standard results on toroidal compactifications discussed in \S \ref{secextpermap}.
We sketch the proofs of the remaining parts following Alexeev and Brunyate.

(1)
From Proposition \ref{propmoncone} and Lemma \ref{lemsecvor}, it follows that the Torelli map extends to a morphism to $\AV$ in a neighborhood of  $C$ if and only if for any collection  $e_1,\ldots,e_m\in E(\Gamma)$  of edges such that the co-cycles $e_1^\vee,\ldots,e_m^\vee$ form a basis for $H^1( \Gamma,\RR)$, the co-cycles $e_1^\vee,\ldots,e_m^\vee$ in fact form a $\ZZ$-basis for $H^1( \Gamma,\ZZ)$.
On the other hand, an elementary result from graph theory  (see eg.~\cite[Lem.~3.3]{ab}, \cite[(J6), p.95]{abh}) asserts the following:  \emph{For a graph $\Gamma$ and a collection of edges  $e_1,\ldots, e_m\in E(\Gamma)$, the co-cycles $e_1^\vee,\ldots,e_m^\vee$ form a $\ZZ$-basis of $H^1(\Gamma,\ZZ)$ if and only if  the co-cycles  $e_1^\vee,\ldots,e_m^\vee$ form an $\RR$-basis of $H^1(\Gamma,\RR)$, if and only if  the graph obtained from $\Gamma$ by removing the edges $\{e_1,\ldots,e_m\}$ is a spanning tree} (i.e.~$b_0=1$, $b_1=0$, and it contains all the vertices).  This completes the proof.

(2) The monodromy cone is generated by rank $1$ quadrics, and
in the previous paragraph was shown to be matroidal.
Consequently, the monodromy cone is a perfect cone \cite[Thm.~A]{MV12} (see Remark \ref{remMV} and Remark \ref{remTorPC}).  Thus the period map extends.

(3) This is  a restatement of Lemma \ref{lemcc}.
\end{proof}

\begin{rem}\label{remTorPC}
For (2), it should be noted
 that  from  Lemma \ref{lempc} it follows that the Torelli map extends   in a neighborhood of $C\in \overline M_g$ if and only if there exists a positive definite quadratic form $Q$ on $H^1(\Gamma,\RR)$ such that  $Q(\ell)\ge 1$ for all $\ell \in H^1(\Gamma,\ZZ) -\{0\}$ and  $Q(e^\vee)=1$, for all $e\in E(\Gamma)$ such that $e^\vee \ne 0$.    In \cite[Thm.~6.7, p.194]{ab} Alexeev and Brunyate establish the existence of such quadratic forms, providing the proof of this case of   \cite[Thm.~A]{MV12}.
\end{rem}

\begin{rem}\label{remTorCC}
For $g\le 4$ the central cone compactification agrees with the perfect cone compactification, and consequently the Torelli map extends to a morphism to  $\AC$ for $g\le 4$.  In fact, in \cite[Cor.~5.4]{ab}, \cite[Cor.~1.2]{AETAL} it is established that all dual graphs of genus $g\le 8$ admit a quadratic form as in Theorem \ref{teotorelliext} (3), and so the Torelli map  extends to a morphism to $\AC$ for all $g\le 8$.  On the other hand, there are  dual graphs of curves of all genera $g\ge 9$ that do not admit such quadratic forms  \cite[Cor.~5.6]{ab}.  Consequently, the Torelli map does not extend to a morphism to  $\AC$ for $g\ge 9$.
\end{rem}

\begin{rem}
Recall from Remark \ref{remstack}
that
the Torelli map extends to a toroidal compactification, as a map of stacks, if and only if it extends as a map of the coarse moduli spaces.
\end{rem}

\subsection{Extension of the Prym map}
Throughout  this subsection, we will use the following notation.   Fix $g+1\ge 2$ and $\pi:\widetilde C\to C$ an admissible double cover in $\overline{\mathcal R}_{g+1}$.  Let $\widetilde \Gamma$ and $\Gamma$ be the dual graphs of $\widetilde C$ and $C$, respectively.  To simplify the discussion,  fix once and for all, for  each edge $e\in E(\Gamma)$ a choice of edge $\tilde e\in E(\widetilde \Gamma)$ lying over $e$.
Having made this choice, then for each  edge $e\in E(\Gamma)$, fix a co-cycle $\ell_e\in H^1(\widetilde \Gamma,\ZZ)^-$ by the rule:
\begin{equation}\label{EQNdefle}
 \ell_e:=
 \left\{
 \begin{array}{ll}
 \tilde e^\vee-\iota\tilde e^\vee& \text{if } \  \iota \tilde e^\vee \ne -\tilde e^\vee  \in H^1(\widetilde \Gamma,\mathbb Z),\\
 \tilde e^\vee& \text{if }\  \iota\tilde e^\vee=-\tilde e^\vee \in H^1(\widetilde \Gamma,\mathbb Z).\\
 \end{array}
 \right.
\end{equation}
Recall from Proposition \ref{propmoncone} that the closure of the  monodromy cone for the admissible cover is the cone
$$
\overline \sigma(\widetilde C/C):=\RR_{\ge 0}\langle \ell_e^2\rangle_{e\in E(\Gamma)}\subseteq \left(\Sym^2H_1(\widetilde \Gamma,\ZZ)^{[-]}\right)^\vee_{\RR}.
$$
Note that $\ell_e^2$ does not depend on the choice of $\tilde e$ lying over a fixed $e\in \Gamma$.

\begin{rem}\label{REMle}
The definition of $\ell_e$ is made to ensure that $\ell_e$ is primitive in $H^1(\widetilde \Gamma,\mathbb Z)^-$.   It is important that one takes the condition $\iota \tilde e^\vee=-\tilde e^\vee$ as being in  $H^1(\widetilde \Gamma,\mathbb Z)$. Note in particular that $\iota \tilde e^\vee$ never agrees with $-\tilde e^\vee$ in $C^1(\widetilde \Gamma,\mathbb Z)$, but always agrees with $-\tilde e^\vee$ viewed as a linear function on $H_1(\widetilde \Gamma,\mathbb Z)^{[-]}$.
\end{rem}

Fix a free $\ZZ$-module $\Lambda$ of rank $g$, and a $GL(\Lambda)$-admissible cone decomposition $\Sigma$ of $\overline C^{\QQ}(\Lambda)$.  Let $\bar A_{g}^{\Sigma}$ be the associated toroidal compactification.     Fix a surjection $\Lambda \to \Lambda_{\widetilde C/C} =H_1(\widetilde \Gamma,\ZZ)^{[-]}$, and denote by $\Sigma_{\widetilde C/C}$ the  $GL(\Lambda_{\widetilde C/C})$-admissible cone decomposition of $\overline C^{\QQ}(\Lambda_{\widetilde C/C})$ (induced by the inclusion  $\overline C^{\QQ}(\Lambda_{\widetilde C/C})\hookrightarrow \overline C^{\QQ}(\Lambda)$).  Recall that $\Sigma_{\widetilde C/C}$ does not depend on the surjection $\Lambda\to \Lambda_{\widetilde C/C}$.
We now use this to prove an extension theorem for the Prym map.  The case of the second Voronoi compactification gives another proof of  \cite[Thm.~3.2 (1), (4)]{abh} (see Remark \ref{remABH} below), while the results for the perfect and central cone are new.

\begin{teo}\label{teoprymext}  Fix $g\ge 1$.
The Prym map
$$
P^\Sigma:\overline {R}_{g+1}\dashrightarrow \bar {A}_{g}^{\Sigma}
$$
extends to a morphism in a neighborhood of an admissible cover  $\pi:\widetilde C\to C$ if and only if there exists a cone $\sigma\in \Sigma_{\widetilde C/C}$ of the admissible decomposition containing the monodromy cone $\sigma(\widetilde C/C)$.

\begin{enumerate}
\item \emph{(Alexeev--Birkenhake--Hulek \cite[Thm.~3.2 (1), (4)]{abh})}   The Prym map extends to a morphism to the second Voronoi compactification $\AV$ in a neighborhood of  $(\pi:\widetilde C\to C)\in \overline{R}_{g+1}$ if and only if:
\vskip .2 cm

\begin{enumerate}

\item[(V)] For any collection  $e_1,\ldots,e_m\in E(\Gamma)$  of edges such that the corresponding co-cycles $\ell_{e_1},\ldots,\ell_{e_m}$ form a basis for $H^1(\widetilde \Gamma,\RR)^-$, the co-cycles $\ell_{e_1},\ldots,\ell_{e_m}$ in fact form a $\ZZ$-basis for $H^1(\widetilde \Gamma,\ZZ)^-$.
\end{enumerate}
 \vskip .2 cm

\item  The Prym map extends to a morphism to the perfect cone compactification   $\bar{  A}_{g}^{P}$  in a neighborhood of $(\pi:\widetilde C\to C)\in \overline{R}_{g+1}$ if and only if:
 \vskip .2 cm

\begin{enumerate}
\item[(P)] There exists a quadratic form $Q$ on $H^1(\widetilde \Gamma,\RR)^-$ such that:
\begin{enumerate}
\item $Q(r)>0$ for all $r\in H^1(\widetilde \Gamma,\RR)^--\{0\}$; i.e.~$Q$ is positive definite.

\item $Q(\ell)\ge 1$ for all $\ell \in H^1(\widetilde \Gamma,\ZZ)^- -\{0\}$.

\item $Q(\ell_e)=1$, for all $e\in E(\Gamma)$ such that $\ell_e\ne 0$.
\end{enumerate}
\end{enumerate}
 \vskip .2 cm

\item The Prym map extends to a morphism to the central cone compactification   $\bar {A}_{g}^{C}$  in a neighborhood of $(\pi:\widetilde C\to C)\in \overline{R}_{g+1}$ if and only if:

\begin{enumerate}
\item[(C)]  There exists a quadratic form $Q$ on $H^1(\widetilde \Gamma,\RR)^-$ such that in addition to satisfying \emph{(i)-(iii)} above, $Q$ also satisfies:

 \vskip .2 cm
\begin{enumerate}
\item[(iv)] $Q(\ell)\in \ZZ$ for all $\ell \in H^1(\widetilde \Gamma,\ZZ)^-$.

\end{enumerate}
\end{enumerate}

\end{enumerate}
\end{teo}

\begin{proof}
The first statement of the theorem  follows from the standard results on toroidal compactifications discussed in \S \ref{secextpermap}.  (1) then follows from  Proposition \ref{propmoncone} and Lemma \ref{lemsecvor}.  (2) follows from Lemma \ref{lempc} and (3) from Lemma \ref{lemcc}.
\end{proof}

\begin{rem}\label{remABH}
To see that  Theorem \ref{teoprymext} (1)  is equivalent to \cite[Thm.~3.2 (1), (4)]{abh} observe that it follows from Lemma \ref{lemdice} that (V) is equivalent to
\begin{enumerate}
\item[(V')] The linear functions $\{\ell_{e}\}_{e\in E(\Gamma)}$  define a dicing of the lattice $H_1(\widetilde \Gamma,\ZZ)^{[-]}$.
\end{enumerate}
Then note that
 as functions on $H_1(\widetilde \Gamma,\ZZ)^{[-]}$, the linear forms  $\tilde e^\vee-\iota \tilde e^\vee$ and  $2\tilde e^\vee$ agree (see Remark \ref{REMle}).
Thus the  condition (V') here is the same as the condition ($\ast$) of \cite[p.98]{abh}, and so   (V) is equivalent to  the condition for extension given in \cite[Thm.~3.2 (1), (4)]{abh}.
\end{rem}

\begin{rem}
Recall from Remark \ref{remstack}
that the following statement holds also for stacks.  The Prym map
$$
P^\Sigma:\overline {\mathcal R}_{g+1}\dashrightarrow \bar{\mathcal A}_g^\Sigma
$$
 extends to a morphism in a neighborhood of an admissible cover  $\pi:\widetilde C\to C$ if and only if there exists a cone $\sigma\in \Sigma_{\widetilde C/C}$ of the admissible decomposition containing the monodromy cone $\sigma(\widetilde C/C)$.
 \end{rem}


\section{Monodromy cones for Friedman--Smith covers}\label{secFSexamples}
We now investigate a class of admissible covers discovered by Friedman and Smith \cite{fs}, who used these examples to show that the Prym map does not extend to the second Voronoi, perfect cone or central cone compactifications.  Alexeev, Birkenhake, and Hulek \cite{abh} and Vologodsky \cite{vologodsky}  then showed that   these examples characterize the indeterminacy locus of the Prym map to the second
Voronoi compactification.  In this section we give a detailed description of the monodromy cone for these examples with the aim of giving  a geometric
characterization of the indeterminacy locus of the Prym map to the perfect and central cone compactifications.  In the subsequent sections  we will actually need some
more elaborate monodromy computations for further degenerations of these examples.  The method for obtaining these is the same as the one discussed here,
and thus all such further computations will be given in the appendix.  In the main body of the paper we will reference those combinatorial results as needed.

\subsection{Friedman--Smith covers}
A \emph{Friedman--Smith cover with $2n\ge 2$ nodes} (see also Figure \ref{Fig:dgFSG}) is an admissible cover $\pi:\widetilde C\to C$ such that
\begin{enumerate}
\item $\widetilde C=\widetilde C_1\cup \widetilde C_2$ with $\widetilde C_1$ and $\widetilde C_2$ irreducible and smooth, and $$\widetilde C_1\cap \widetilde C_2=\{\tilde p_1^+,\tilde p_1^-\ldots,\tilde p_{n}^+,\tilde p_{n}^-\}.$$
\item  $\iota \widetilde C_i=\widetilde C_i$ for $i=1,2$,
\item   $\iota \tilde p_i^+=\tilde p_i^-$ for $i=1,\ldots,n$.
\end{enumerate}

\begin{rem}
An admissible cover $\pi:\widetilde C\to C$ is called a {\em degeneration of a Friedman--Smith cover with $2n$ nodes} if it can be obtained from  a Friedman--Smith cover by a further degeneration. More precisely, an admissible cover $\pi:\widetilde C\to C$  is such a degeneration if and only if   $\widetilde C=\widetilde C_1\cup \widetilde C_2$ with $\widetilde C_1$ and $\widetilde C_2$ connected (possibly reducible), $\widetilde C_1\cap \widetilde C_2=\{\tilde p_1^+,\tilde p_1^-\ldots,\tilde p_{n}^+\tilde p_{n}^-\}$,  $\iota \widetilde C_i=\widetilde C_i$ for $i=1,2$, and   $\iota \tilde p_i^+=\tilde p_i^-$ for $i=1,\ldots,n$.
\end{rem}
For later use, we denote by $FS_{n}\subseteq \overline{R}_{g+1}$ the locus of Friedman--Smith covers, and by $\overline{FS}_{n}$ its closure; i.e.~the locus of degenerations of Friedman--Smith covers.
A \emph{(degeneration of a) Friedman--Smith graph} is a dual graph together with an involution, which can be obtained as the dual graph of a (degeneration of a) Friedman--Smith cover with induced involution.

\begin{rem}
The following is slightly stronger than a direct translation of the remark above into the language of graphs.
A graph $\widetilde \Gamma$ with admissible  involution $\iota$ is a  degeneration of a Friedman--Smith graph with at least $2n\ge 2$ edges if and only if  $\widetilde \Gamma$ admits disjoint,  connected subgraphs  $\widetilde \Gamma_1,\widetilde \Gamma_2$ connected by exactly  $2m\ge 2n$ edges $\tilde e_1^+,\tilde e_1^-,\ldots,\tilde e_m^+,\tilde e_m^-$,  with  $\iota (\widetilde \Gamma_i)=\widetilde \Gamma_i$ ($i = 1,2$), and $\iota \tilde e_i^+=\tilde e_i^-$ ($i=1,\ldots,m$), and furthermore  $\widetilde \Gamma_1$ and  $\widetilde \Gamma_2$ are not connected by a $\iota$-invariant path \cite[Lem.~1.2]{vologodsky}.
\end{rem}

One can see that $\overline{FS}_1=\bigcup \delta_{i,g-i}$ and for $n\ge 2$,  $\overline{FS}_n$ is codimension $n$ (if non-empty), and  contained in  the $n$-fold self-intersection of  $\delta_0'$.
In  $\overline {\mathcal R}_{g+1}$  there are $
\lfloor \frac{g-n+2}{2} \rfloor
$
irreducible components of $FS_{n}$, determined by the pairs of genera $(g(C_1),g(C_2))$ given by
$$
(1,g-n+1), (2,g-n),\ldots (\lfloor \frac{g-n+2}{2} \rfloor,\lfloor \frac{g-n+3}{2} \rfloor).
$$
In particular $FS_{n}=\emptyset$ in $\overline{\mathcal R}_{g+1}$  if $n\ge g+1$.  Note also that the covers $\widetilde C_i\to C_i$ are \'etale, so that in particular, the curves $\widetilde C_i$  have odd genus $2g(C_i)-1$.

\begin{rem}
The index $n$ for $\overline{FS}_{n}$ refers to the codimension of the locus  in $\overline{R}_{g+1}$, or equivalently the number of edges in the dual graph of the base curve. We will use similar notational conventions for other loci occurring later in the paper.
\end{rem}

\subsection{The monodromy cone}\label{secFSMonCone}
Let $\pi:\widetilde C\to C$ be a Friedman--Smith cover with $2n\ge 2$ nodes.    The dual graph $\widetilde \Gamma$ of $\widetilde C$ has vertices $V(\widetilde \Gamma)=\{\tilde v_1,\tilde v_2\}$ and edges $E(\widetilde \Gamma)=\{\tilde e_1^+,\tilde e_1^-,\ldots,\tilde e_n^+,\tilde e_n^-\}$.    The involution $\iota$ acts by $\iota(\tilde v_i)=\tilde v_i$ ($i=1,2$) and $\iota(\tilde e_i^+)=\tilde e_i^-$ ($i=1,\ldots,n$).
For simplicity, we will fix a compatible orientation on $\widetilde \Gamma$, as in Figure \ref{Fig:dgFSG};  i.e.~for all $i$ set $t(\tilde e_i^{\pm})=\tilde v_2$ and $s(\tilde e_i^{\pm})=\tilde v_1$.

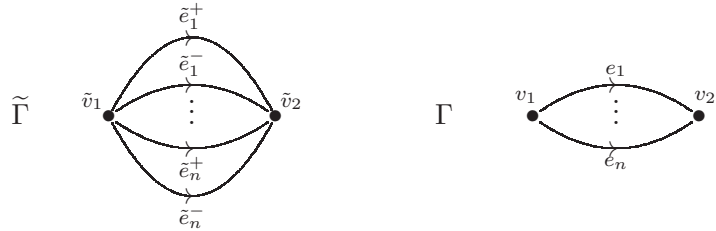
\begin{figure}[htb]
\begin{equation*}
\xymatrix{
\widetilde \Gamma & *{\bullet} \ar @{-}@/_1pc/[rr]|-{\SelectTips{cm}{}\object@{>}}_{\tilde e_n^+} \ar @{-} @/_2.5pc/[rr] |-{\SelectTips{cm}{}\object@{>}}_{\tilde e_n^-}
 \ar@{-} @/^1pc/[rr]|-{\SelectTips{cm}{}\object@{>}}^{\tilde e_1^-}  \ar@{-}@/^2.5pc/[rr]|-{\SelectTips{cm}{}\object@{>}}^{\tilde e_1^+}^<{\tilde v_1}^>{\tilde v_2}  &\vdots &*{\bullet}
&&\Gamma&
 *{\bullet} \ar @{-}@/_1pc/[rr]|-{\SelectTips{cm}{}\object@{>}}_{e_n} \ar @{-}
 \ar@{-} @/^1pc/[rr]|-{\SelectTips{cm}{}\object@{>}}^{e_1}  ^<{ v_1}^>{ v_2}  &\vdots &*{\bullet}
}
\end{equation*}

\caption{  Dual graph of a Friedman--Smith example with $2n\ge 2$ nodes ($FS_n$).}\label{Fig:dgFSG}
\end{figure}
One has
\begin{equation}\label{eqnh1fs}
H_1(\widetilde \Gamma,\ZZ)= \ZZ\langle \tilde e_1^+-\tilde e_1^-,\ldots,\tilde e_n^+-\tilde e_n^-, \tilde e_1^+-\tilde e_2^-,\ldots,\tilde e_{n-1}^+-\tilde e_n^-\rangle.
\end{equation}
Indeed, we have
$b_1(\widetilde \Gamma)=\#E(\widetilde \Gamma)-\#V(\widetilde \Gamma)+b_0(\widetilde \Gamma)=2n-1$, since $\widetilde \Gamma$ is connected.    The $2n-1$ elements listed above are in fact  a generating set for $H_1(\widetilde \Gamma,\ZZ)$,  as can be easily detected from  the associated matrix.  For instance, if one takes the elements in the order
$\tilde e_1^+-\tilde e_1^-, \tilde e_1^+-\tilde e_2^-,\ldots,\tilde e_n^+-\tilde e_n^-,\tilde e_{n-1}^+-\tilde e_n^-$ and constructs a matrix with rows expressing these elements  with respect to  the basis $ \tilde e_1^-,\tilde e_1^+, \ldots, \tilde e_n^-,\tilde e_n^+$,  one obtains a $(2n-1)\times (2n)$ matrix whose first $(2n-1)\times (2n-1)$ sub-matrix is upper triangular
with all the diagonal entries  equal to $\pm1$.

Recall that $H_1(\widetilde \Gamma,\ZZ)^{[-]}=H_1(\widetilde \Gamma, \ZZ)/H_1(\widetilde\Gamma,\ZZ)^+$ and is isomorphic to the image of the map $$\frac{1}{2}(\Id -\iota):H_1(\widetilde \Gamma,\ZZ)\to H_1(\widetilde \Gamma,\RR).$$
From \eqref{eqnh1fs}, one has
$$
H_1(\widetilde \Gamma,\ZZ)^{[-]}\cong  \ZZ\langle \tilde e_1^+-\tilde e_1^-,\frac{1}{2}(\tilde e_1^+-\tilde e_1^-)+\frac{1}{2}(\tilde e_2^+-\tilde e_2^-),\ldots,\frac{1}{2}(\tilde e_{n-1}^+-\tilde e_{n-1}^-)+\frac{1}{2}(\tilde e_n^+-\tilde e_n^-)\rangle.
$$
For brevity, set
$$
z_1=\tilde e_1^+-\tilde e_1^- , \ z_2=\frac{1}{2}(\tilde e_1^+-\tilde e_1^-)+\frac{1}{2}(\tilde e_2^+-\tilde e_2^-), \ldots, \ z_n=\frac{1}{2}(\tilde e_{n-1}^+-\tilde e_{n-1}^-)+\frac{1}{2}(\tilde e_n^+-\tilde e_n^-)
$$
so that
$
H_1(\widetilde \Gamma,\ZZ)^{[-]}\cong \ZZ\langle z_1,\ldots,z_n\rangle.
$
Then $$H^1(\widetilde \Gamma,\ZZ)^-=\left(H_1(\widetilde \Gamma,\ZZ)^{[-]}\right)^\vee\cong \ZZ\langle z_1^\vee,\ldots,z_n^\vee \rangle .$$

Now observe that
$$H^1(\widetilde \Gamma,\ZZ)=\ZZ\langle (\tilde e_1^+)^\vee, (\tilde e_1^-)^\vee, \ldots, (\tilde e_n^+)^\vee, (\tilde e_n^-)^\vee \rangle/ \langle  (\tilde e_1^+)^\vee+(\tilde e_1^-)^\vee+ \ldots+ (\tilde e_n^+)^\vee+ (\tilde e_n^-)^\vee \rangle.
$$
It follows that for $i=1,\ldots,n$,
$$
\begin{array}{ll}
\iota(\tilde e_i^+)^\vee=(\tilde e_i^-)^\vee=-(\tilde e_i^+)^\vee& \text{if }\  n=1,\\
\iota (\tilde e_i^+)^\vee=(\tilde e_i^-)^\vee\ne -(\tilde e_i^+)^\vee& \text{if }\  n\ge 2.\\
\end{array}
$$
Consequently, we may choose for $i=1,\ldots,n$,
$$
\ell_{e_i}:=
\left\{
\begin{array}{ll}
(\tilde e_i^+)^\vee& \text{if }\ n=1,\\
(\tilde e_i^+)^\vee-(\tilde e_i^-)^\vee& \text{if } \ n\ge 2 .\\
\end{array}
\right.
$$

For $n=1$, $\ell_{e_1}$ is clearly a basis for $H^1(\widetilde \Gamma,\ZZ)^-$, and so we note that condition (V) of Theorem \ref{teoprymext} holds in this case.  Now consider the case $n\ge 2$.
  Evaluating the $\ell_{e_i}$ on the basis $z_1,\ldots, z_n$, we obtain that
\begin{equation}
\begin{array}{ccccccccccc}
\ell_{e_1}& =& 2z_1^\vee& +&z_2^\vee&&&&&&\\
\ell_{e_2}& =& & &z_2^\vee&+&z_3^\vee&&&&\\
\vdots &\vdots &&&&\ddots&&&&\\
\vdots &\vdots &&&&&\ddots&&&\\
\ \ \ell_{e_{n-1}}& =& & &&&&&z_{n-1}^\vee&+&z_n^\vee\\
\ell_{e_{n}}& =& &&&&&&&&z_n^\vee\\
\end{array}
\end{equation}
Thus, with respect to these bases,  the matrix representation of the monodromy cone is:
\begin{equation}\label{eqnFSMCo}
\left(\begin{smallmatrix}
2&1&&&&&\\
&1&1&&&&\\
&&1&1&&&\\
&&&\ddots&\ddots&&\\
&&&&1&1&\\
&&&&&1&1\\
&&&&&&1\\
\end{smallmatrix}\right).
\end{equation}
%
Since the determinant of this matrix is $2$, it follows that for $n\ge 2$,  $\{\ell_{e_1},\ldots,\ell_{e_n}\}$ is a basis for  $H^1(\widetilde \Gamma,\RR)^-$, but is \emph{not} a $\ZZ$-basis for $H^1(\widetilde \Gamma,\ZZ)^-$.   Consequently, condition (V) does not hold for $n\ge 2$.

\subsection{Properties of Friedman--Smith monodromy cones}

With the above description of the Friedman--Smith monodromy cone, it is now a combinatorial problem to describe the relationship of the Friedman--Smith monodromy cone to the various cone decompositions.  The details of the arguments are contained in the appendix.  Here we compile the results for reference.

\begin{teo}\label{teoFSMCP}  A Friedman--Smith cone is:
\begin{enumerate}
\item Basic for $n \geq 3$, $n=1$, and simplicial but not basic for $n=2$.

\item
Matroidal if and only if $n=1$.  Every proper face of a Friedman--Smith cone is matroidal.

\item
Contained in a cone in the  perfect cone decomposition  if and only if $n\ne 2,3$.   In fact, a Friedman--Smith cone  \emph{is} a cone in the perfect cone decomposition if and only if $n\ne 2,3,4$.

\end{enumerate}

\end{teo}

\begin{proof}
(1) See Lemma \ref{lemBasicCone}.
(2) See Lemma \ref{lemFSMAT}.
(3)  See Proposition \ref{proFS=PC}.
\end{proof}

\begin{rem}
In Appendix \ref{secappMDS}, Dutour  Sikiri\'c shows that   the Friedman--Smith monodromy cone is contained in a cone in the  central cone decomposition  if and only if $n\ne 2,3$.
\end{rem}

\begin{rem}\label{remFSSVD}
It is an unpublished result of Alexeev that \emph{for $n\ge 2$, each cone in the  barycentric subdivision of a Friedman--Smith cone  is contained in a cone in the   second Voronoi decomposition} (see \cite[p.3159]{vologodskyres}). It is easy to see that the decomposition of the Friedman--Smith cone into cones contained in second Voronoi cones must be  a refinement of the barycentric subdivision.  One can then argue as in Vologodsky \cite{vologodskyres} to show that the  barycentric subdivision suffices.    We will use this in the case $n=2,3$ in our investigation of the resolution of the period map to the perfect cone compactification  (see Remark \ref{relationsconedec}).  In the appendix (\S \ref{secFScd}) we give an explicit description of second  Voronoi cones  generated by rank $1$ quadrics that decompose the Friedman--Smith cone for $n=2,3$, providing another proof in these special cases.  We will use these explicit cones in studying other monodromy cones of degenerations of Friedman--Smith covers, and also in describing Delaunay decompositions.
\end{rem}


\section{The indeterminacy locus of the Prym map}\label{sectindeterm}

Here we further investigate the indeterminacy locus of the Prym map by  reformulating  the combinatorial characterization given in  Theorem \ref{teoprymext}, in terms of geometry.  For the second Voronoi compactification, Vologodsky \cite[Thm.~0.1]{vologodsky} has shown that the combinatorial condition in Theorem \ref{teoprymext} (1) is equivalent to the cover being a degeneration of a Friedman--Smith cover with at least $4$ nodes.
In other words, the indeterminacy locus for the Prym map to the   second Voronoi compactification is equal to $\bigcup_{n\ge 2} \overline {FS}_n$.   Consequently, here we   focus on the period map to the  perfect cone compactification, for which it turns out that the indeterminacy locus is smaller.  While at the moment we are unable to obtain a statement if full generality analogous to \cite[Thm.~0.1]{vologodsky}, we describe completely the situation  up to codimension $6$.
\begin{teo}\label{teoindPM}
The indeterminacy locus of the Prym map $P_g^P:\overline R_{g+1}\dashrightarrow \bar A_g^P$ satisfies
\begin{equation}\label{eqnIndLoc}
\overline {FS}_2\cup \overline {FS}_3\subseteq Ind(P_g^P)\subseteq \overline {FS}_2\cup \overline {FS}_3\cup \partial \overline{FS}_4\cup\ldots \cup \partial \overline {FS}_{g}
\end{equation}
where $\partial \overline{FS}_{n}=\overline{FS}_{n}-FS_{n}$.  Moreover,
$$
 \operatorname{codim}_{\overline R_{g+1}} Ind(P_g^P)\setminus \left(\overline {FS}_2\cup \overline {FS}_3\right)\ge 6.
$$
\end{teo}

\begin{rem}
Since $\bar A^P_1=A^*_1$, it is immediate (from the Borel extension theorem) that for $g=1$ the Prym map is a morphism $\overline{R}_2\to\bar {A}_1^P$; in this case both $FS_2$ and $FS_3$ are empty. For $g=2$ the locus $FS_3$ is empty, so we have $Ind(P_2^P)=\overline {FS}_2$.  Similarly, $Ind(P_3^P)=\overline {FS}_2\cup \overline {FS}_3$.
For $P^P_4:\overline R_5\dashrightarrow \bar A_4^P$ we have $\partial \overline {FS_4}\setminus \left(\overline {FS}_2\cup \overline {FS}_3\right)\ne \emptyset$.  In this case, the theorem above says roughly that the ``generic points'' of this locus do not lie in the indeterminacy locus. We note that in the course of the proof we will obtain a slightly stronger result than the statement of the theorem, by showing that the Prym map extends to $\overline{FS}_4$ and $\overline{FS}_5$ up to codimension two.
\end{rem}

\begin{rem}  We have the following relationships among the indeterminacy loci.  For $g\le 3$, $Ind(P_g^P)=Ind(P_g^V)=Ind(P_g^C)$.  
For $g\ge 4$, $Ind(P_g^P)\subsetneq Ind(P_g^V)$.        In  Appendix \ref{secappMDS}, Mathieu Dutour Sikiri\'c  shows that for $g\ge 4$, $Ind(P_g^C)\nsupseteq  Ind(P_g^V)$, and 
for $g\ge 9$, $Ind(P_g^C)\nsubseteq Ind(P_g^V)$.   
\end{rem}

In the process of proving the results in this paper, we have considered a number of degenerations of Friedman--Smith covers.   In these examples, the monodromy cone has failed to be contained in a cone in the PCD if and only if the example lies in $\overline{FS}_2\cup \overline{FS}_3$.    We thus pose the following question.

\begin{ques} \label{quesindeterminacy}
Is it true that the indeterminacy locus $Ind(P_g^P)$ is equal to $\overline{FS}_2\cup \overline{FS}_3$?
\end{ques}

\begin{proof}[Proof of Theorem \ref{teoindPM}]
We start by showing:
$$
\overline {FS}_2\cup \overline {FS}_3\subseteq Ind(P_g^P)\subseteq \overline {FS}_2\cup \overline {FS}_3\cup \partial \overline{FS}_4\ldots \cup \partial \overline {FS}_{g}.
$$
Theorem \ref{teoFSMCP} (3) implies that  the covers in the loci $FS_2, FS_3$ have monodromy cones not contained in cones in the PCD.  This gives the left inclusion. For the right inclusion, the results \cite[Thm.~3.2 (1), (4)]{abh} and \cite[Thm.~0.1]{vologodsky} imply that on $\overline R_{g+1}\setminus \left(\overline {FS}_2\ldots \cup  \overline {FS}_{g}\right)$ the monodromy cones are matroidal, which by \cite[Thm.~A]{MV12} are also perfect (see Remark \ref{remMV}). Moreover, in Theorem \ref{teoFSMCP} (3) we showed that for $4\le n\le g$ a  cover in  $FS_n$ has a monodromy cone contained in a cone in the PCD, and thus the period map extends there as well.

We now prove   $$ \operatorname{codim}_{\overline R_{g+1}} Ind(P_g^P)\setminus \left(\overline {FS}_2\cup \overline {FS}_3\right)\ge 6.$$   Since $\operatorname{codim} \overline{FS}_n=n$, it is enough to restrict attention to $\partial \overline{FS}_4$.    In fact
we will show the stronger statement that
$$
\operatorname{codim}_{\overline {FS}_{n}} (Ind(P_g^P) \cap \overline{FS}_n) \setminus \left(\overline {FS}_2\cup \overline {FS}_3\right)\ge 2
$$
for $n=4,5$.
To achieve this, we simply need to enumerate the codimension $1$ degenerations in $\overline{FS}_n$ for $n=4,5$ and check that for each of them the monodromy cone is not contained in a cone in the PCD if and only the degeneration also lies in $\overline {FS}_2$ or $\overline{FS}_3$.
To be precise, we will consider all degenerations of an $\overline{FS}_n$ cover so that the dual graph of the base curve has exactly $n+1$ edges.  The complement of this locus has codimension $2$ in $\overline {FS}_n$.  
We observe that the dual graph of the base of a  degeneration of an $\overline{FS}_n$ cover has     exactly $n+1$ edges if and only if the dual graph of the covering curve is obtained by replacing a vertex in the dual graph of an $FS_n$ cover (see Figure \ref{Fig:dgFSG})  with one of the dual graphs in Figures
\ref{Fig:d0'}-\ref{Fig:digi}.    Thus we have five cases to consider.

First consider the case where we replace the vertex with a graph as in Figure \ref{Fig:d0''} (corresponding to $\delta_0''$).   This give rise to  an $FS_{n}+W_1$ example (see Figure \ref{Fig:FSn+Wn} with  $m=1$). The monodromy computation is made in  \S \ref{secFS+W}, and establishes Lemma \ref{lemmcFSW} stating that for $n\le 7$, the monodromy cone is contained in a cone in the PCD if and only if the cover is not a degeneration of an $FS_2$ or $FS_3$ cover.

Next consider the case where we replace the vertex with a graph as in Figure \ref{Fig:digi} (corresponding to $\delta_{i,g+1-i}$).
This gives  rise to $FS_{n_1+n_2}+FS_1$    examples  with $n_1+n_2=n$ (see Figure \ref{Fig:FSn+FSm} with $m=1$). The monodromy computation is made in   \S \ref{secFS+FS}, and establishes Lemma \ref{lemFS+FS} stating that the monodromy cone is contained in a cone in the PCD if and only if the cover is not a degeneration of an $FS_2$ or $FS_3$ cover.

Next consider the case where we replace the vertex with a graph as in Figure \ref{Fig:di} (corresponding to $\delta_{i}$).
This gives rise to  $FS_{n_1+n_2}+\delta_i$ examples with $n_1+n_2=n$ (see Figure \ref{Fig:dgFSn+d}).  The monodromy computation is made in \S \ref{secFS+d}, culminating in Lemma  \ref{lemFS+d}, which shows that
for $n\le 5$, the monodromy cone is contained in a cone in the PCD if and only if the cover is not a degeneration of an $FS_2$ or $FS_3$ cover.

The cases where we replace the vertex with a graph as in Figure \ref{Fig:d0ram} ($\delta_0^{\operatorname{ram}}$) or Figure \ref{Fig:d0'} $(\delta_0'$) are similar.  These give rise to $FS_n+B_1$ (resp.~$FS_n+EE_1$) examples (see \S \ref{secFS+B}, resp.~\S \ref{secFS+EE}).  Lemmas \ref{AlemFSB} and \ref{AlemFSEE} show that the monodromy cone is contained in a cone in the PCD if and only if the cover is not a degeneration of an $FS_2$ or $FS_3$ cover.
\end{proof}


\section{Resolving the Prym map}\label{sectRes}
As discussed previously, in contrast to the case of the Torelli map for curves, the Prym map is not regular along (certain components of) the Friedman--Smith locus.  For geometric applications (eg.~the study of moduli of cubic threefolds) it is important to have a regular map.  Using the fact that the normal crossing compactifications and the toroidal compactifications have a toric structure at the boundary, it is always possible to refine the normal crossing compactification (by further toric blow-ups) to get a regular map. In this section we   resolve the Prym map up to codimension $4$. In the appendices  we have worked out some further  special cases in all genera for the perfect cone compactification; some special cases in all genera have also been considered by Alexeev and Vologodsky for the second Voronoi compactification (see \S \ref{secDR} and \cite{vologodskyres}).

\begin{teo} \label{teoresPM}
There is a closed locus $Z\subseteq  \bigcup_{n=2}^g \partial \overline {FS}_n\subseteq \overline{R}_{g+1}$ with $\operatorname{codim}_{\overline{R}_{g+1}} Z\ge 4$,
such that setting $U=\overline R_{g+1}\setminus Z$, the restriction to $U$  of the  Prym period map  $P_g^P:\overline R_{g+1}\dashrightarrow \bar A_g^P$ can be resolved in the following way:
\begin{enumerate}
\item The period map is regular on $U\setminus (\overline {FS}_2\cup \overline{FS}_3)$.
\item If $x\in U\cap \overline {FS}_2$, then \'etale locally there are either 1, 2, or 3 components of $\overline{FS}_2$ meeting at $x$.   If there are 1 or 2 components meeting, the period map is resolved by blowing up the union of the components.   If there are 3 components meeting at $x$, the period map is resolved by the toric morphism determined by Figure \ref{Fig:3comp}.
\item We have $U\cap \overline{FS}_3=FS_3$, and at a point $x\in FS_3$ the period map is resolved by blowing up the locus $FS_3$.
\end{enumerate}
In addition, for $g=2$ the period map $\overline R_3\dashrightarrow \bar A_2^P$ ($=\bar A_2^V,\bar A_2^C$)  is resolved simply by blowing up $\overline {FS}_{2}$, which is irreducible (globally and \'etale locally).
\end{teo}

\begin{rem}
One may take $Z$ in the theorem above so that for $n\ge 4$,  $U\cap\overline {FS}_n=FS_n$.  Then if in addition one blows-up along $FS_n$ for $n\ge 4$, this resolves the period map on $U$ to the second Voronoi compactification. 
\end{rem}

\begin{rem}
In the appendix, we provide explicit resolutions of the period map to $\bar A_g^P$ for many more types of degenerations of Friedman--Smith covers.    While these  still do not cover enough special cases to resolve the period map $\overline{R}_{g+1}\dashrightarrow \bar A_g^P$ for any $g\ge 3$, in principle, these computations could be carried out further to completely resolve the period map for low $g$.
\end{rem}

\begin{rem}\label{remresR3R4}
In part (2) of the theorem,   in the case where $3$ components of $\overline {FS}_2$ meet, we point out   that the birational modification  is not the blow-up of the union of the  $3$ components, nor is it obtained by blowing up the intersection of the $3$ components, followed by blowing up the strict transforms of the components (and neither of these birational modifications resolves the period map).
\end{rem}

\begin{proof}
First let us define the locus $Z$ in the statement of the theorem.
 Let $Z_2\subseteq \partial \overline{FS}_2$ be the locus of degenerations
whose dual graph is not obtained  by replacing a vertex in the dual graph of an $FS_2$ cover (see Figure \ref{Fig:dgFSG})  with one of the dual graphs in Figures
\ref{Fig:d0'}-\ref{Fig:digi}.   Let
$$Z=Z_2\cup \bigcup_{n=3}^g \partial \overline{FS}_n$$
Thus $\operatorname{codim} Z\ge 4$, and on $U$ the period map only fails to be regular along $U\cap \overline{FS}_2$ and $U\cap \overline{FS}_3=FS_3$.  From Remark \ref{remFSSVD}, at points of $FS_2$ and $FS_3$, the period map is resolved in a neighborhood by a blow-up of the $FS$ locus.
The proof now proceeds in a similar fashion to the proof of Theorem \ref{teoindPM}.  We enumerate the dual graphs obtained from covers in $U\cap \partial \overline{FS}_2$, and for each of them decompose the corresponding monodromy cone into cones in the PCD.
Recall that this provides a resolution of the period map in the following way.  Given an admissible double cover $\widetilde C\to C$, the miniversal space has snc boundary with components in bijection with edges of the  dual graph $\Gamma$ of  $C$.  Consequently, for each edge $e$ of $\Gamma$, there is a corresponding quadratic form obtained via the log of monodromy.  This induces  a map from the standard simplex with vertices indexed by the edges of $\Gamma$, to the closure of the monodromy cone.    Decomposing the monodromy cone into cones contained in the admissible cone decomposition, and then pulling back to the standard simplex, gives a decomposition of the standard simplex, which determines the minimal resolution of the period map in a neighborhood of the admissible cover $\widetilde C\to C$.

We now proceed to implement this, using the same enumeration as in the proof of Theorem \ref{teoindPM}
For the case of a $FS_2+W_1$ cover $\widetilde C\to C$, the dual graph $\Gamma$ has $3$ edges $e_1,e_2,f$.  The cone decomposition is given in \S \ref{secFS+W} and has the form

$$
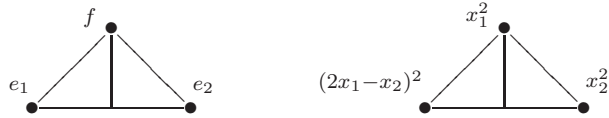
\xymatrix{
&*{\bullet} \ar@{-}[rd]|-{\SelectTips{cm}{}\object@{}}^>{e_2} \ar@{-}[ld]|-{\SelectTips{cm}{}\object@{}}_<{f}  _>{e_1} \ar@{-}[d]|-{\SelectTips{cm}{}\object@{}} & \\
*{\bullet} \ar@{-}[rr]|-{\SelectTips{cm}{}\object@{}}&*{} &*{\bullet}\\
}\ \  \ \  \ \ \ \ \
\xymatrix{
&*{\bullet} \ar@{-}[rd]|-{\SelectTips{cm}{}\object@{}}^>{x_2^2} \ar@{-}[ld]|-{\SelectTips{cm}{}\object@{}}_<{x_1^2}  _>{(2x_1-x_2)^2} \ar@{-}[d]|-{\SelectTips{cm}{}\object@{}} & \\
*{\bullet} \ar@{-}[rr]|-{\SelectTips{cm}{}\object@{}}&*{} &*{\bullet}\\
}
$$
where $e_1\mapsto (2x_1-x_2)^2$,  $e_2\mapsto x_2^2$, and $f\mapsto x_1^2$.
\'Etale locally, the divisors corresponding to $f$, $e_1$, and $e_2$ are all of type $\delta_0'$.  The intersection of the two copies of $\delta_0'$ corresponding to $e_1$ and $e_2$ is exactly the locus of Friedman--Smith covers, which   are of type $\overline{FS}_2$.  The decomposition above indicates that this locus is blown-up in the minimal resolution.

For the case of a $FS_{2+0}+FS_1$ cover,  the cone decomposition is given in \S \ref{secFS+FS} (see Figure \ref{figFS20FS1decomp}) and (in similar notation to the example above) has the form

$$
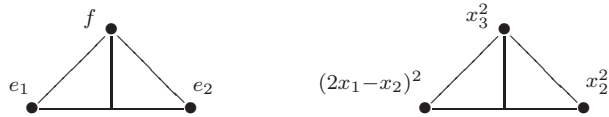
\xymatrix{
&*{\bullet} \ar@{-}[rd]|-{\SelectTips{cm}{}\object@{}}^>{e_2} \ar@{-}[ld]|-{\SelectTips{cm}{}\object@{}}_<{f}  _>{e_1} \ar@{-}[d]|-{\SelectTips{cm}{}\object@{}} & \\
*{\bullet} \ar@{-}[rr]|-{\SelectTips{cm}{}\object@{}}&*{} &*{\bullet}\\
}\ \  \ \  \ \ \ \ \
\xymatrix{
&*{\bullet} \ar@{-}[rd]|-{\SelectTips{cm}{}\object@{}}^>{x_2^2} \ar@{-}[ld]|-{\SelectTips{cm}{}\object@{}}_<{x_3^2}  _>{(2x_1-x_2)^2} \ar@{-}[d]|-{\SelectTips{cm}{}\object@{}} & \\
*{\bullet} \ar@{-}[rr]|-{\SelectTips{cm}{}\object@{}}&*{} &*{\bullet}\\
}
$$
where $e_1\mapsto (2x_1-x_2)^2$,  $e_2\mapsto x_2^2$, and $f\mapsto x_3^2$.
\'Etale locally, the divisors corresponding to $e_1$, and $e_2$ are of type $\delta_0'$.    The divisor corresponding to $f$ is of type $\delta_{1,1}$ (the $\overline {FS}_1$ locus). The intersection of the two copies of $\delta_0'$ corresponding to $e_1$ and $e_2$ is exactly the locus of Friedman--Smith covers of type $\overline{FS}_2$.  The decomposition above indicates that this locus is blown-up in the minimal resolution.

For the case of a $FS_{1+1}+FS_1$ cover it is shown in \S \ref{secFS+FS} (see also \S \ref{secDR} and \cite{vologodskyres}),   that the  cone decomposition is (in similar notation)
\begin{equation}\label{Fig:3comp}
\xymatrix@C=.4cm@R=.5cm{
&&*{\bullet} \ar@{-}[rrdd]|-{\SelectTips{cm}{}\object@{}}^>{e_2} ^<{f} \ar@{-}[lldd]|-{\SelectTips{cm}{}\object@{}}_>{e_1}&&\\
&*{} \ar@{-}[rd] \ar@{-}[rr]&&*{} \ar@{-}[ld] &\\
*{\bullet}  \ar@{-}[rrrr] && *{}&&*{\bullet}
} \ \ \ \ \  \ \ \  \ \ \  \
\xymatrix@C=.4cm@R=.5cm{
&&*{\bullet} \ar@{-}[rrdd]|-{\SelectTips{cm}{}\object@{}}^>{x_2^2} ^<{(-x_2+2x_3)^2} \ar@{-}[lldd]|-{\SelectTips{cm}{}\object@{}}_>{(2x_1-x_2)^2}&&\\
&*{} \ar@{-}[rd] \ar@{-}[rr]&&*{} \ar@{-}[ld] &\\
*{\bullet}  \ar@{-}[rrrr] && *{}&&*{\bullet}
}
\end{equation}
\'Etale locally, the divisors corresponding to $f$, $e_1$ and $e_2$ are of type $\delta_0'$.  In this case, each of the $3$ pairwise intersections of these divisors  is an $\overline {FS}_2$ locus.    The associated birational modification that resolves the period map is an isomorphism away from this locus, and introduces $3$ exceptional divisors.  The corresponding birational modification of $\mathbb A^3_{\mathbb C}$ has as  fiber over the origin equal to $3$ copies of $\mathbb P^1_{\mathbb C}$ attached at a point.

For the $FS_{2+0}+\delta_1$ example, from the analysis in section  \S \ref{secFS+d} we see that the cone decomposition is
$$
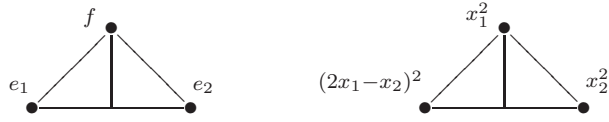
\xymatrix{
&*{\bullet} \ar@{-}[rd]|-{\SelectTips{cm}{}\object@{}}^>{e_2} \ar@{-}[ld]|-{\SelectTips{cm}{}\object@{}}_<{f}  _>{e_1} \ar@{-}[d]|-{\SelectTips{cm}{}\object@{}} & \\
*{\bullet} \ar@{-}[rr]|-{\SelectTips{cm}{}\object@{}}&*{} &*{\bullet}\\
}\ \  \ \  \ \ \ \ \
\xymatrix{
&*{\bullet} \ar@{-}[rd]|-{\SelectTips{cm}{}\object@{}}^>{x_2^2} \ar@{-}[ld]|-{\SelectTips{cm}{}\object@{}}_<{x_1^2}  _>{(2x_1-x_2)^2} \ar@{-}[d]|-{\SelectTips{cm}{}\object@{}} & \\
*{\bullet} \ar@{-}[rr]|-{\SelectTips{cm}{}\object@{}}&*{} &*{\bullet}\\
}
$$
The divisors corresponding to $e_1$, $e_2$ and $f$ are of type $\delta_0'$.  The Friedman--Smith locus ($\overline {FS}_2$)  is given as the intersection of the two divisors corresponding to $e_1$ and $e_2$.  The decomposition tells us that in the neighborhood of such a point, the minimal resolution is the blow-up of the $\overline {FS}_2$ locus.

For the $FS_{1+1}+\delta_1$ example, in \S \ref{secFS+d} we see that
the cone decomposition is given as
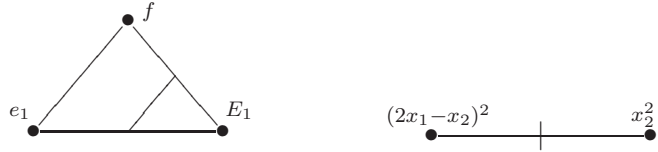
\begin{equation*}
\xymatrix@C=.4cm@R=.5cm{
&&*{\bullet} \ar@{-}[rrdd]|-{\SelectTips{cm}{}\object@{}}^>{E_1} ^<{f} \ar@{-}[lldd]|-{\SelectTips{cm}{}\object@{}}_>{e_1}&&\\
&*{} &&*{}\ar@{-}[ld]   &\\
*{\bullet}  \ar@{-}[rrrr] && *{}&&*{\bullet}
} \ \ \ \ \  \ \ \  \ \ \  \
\xymatrix@C=.5cm@R=.5cm{
&&&&\\
&&\\
*{\bullet}  \ar@{-}[rrrr]^>{x_2^2} ^<{(2x_1-x_2)^2}&& *{|}&&*{\bullet}
}
\end{equation*}
where $e_1\mapsto (2x_1-x_2)^2$, $E_1\mapsto x_2^2$ and $f\mapsto (2x_1-x_2)^2$.
The divisors corresponding to $e_1$, $E_1$ and $f$ are of type $\delta_0'$.    In this case, the two components of the Friedman--Smith locus correspond to the intersection of the divisor corresponding to $e_1$ with the divisor corresponding to $E_1$, and also to the divisor corresponding to $f$ intersecting the divisor corresponding to $E_1$.  These two loci are both of type $\overline {FS}_2$.  The period map is resolved by blowing up the union of these loci.

For the $FS_{0+2}+\delta_1$ example, from the analysis in section  \S \ref{secFS+d} we see that the cone decomposition is
$$
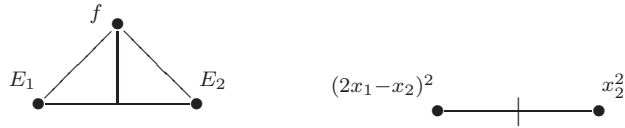
\xymatrix{
&*{\bullet} \ar@{-}[rd]|-{\SelectTips{cm}{}\object@{}}^>{E_2} \ar@{-}[ld]|-{\SelectTips{cm}{}\object@{}}_<{f}  _>{E_1} \ar@{-}[d]|-{\SelectTips{cm}{}\object@{}} & \\
*{\bullet} \ar@{-}[rr]|-{\SelectTips{cm}{}\object@{}}&*{} &*{\bullet}\\
}\ \  \ \  \ \ \ \ \
\xymatrix{
& \ar@{}[rd]|-{\SelectTips{cm}{}\object@{}}^>{x_2^2} \ar@{}[ld]|-{\SelectTips{cm}{}\object@{}}_<{}  _>{(2x_1-x_2)^2} \ar@{}[d]|-{\SelectTips{cm}{}\object@{}} & \\
*{\bullet} \ar@{-}[rr]|-{\SelectTips{cm}{}\object@{}}&*{|} &*{\bullet}\\
}
$$
where $E_1\mapsto (2x_1-x_2)^2$, $E_2\mapsto x_2^2$, and $f\mapsto 0$.
The divisors corresponding to $E_1$, $E_2$  are of type   $\delta_0'$, while the divisor corresponding to $f$ is of type $\delta_1$.  The Friedman--Smith locus ($\overline {FS}_2$)  is given as the intersection of the two divisors corresponding to $E_1$ and $E_2$.  The decomposition tells us that in the neighborhood of such a point, the minimal resolution is the blow-up of the $\overline {FS}_2$ locus.

For the $FS_{2}+B_1$ example, from the analysis in section  \S \ref{secFS+B} we see that the cone decomposition is
$$
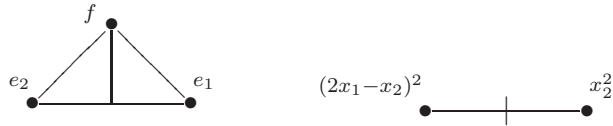
\xymatrix{
&*{\bullet} \ar@{-}[rd]|-{\SelectTips{cm}{}\object@{}}^>{e_1} \ar@{-}[ld]|-{\SelectTips{cm}{}\object@{}}_<{f}  _>{e_2} \ar@{-}[d]|-{\SelectTips{cm}{}\object@{}} & \\
*{\bullet} \ar@{-}[rr]|-{\SelectTips{cm}{}\object@{}}&*{} &*{\bullet}\\
}\ \  \ \  \ \ \ \ \
\xymatrix{
& \ar@{}[rd]|-{\SelectTips{cm}{}\object@{}}^>{x_2^2} \ar@{}[ld]|-{\SelectTips{cm}{}\object@{}}_<{}  _>{(2x_1-x_2)^2} \ar@{}[d]|-{\SelectTips{cm}{}\object@{}} & \\
*{\bullet} \ar@{-}[rr]|-{\SelectTips{cm}{}\object@{}}&*{|} &*{\bullet}\\
}
$$
where $e_1\mapsto (2x_1-x_2)^2$, $e_2\mapsto x_2^2$, and $f\mapsto 0$.
The divisors corresponding to $e_1$, $e_2$  are of type   $\delta_0'$, while the divisor corresponding to $f$ is of type $\delta_0^{\operatorname{ram}}$.  The Friedman--Smith locus ($\overline {FS}_2$)  is given as the intersection of the two divisors corresponding to $e_1$ and $e_2$.  The decomposition tells us that in the neighborhood of such a point, the minimal resolution is the blow-up of the $\overline {FS}_2$ locus.

For the $FS_{2}+EE_1$ example, from the analysis in section  \S \ref{secFS+EE} we see that the cone decomposition is
$$
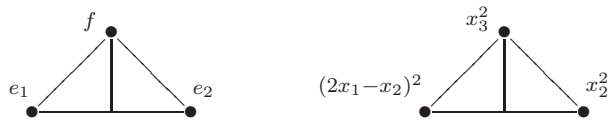
\xymatrix{
&*{\bullet} \ar@{-}[rd]|-{\SelectTips{cm}{}\object@{}}^>{e_2} \ar@{-}[ld]|-{\SelectTips{cm}{}\object@{}}_<{f}  _>{e_1} \ar@{-}[d]|-{\SelectTips{cm}{}\object@{}} & \\
*{\bullet} \ar@{-}[rr]|-{\SelectTips{cm}{}\object@{}}&*{} &*{\bullet}\\
}\ \  \ \  \ \ \ \ \
\xymatrix{
&*{\bullet} \ar@{-}[rd]|-{\SelectTips{cm}{}\object@{}}^>{x_2^2} \ar@{-}[ld]|-{\SelectTips{cm}{}\object@{}}_<{x_3^2}  _>{(2x_1-x_2)^2} \ar@{-}[d]|-{\SelectTips{cm}{}\object@{}} & \\
*{\bullet} \ar@{-}[rr]|-{\SelectTips{cm}{}\object@{}}&*{} &*{\bullet}\\
}
$$
The divisors corresponding to $e_1$, $e_2$ and $f$ are of type $\delta_0'$.  The Friedman--Smith locus ($\overline {FS}_2$)  is given as the intersection of the two divisors corresponding to $e_1$ and $e_2$.  The decomposition tells us that in the neighborhood of such a point, the minimal resolution is the blow-up of the $\overline {FS}_2$ locus.

The proof that the period map $\overline R_3\dashrightarrow \bar A_2$, is resolved by blowing up $\overline {FS}_{2}$ is similar.  There are more cases to consider, but in each case, the associated combinatorial data is a simplex that is star-subdivided along the edge corresponding to two the Friedman--Smith edges.
\end{proof}

\section{Fibers of the resolution} \label{secsubFibers}
We now consider the question of describing the fibers of the resolution.  We expect that with \cite{donagi}  and the techniques we describe here, it should be possible to give compete descriptions of the fibers over certain strata in low genus.  We will pursue this elsewhere; here we limit ourselves to the following. Given a point $x\in \bar A_g$, in a given stratum, describe loci in the resolution of $\overline R_{g+1}$ that  map to the same stratum.  This can be rephrased in terms of $1$-parameter families, which is what we actually consider.    Moreover, since    $\bar A_g^V$, $\bar A_g^P$ and $\bar A_g^C$ coincide outside $\overline{\beta}_{4}$ (torus rank $4$ or more),  for $\beta_0\cup\beta_1\cup \beta_2 \cup \beta_3$
we can work with any one of them, eg.~we can adopt the language of the second Voronoi compactification as we shall do below.

\subsection{Degeneration data for Pryms}
Limits of one parameter families of ppav are determined by degeneration data, which in turn determine the limit point in the toroidal compactification  (see \S\ref{secModStAbVar}).  Here we recall from \cite{abh} the degeneration data for Pryms.    We begin with an admissible cover $\widetilde C\to C$, and a $1$-parameter deformation associated to a map $\psi:S\to \operatorname{Def}_{\widetilde C/C}$, from the  unit disk to the base of a mini-versal deformation.   Let $\Psi:S\to \bar A_{g}^\Sigma $ be the composition of $\psi$ with the rational map $\operatorname{Def}_{\widetilde C/C}\dashrightarrow \bar A_{g}^\Sigma$.

First let us consider the degeneration data for the  Jacobian of the covering curve $\widetilde C$.  In this case, we saw in \eqref{eqnExtJac} that the generalized Jacobian $J\widetilde C$ corresponds to a morphism $\tilde c:H_1(\widetilde \Gamma,\mathbb Z)\to \widehat {J\widetilde N}$, where $\widetilde N$ is the normalization of $\widetilde C$.
We recall the morphism $\tilde c$.
A node on ${\widetilde  C}$ corresponds to an edge $e$ in $C_1(\widetilde {\Gamma},\mathbb Z)={\oplus} \ZZ {\tilde e_j}$ going from a vertex $\tilde v^+$ to a vertex $\tilde v^{-}$.
Let $\tilde Q_+(e)$ be the point corresponding to $\tilde v^+$ in ${\widetilde N}$ and similarly with $\tilde v^-$ (if $\tilde v^+=\tilde v^-$ then it does not matter which of the two points above the double point
is  $\tilde Q_+(e)$ and which one is $\tilde Q_-(e)$).
The map $\tilde c$ is defined by restricting the map
$$
\tilde c: C_1({\tilde{\Gamma}},\mathbb Z) \to J{\widetilde N}_0, \ \ \ \
e \mapsto {\mathcal O}(Q_+(e)) \otimes {\mathcal O}(Q_-(e))^{-1}
$$
to the sub-lattice $H_1(\widetilde \Gamma,\mathbb Z)$.
The isomorphism $\phi:H^1(\widetilde \Gamma,\mathbb Z)\to H_1(\widetilde \Gamma,\mathbb Z)$ is the canonical isomorphism, and $\hat{\tilde c}:H^1(\widetilde \Gamma,\mathbb Z)\to J\widetilde N$ is defined as $\lambda^{-1}\circ c \circ \phi$, where $\lambda$ is the canonical principal polarization.   The biholomorphism $\tilde \tau$ is related to the Deligne symbol, and the quadratic form $\widetilde B$ is given by the valuation of $\tilde \tau$ or alternatively the log of monodromy computed in Proposition \ref{projmoncone}.

Now let us describe the degeneration data for the Prym.  As discussed in \eqref{eqnExtPrym}, the generalized (open) Prym  corresponds to a morphism $ c^-:H_1(\widetilde \Gamma,\mathbb Z)^{[-]}\to \widehat A$, where $A$ is a finite  cover  of  the abelian variety
$$P_N:=\ker(\operatorname{Nm}:J\widetilde N \to JN)_0$$ (see \cite[Prop.~1.5]{abh}).
The map $c^-$ is given by the  commutative diagram
$$
\begin{CD}
 H_1(\widetilde \Gamma,\mathbb Z) @>{} >> H_1(\widetilde \Gamma,\mathbb Z)^{[-]}@.\\
 @V \tilde c VV @V c^- VV\\
\widehat {J\widetilde N} @>{}>> \widehat A.
\end{CD}
$$
The biholomorphism $\tau^-$  is given by the ``restriction'' of the bihomomorphism $\tilde \tau$  (\cite[\S 3.2, esp.~\S 2.2, p.93]{abh}).
 The bilinear form $B^-$
is again given by the valuation of $\tau^-$ or equivalently the log of monodromy computed in Proposition \ref{propmoncone}.

In summary, $\Psi(0) \in \beta_i$ if and only if  $\operatorname{rank} H_1(\widetilde \Gamma,\mathbb Z)=i$, and $\Psi(0)\in \beta(\sigma)$ if and only if $B^-\in \sigma$ where
$\sigma$ is the minimal cone with this property.
The remaining modulus for $\Psi(0)$ is  determined by the remaining degeneration data.

Note that by definition, $B^-\in \sigma(\widetilde C/C)$, the monodromy cone.  If $\sigma(\widetilde C/C)$ is contained in a cone in $\Sigma$, then the Prym map extends in a neighborhood of $\widetilde C\to C$, and so $\Psi(0)$ depends only on $\widetilde C\to C$, and not on the $1$-parameter family.   Otherwise, a decomposition of $\sigma(\widetilde C/C)$ into cones in $\Sigma$ shows the dependence of even the combinatorial data on the $1$-parameter family.  More precisely, in the notation of Remark \ref{remPrymMonCone}, $B^-=\sum_{e\in \Gamma}\operatorname{ord} \psi^*(z_e)(\tilde e^\vee-\iota\tilde e^\vee)^2$,   and
the order of vanishing determines the sub cone of $\sigma(\widetilde C/C)$ in which the quadratic form $B^-$ lies.    We shall now illustrate the above discussion with several examples.

As a final note,  the combinatorics of rank $1$ quadratic forms on a lattice are best considered by using squares of primitive rank $1$ linear forms.  This is the convention we use in the appendices.  To match those descriptions to the ones in this section,  it works best to describe the quadratic form $B^-$ as
$$
B^-=\sum_{e\in \Gamma}\alpha_e\ell_e^2
$$
where  $\ell_e$ is defined in \eqref{EQNdefle} and
$$
 \alpha_e:=
 \left\{
 \begin{array}{ll}
 \operatorname{ord}\psi^*(z_e) & \text{if } \  \iota \tilde e^\vee \ne -\tilde e^\vee .\\
4\operatorname{ord}\psi^*(z_e)& \text{if }\  \iota\tilde e^\vee=-\tilde e^\vee,\\
 \end{array}
 \right.
$$

\subsection{The Friedman--Smith loci $FS_n$}   In this section we consider the image of the strict transforms of the Friedman--Smith loci. That is, given a $1$-parameter family of admissible covers, degenerating to a Friedman--Smith cover in $FS_n$, we want to describe the associated point in $\bar A_g^V$.

As we have already seen the Friedman--Smith loci $FS_{n}$ consist of several components. These can be enumerated as follows: the curve $C$ is
reducible, more precisely $C=C_1 \cup C_2$ where $C_1$ and $C_2$ intersect in $n$ points. The curves $C_i$ are smooth and  irreducible. If $g_i=g(C_i)$,
then $g_1 + g_2 + n-1=g+1$. The components  of $FS_{n}$ then correspond to the different possibilities for $g_1 \geq g_2 > 0$.   From \S \ref{sectPrym}, \S \ref{secFSexamples} one can see that $\operatorname{rank}H_1(\widetilde \Gamma,\mathbb Z)^{[-]}=n$ and  $A=P_N=P_{N_1}\times P_{N_2}$ (see also \cite{abh}). This determines the image in $A_{g-1}^*$.
For the remaining extension data, we will focus on the quadratic form $B^-$, starting with a general point on an irreducible component of $FS_2$.

\subsubsection{The $FS_2$ locus}
Here we are in the torus rank $2$ case, i.e.~$i=2$. Thus we are no longer in Mumford's partial compactification, but we are still in the range
where all known toroidal compactifications coincide, in particular also the   second Voronoi compactification and the perfect cone compactification.   In the notation of \S \ref{secFScd}, where the decomposition of the monodromy cone is established,  the form $B^-$ is given by
$
\alpha_1(\tilde e_1^\vee-\iota \tilde e_1^\vee )^2+\alpha_2(\tilde e_2^\vee-\iota \tilde e_2^\vee )^2
$
with $\alpha_i=\operatorname{ord}\psi^*(z_{e_i})$, and the monodromy  cone decomposes as $\alpha_1<\alpha_2$, $\alpha_1=\alpha_2$, $\alpha_1>\alpha_2$.
In the general case, $\alpha_1=\alpha_2$.
As explained in Remark \ref{remFS2Del}, the associated Delaunay decomposition of $\RR^2$ is that of squares and the corresponding cone  is equivalent to  the standard cone $\sigma_{1+1}=\langle x_1^2,x_2^2 \rangle$.
The degenerate Prym is in this case  a $\PP^1 \times \PP^1$ bundle over $A$ with ``opposite'' coordinate lines $\{0\} \times \PP^1$ and $\{\infty\} \times \PP^1$  as well
as  $\PP^1 \times \{0\}$ and $\PP^1 \times \{\infty\}$ identified with a shift over $A$. This shift is
determined by $c^-$. There is a further parameter $b\in \CC^*$ which describes the gluing of the lines which are identified. This parameter is given by the
bihomomorphism $\tau$ and varies with the chosen $1$-parameter family (even if $B^-$ does not change). In fact this is the parameter in the fibers of the $\PP^1$-bundle
given by blowing up the general point of an irreducible component of $FS_2$.

We can also choose $1$-parameter families with $\alpha_1>\alpha_2$ (the case $\alpha_1<\alpha_2$ can be obtained by symmetry).
Now, as explained in Remark \ref{remFS2Del}, the Delaunay decomposition changes: every square breaks up into two triangles
and the corresponding cone is equivalent to $\sigma_{K_3} = \langle x_1^2, x_2^2, (x_1 + x_2)^2 \rangle$.
The degenerate Prym is the union of two $\PP^2$-bundles over $A$
with their coordinate lines identified appropriately again with a shift over $A$ which is determined by $c^-$. In this case the torus bundle ${\mathcal T}(\sigma)$ has rank $0$
and the bihomomorphism  $\tau$ is trivial. We shall see below that points in $\beta(\sigma_{1+1})$ can also arise from other degenerations.

\subsubsection{The $FS_3$ locus}
We shall now move on to general points on $FS_3$. As before the crucial point is the form $B^-$.
In the notation of \S \ref{secFScd}, where the decomposition of the monodromy cone is established,  the form $B^-$ is given by
$
\sum_{i=1}^3\alpha_i(\tilde e_i^\vee-\iota \tilde e_i^\vee )^2
$
with $\alpha_i=\operatorname{ord}\psi^*(z_{e_i})$, and the monodromy  cone decomposes as $\alpha_1=\alpha_2=\alpha_3$, $\alpha_1<\alpha_2=\alpha_3$,  $\alpha_1<\alpha_2 < \alpha_3$,
together with  all permuations.    Here we discuss the general case, $\alpha_1=\alpha_2=\alpha_3$.
In Remark \ref{remFS3Del} is is shown that  the associated cone is equivalent to  $\sigma_{C_4}=\langle x_1^2,  x_2^2, x_3^2, (x_1 + x_2 + x_3)^2 \rangle$. The associated Delaunay decomposition of $\RR^3$ consists of $1$ octahedron and $2$ tetrahedra. The homomorphism $c^-$ again defines
shift parameters and $\tau$ defines gluing parameters, see \cite[Section 7.2]{ghdegenerations}. The latter depends on the $1$-parameter family.   The $1$-parameter families with different orders of vanishing will result in $B^-$ lying in one of the other cones in the decomposition of the monodromy cone (see Remark \ref{remFS3Del}).

\begin{rem}
The interesting point to note here is that this is a codimension $1$ stratum
in $\beta_2^{\Sigma}$ and hence the exceptional divisors over the general points of the irreducible components of $FS_3$ do {\em{not}} map dominantly to $\beta_2^{\Sigma}$, even for small $g$.
Thus the question remains to find an example which maps to the (unique) maximal stratum in $\beta_2^{\Sigma}$, namely the stratum
$\beta(\sigma_{1+1+1})$ where $\sigma_{1+1+1}= \langle x_1^2, x_2^2, x_3^2\rangle$. Indeed this is not difficult to find. We can take an
elementary
\'etale  covering with $6$ nodes (discussed below).
\end{rem}

\begin{rem}
More generally, for an $FS_n$ locus with $n\ge 2$,
if both $g_1, g_2 > 1$, then
the abelian variety $A$ is reducible.   Hence the strict transforms of the corresponding Friedman--Smith loci cannot map dominantly onto $\beta_i$,  even when mapping to the relevant stratum $A_{g-n}$ of  $A_g^{*}$.
For the remaining stratum $g_2=1$, it is possible that the map from this component of $FS_2$ to $\beta_2$ is dominant for $g\leq 5$.
This is clear for $g=2$, but in general needs some discussion of the
continuous parameters. The main issue is whether the projection under $q$ (see \S \ref{secModStAbVar}) maps surjectively to  ${\mathcal X}^{\times 2}$.
\end{rem}

\subsection{Elementary \'etale examples}
In this section we show that the elementary \'etale examples (see \S \ref{secEEE}) map to the (unique) maximal stratum in $\beta_2^{\Sigma}$, namely the stratum
$\beta(\sigma_{1+1+1})$ where $\sigma_{1+1+1}= \langle x_1^2, x_2^2, x_3^2\rangle$.
  Indeed it is  shown in \S \ref{secEEE} that the monodromy cone of an elementary \'etale example with $2n$ nodes is of type $\sigma_{1+\cdots+1}$.
 The associated degenerate abelian varieties
are $(\PP^1)^n$-bundles over abelian varieties with ``opposite sides'' glued with a shift. The shift is given by $c^-$, the gluing by $\tau$.
In particular points in $\beta(\sigma_{1 + 1})$ can arise not only from $FS_2$ covers but also from elementary \'etale covers with $4$ nodes.

\subsection{The Wirtinger locus $W_n$}
We conclude this discussion with the Wirtinger examples $W_n$ (see \S \ref{secWE}).  In this case $\widetilde C$  has two components which are
exchanged by the involution and $2n$ nodes which are pairwise interchanged.
In Remark \ref{remWDel} it is established that the
$W_n$ monodromy cone is $\mathbb R_{\ge 0}\langle x_1^2,(x_1-x_2)^2, \ldots, (x_{n-2}-x_{n-1})^2,x_{n-1}^2\rangle$.  For $n=3$, the
Delaunay decomposition of $\mathbb R^2$ is the tiling by $2$-simplices obtained by slicing the standard square into $2$ triangles.
The associated degenerate abelian variety is a union of $2$ copies of  $\PP^2$-bundles, corresponding to the slicing of the square into $2$ triangles. As always the gluing
is determined by $c^-$ the biholomorphism $\tau$ is trivial here. We also see that points in $\beta(\sigma_{K_3})$ can arise from both the blow-up of $FS_2$ as well
as from Wirtinger examples $W_3$.
For $n=4$ the toric part of the semi-abelic variety corresponds to the dicing of a cube into an octahedron and two tetrahedra, and thus consists of a complete intersection
$F(2,2)$ of two quadrics in $\PP^5$ and two copies of $\PP^3$. We take an opportunity to correct here an error in \cite[Example 5.2.2]{abh} where it was claimed that the
toric parts are all projective spaces. We note that the general $FS_3$ degenerations are mapped to the same stratum.


\appendix

\section{Combinatorics of Friedman--Smith monodromy cones} \label{seccombinatorics}

In this appendix we establish various combinatorial properties of the monodromy cones of Friedman--Smith covers.    Our starting point will be the computation in \S \ref{secFSMonCone}, which culminated in \eqref{eqnFSMCo} giving a description of the monodromy cone in matrix form.
We first note that changing basis, one may take the Friedman--Smith monodromy cone for a cover of type $FS_n$, $n\ge 2$,  to be generated by the quadratic forms that are given by the squares of the linear forms determined by the rows of the following matrix:
 \begin{equation}\label{eqnFSMC}
FS_n \ \ \ \  \left(\begin{smallmatrix}
2&-1&-1&&&&-1\\
0&1&0&&&&0\\
0&0&1&0&&&\\
&&&\ddots&\ddots&&\\
&&&0&1&0&0\\
&&&&0&1&0\\
&&&&&0&1\\
\end{smallmatrix}\right)
\end{equation}
This matrix can be obtained from \eqref{eqnFSMCo} by  integral column operations, and by replacing rows  with their negatives.

\subsection{Friedman--Smith cones, simplicial cones, and basic cones}

\begin{lem}\label{lemBasicCone}
The Friedman--Smith cones are basic for  $n \geq 3$, $n=1$, and simplicial but not basic for $n=2$.
\end{lem}

\begin{proof}  For $n=1$ the assertion is clear.  For $n\ge 2$
we use the description in \eqref{eqnFSMC}. For $n=2$ the forms $(2x_1-x_2)^2$ and $x_2^2$ are independent, hence the cone is simplicial. On the other hand these
forms are not part of a basis since they do not span a primitive sublattice: the difference  $(2x_1-x_2)^ 2- x_2^2$ is not primitive.
For $n \geq 3$ the  situation is different. Adding the forms $x_1^2, (x_1-x_i)^2, i= 2, \ldots, n$ and $(x_i-x_j)^2, 2 \leq i< j \leq n, (i,j)\neq (2,3)$
one obtains a basis of $\Sym^2(\ZZ^n)$ and thus the cone is basic.
\end{proof}

\subsection{Friedman--Smith monodromy cones and   second Voronoi cones}

\begin{lem}\label{lemFSMAT}
The Friedman--Smith monodromy cone is matroidal if and only if $n=1$.  Every proper face of a Friedman--Smith cone is matroidal.
\end{lem}

\begin{proof}
The first statement  follows from \S \ref{secFSMonCone} and  \eqref{eqnFSMC} (the determinant is $2$, if $n> 1$).   The second
follows since any $n-1$ rows of the matrix \eqref{eqnFSMC} can be extended to a $\ZZ$-basis of $\ZZ^n$ (this also shows that every face of a
Friedman--Smith monodromy cone is contained in a cone of type $A$, see \eqref{Acone}).
\end{proof}

\subsection{Friedman--Smith monodromy cones and perfect cones}

\begin{pro}\label{proFS=PC}
The Friedman--Smith monodromy cone is contained in a cone in the  perfect cone decomposition  if and only if $n\ne 2,3$.   In fact, it \emph{is} a cone in the perfect cone decomposition if and only if $n\ne 2,3,4$.
\end{pro}

\begin{rem}
In Appendix \ref{secappMDS},  Sikiri\'c shows that   the Friedman--Smith monodromy cone is contained in a cone in the  central cone decomposition  if and only if $n\ne 2,3$.  His proof also shows that the Friedman--Smith monodromy cone is contained in a cone in the  perfect cone decomposition  if and only if $n\ne 2,3$.  The proof we give below establishes, for the PCD, the stronger statement that the Friedman--Smith monodromy cone
 \emph{is} a cone in the PCD if and only if $n\ne 2,3,4$. \end{rem}

\begin{proof}
Let $\Lambda=\ZZ^n$ and fix the standard basis $e_1,\ldots,e_n$.   Let $e_1^\vee,\ldots,e_n^\vee$ be the standard dual basis  for $\Lambda^\vee$. First let us show  that the  Friedman--Smith cones, the cones generated by the quadratic forms $(2e_1^\vee-e_2^\vee-\ldots-e_n^\vee)^2, \ (e_2^\vee)^2,\   \ldots,\ (e_n^\vee)^2$, are contained in
a cone in the perfect cone decomposition if and only if $n\ne 2,3$.

For $n=1$, the cone is clearly matroidal,  so it is a cone in the perfect cone decomposition (eg.~Remark \ref{remMV}).  For lattices of rank $2,3$, every cone in the perfect cone decomposition is also a matroidal cone (see Remark \ref{remlowdim}).  Consequently,  the fact that the Friedman--Smith cone is generated by rank $1$ quadrics and is not matroidal (Lemma \ref{lemFSMAT}) implies it is not contained in a matroidal cone (Lemma \ref{lemsecvor}), and hence not contained in a cone in the perfect cone decomposition.

For $n\ge 4$, consider the metric on $\mathbb R^n$ induced by the matrix
\begin{equation}\label{eqnQ}
Q=\left(
\begin{smallmatrix}
\frac{n}{4}&\frac{1}{2}&\ldots&\frac{1}{2}&\frac{1}{2}\\
\frac{1}{2}&1&0&\ldots&0\\
\vdots &&\ddots&&\\
\vdots &&&\ddots&\\
\frac{1}{2}&0&\ldots&0&1\\
\end{smallmatrix}
\right)
\end{equation}
The matrix  has determinant $1/4$, and so is clearly positive definite (considering leading principal minors starting from say the bottom right corner).   The metric takes value $1$ on $(2e_1^\vee-e_2^\vee-\cdots-e_n^\vee), e_2^\vee, \ldots,e_n^\vee$.  Thus the Friedman--Smith cone is contained in a cone in the perfect cone decomposition for every $n\ge 4$, provided we can show that $Q$ takes values at least $1$ on all non-zero elements of $\Lambda^\vee$.
This is immediate by inspection for $n\equiv 0\pmod 4$; we will prove this in general by describing the cones determined by $Q$ in more detail.

To this end, consider  another copy of $\ZZ^n$ with basis $f_1,\ldots,f_n$.  We will also use the standard basis $f_1^\vee,\ldots,f_n^\vee$ for $(\ZZ^n)^\vee$.  Inside of $(\RR^n)^\vee$ consider the lattice $\mathbb L$ generated by $\ell=\frac{1}{2}f_1^\vee+\ldots +\frac{1}{2}f_n^\vee$, and $ f_2^\vee, \ldots, f_n^\vee$.    We have $(\ZZ^n)^\vee\subseteq \mathbb L$, since
$$
f_1^\vee=2\ell-f_2^\vee-\cdots-f_n^\vee.
$$
Let $Q_0$  be the standard quadratic form on $(\RR^n)^\vee$:
$$
Q_0=f_1^2+\ldots+f_n^2.
$$
For simplicity, let us work momentarily in coordinates.  It is clear that any vector $$(a_1,\ldots,a_n)\in (\RR^n)^\vee$$ will have length greater than $1$ if any of the coefficients $a_i$ has magnitude greater than $1$, or any one coefficient has magnitude  $1$ and any other coefficient is non-zero.  We can then easily enumerate the non-zero elements of $\mathbb L$ that do not satisfy these conditions:
$$
(\pm 1,0,\ldots,0), (0,\pm 1,0,\ldots,0), \ldots, (0,\ldots, 0,\pm1), (\pm \frac{1}{2},\ldots,\pm \frac{1}{2}).
$$
Since
$$
Q_0(\pm \frac{1}{2},\ldots,\pm \frac{1}{2})=\frac{n}{4}
$$
we see that the non-zero elements of $\mathbb L$ of minimal length are,
\begin{enumerate}
\item $ (\pm \frac{1}{2},\ldots,\pm \frac{1}{2})$, $n<4$,
\item $(\pm 1,0,\ldots,0), (0,\pm 1,0,\ldots,0), \ldots, (0,\ldots, 0,\pm1), (\pm \frac{1}{2},\ldots,\pm \frac{1}{2})$, $n=4$,
\item $(\pm 1,0,\ldots,0), (0,\pm 1,0,\ldots,0), \ldots, (0,\ldots, 0,\pm1)$, $n>4$.
\end{enumerate}
Now consider the isomorphism $\Lambda^\vee \to \mathbb L$ given by $e_1^\vee \mapsto \ell, e_2^\vee \mapsto f_2^\vee,\ldots, e_n^\vee\mapsto f_n^\vee$.   The pull-back of $Q_0$ under this isomorphism is $Q$, and the pull-back of the basis elements $f_1^\vee, \ldots,f_n^\vee$ are the elements $2e_1^\vee-e_2^\vee-\cdots-e_n^\vee, e_2^\vee,\ldots, e_n^\vee$.  Thus for $n>4$, the Friedman--Smith cone is exactly the cone in the perfect cone decomposition determined by $Q$.  For $n=4$, the cone determined by $Q$ has $12$ extremal rays, and is in fact a type $D$ cone, see \eqref{Acone}.  One can show by enumerating the faces of a type $D$ cone in dimension $4$ that the Friedman--Smith cone is not a face (and thus is not a cone in the perfect cone decomposition).
\end{proof}

\subsection{Cone decompositions  for $FS_2$ and $FS_3$} \label{secFScd}

\begin{lem}\label{lemFSPCD}
For $n=2,3$, each cone in the barycentric subdivision of the Friedman--Smith cone is contained in a matroidal cone (and thus also in a cone in the perfect cone decomposition and the central cone decomposition).
\end{lem}
\begin{rem}
This lemma follows from the fact (see Remark \ref{remFSSVD}) that each cone in the barycentric subdivision of the Friedman--Smith cone is contained in a cone in the   second Voronoi decomposition  (since for $n=2,3$ the   second Voronoi cones are matroidal).  However, we will want to know the exact cones containing the cones in the barycentric subdivision for later computations.
\end{rem}
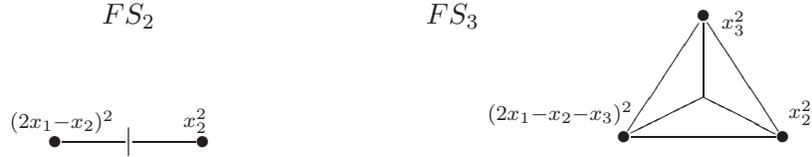
\begin{figure}[htb]
\begin{equation*}
\xymatrix@C=.4cm@R=1.2cm{
&FS_2&\\
*{\bullet}\ar@{-}[r]|-{\SelectTips{cm}{}\object@{}}^<{(2x_1-x_2)^2}& *{|} \ar@{-}[r]|-{\SelectTips{cm}{}\object@{}}^>{x_2^2} &*{\bullet}\\
}
\ \  \hskip 1 in FS_3
\xymatrix@C=.3cm@R=.3cm{
&&*{\bullet} \ar@{-}[dd]|-{\SelectTips{cm}{}\object@{}}^<{\ x_3^2}\ar@{-}[rrddd] \ar@{-}[llddd]&&\\
&&&&\\
&&*{}\ar@{-}[rrd]|-{\SelectTips{cm}{}\object@{}}^>{\ \ x_2^2} \ar@{-}[lld]|-{\SelectTips{cm}{}\object@{}}_>{(2x_1-x_2-x_3)^2 \ \  \  \ }& &\\
*{\bullet}  \ar@{-}[rrrr]&& &&*{\bullet}
}
\end{equation*}
\caption{Decomposition of the monodromy cone for $FS_2$ and $FS_3$ covers.}\label{Fig:degFS2FS3}
\end{figure}

\begin{proof}
First consider the case $n=2$.  The monodromy cone is
$$
\RR_{\ge 0}\langle (2x_1-x_2)^2,  x_2^2\rangle.
$$
Consider the cones
$$
C_1=\RR_{\ge 0}\langle x_1^2,  x_2^2, (x_1-x_2)^2\rangle
 \ \ \text{ and }  \ \
C_2=\RR_{\ge 0}\langle (2x_1-x_2)^2, x_1^2 , (x_1-x_2)^2\rangle.
$$
One can easily  check (by say computing all the minors of the associated matrices) that  these cones are matroidal.
In fact they are of type $A$ \eqref{Acone}.
Now clearly $x_2^2$ is contained in $C_1$ and $(2x_1-x_2)^2$ is contained in $C_2$.   It only remains to check that the ray dividing the monodromy cone, generated by
$$
(2x_1-x_2)^2+x_2^2,
$$
is in both cones.
We have
$$
\frac{1}{2}\left((2x_1-x_2)^2+x_2^2\right)=\frac{1}{2}\left((4x_1^2-4x_1x_2+x_2^2)+x_2^2\right)=x_1^2+(x_1-x_2)^2.
$$
Thus the central ray of the monodromy cone is also the central ray of the common face of $C_1$ and $C_2$.

Now consider the case $n=3$.
The monodromy cone is
$$
\RR_{\ge 0}\langle (2x_1-x_2-x_3)^2,  x_2^2,x_3^2\rangle.
$$
Motivated by the previous example, let us first see if we can find a cone that contains $x_2^2$, $x_3^2$ and also the middle ray of the cone, generated by
$$
(2x_1-x_2-x_3)^2+ x_2^2+x_3^2.
$$
Let $C_1$ be the cone
$$
C_1=\RR_{\ge 0}\langle x_1^2,x_2^2,x_3^2, (x_1-x_2)^2,(x_1-x_3)^2,(x_1-x_2-x_3)^2\rangle.
$$
It is somewhat tedious, but one can easily  check (by say computing all the minors of the associated matrix) that  this cone is matroidal.
In fact this is equivalent to the principal or $A_n$-cone in genus $3$, and every cone in the perfect cone decomposition is equivalent to a face of this cone.
Computing, we have
$$
(2x_1-x_2-x_3)^2+ x_2^2+x_3^2=4x_1^2+2x_2^2+2x_3^2-4x_1x_2-4x_1x_3+2x_2x_3.
$$
We also have
$$
x_1^2+(x_1-x_2)^2+(x_1-x_3)^2+(x_1-x_2-x_3)^2=4x_1^2+2x_2^2+2x_3^2-4x_1x_2-4x_1x_3+2x_2x_3.
$$
Thus $C_1$ contains
$x_2^2$, $x_3^2$ and also the middle ray $(2x_1-x_2-x_3)^2+ x_2^2+x_3^2$.    Note that again, the middle ray is also the middle ray of the appropriate face of the matroidal cone.
This computation is in fact symmetric: by appropriate change of coordinates, one can cover the monodromy cone with $3$ matroidal cones of this type.
\end{proof}

\begin{rem}\label{remFS2Del}
For the $FS_2$ cone, the midpoint is contained in the matroidal cone $$\mathbb R_{\ge 0}\langle x_1^2,(x_1-x_2)^2\rangle.$$    The right half of the $FS_2$ cone (in Figure \ref{Fig:degFS2FS3}) is contained in the matroidal cone $$\mathbb R_{\ge 0}\langle x_1^2,x_2^2,(x_1-x_2)^2\rangle$$ (the other sub-cone may be studied by symmetry).  We may change basis so these cones are  $\mathbb R_{\ge 0}\langle y_1^2,y_2^2 \rangle$ and  $\mathbb R_{\ge 0}\langle y_1^2,y_2^2,(y_1-y_2)^2\rangle$ respectively.    The Delaunay decomposition of $\mathbb R^2$ for the standard cone $\sigma_{1+1}=\mathbb R_{\ge 0}\langle y_1^2,y_2^2 \rangle$
 is that of squares.  The corresponding chain of toric varieties has components equal to
  $\PP^1 \times \PP^1$, glued to each other along   ``opposite'' coordinate lines $\{0\} \times \PP^1$ and $\{\infty\} \times \PP^1$  as well
as  $\PP^1 \times \{0\}$ and $\PP^1 \times \{\infty\}$.
For the cone $\sigma_{K_3}=\mathbb R_{\ge 0}\langle y_1^2,y_2^2,(y_1-y_2)^2\rangle$, the Delaunay decomposition changes: every square breaks up into two triangles.  The corresponding chain of toric varieties has components equal to $\PP^2$, glued to each other along their coordinate lines.
\end{rem}

\begin{rem}\label{remFS3Del}
For the $FS_3$ cone, the midpoint is contained in the matroidal cone $$\mathbb R_{\ge 0}\langle x_1^2,(x_1-x_2)^2, (x_1-x_3)^2,(x_1-x_2-x_3)^2\rangle.$$    The line
segment joining the midpoint to the lower right corner (in Figure \ref{Fig:degFS2FS3}) is contained in the matroidal cone $\mathbb R_{\ge 0}\langle x_1^2,x_2^2,(x_1-x_2)^2, (x_1-x_3)^2,(x_1-x_2-x_3)^2\rangle$.     The right hand cone of full dimension is contained in the matroidal cone $\mathbb R_{\ge 0}\langle x_1^2,x_2^2,x_3^2,(x_1-x_2)^2, (x_1-x_3)^2,(x_1-x_2-x_3)^2\rangle$.
  (The other sub-cones may be studied by symmetry.)  We may change basis so these cones are  $\mathbb R_{\ge 0}\langle y_1^2,y_2^2,y_3^2,(y_1-y_2-y_3)^2 \rangle$,  $\mathbb R_{\ge 0}\langle y_1^2,y_2^2,y_3^2,(y_1-y_2)^2, (y_1-y_2-y_3)^2 \rangle$, and  $\mathbb R_{\ge 0}\langle y_1^2,y_2^2,y_3^2,(y_1-y_2)^2,(y_1-y_3)^2, (y_1-y_2-y_3)^2 \rangle$ respectively.
The Delaunay decomposition of $\mathbb R^3$ with respect to the cone $\sigma_{C_4}=\mathbb R_{\ge 0}\langle y_1^2,y_2^2,y_3^2,(y_1-y_2-y_3)^2 \rangle$
consists of the tiling by  $1$ octahedron and $2$ tetrahedra (i.e.~the cube with $2$ tetrahedra cut from opposite corners).  We direct the reader to Remark \ref{remDelaunay} for the Delaunay decompositions for the other cones.
\end{rem}

\section{Some examples where the Prym map extends} \label{secexamples}

In this section we consider a number of further examples of admissible covers.
In short, in each of these examples, the Prym map extends to the   second Voronoi, perfect cone and central cone compactifications.  These examples will also be useful later when we consider degenerations of Friedman--Smith covers.

\subsection{Beauville examples}

A Beauville example with $n$ nodes is an admissible cover where $\widetilde C$ is irreducible, has exactly  $n$ nodes, and all of the nodes are fixed by the involution.    The dual graph has a unique vertex $\tilde v$ and exactly $n$ edges $\tilde e_1,\ldots,\tilde e_n$, all fixed by the involution (the dual graph in the case $n=1$ is given in Figure \ref{Fig:d0ram}).
 One obtains that $H_1(\widetilde \Gamma,\ZZ)=H_1(\widetilde \Gamma,\ZZ)^+=\ZZ\langle \tilde e_1,\ldots,\tilde e_n \rangle$.  Consequently, $H_1(\widetilde \Gamma,\ZZ)^-=0$.   Thus the Prym is an abelian variety, and the Prym map extends in  a neighborhood of $\widetilde C$.  These examples lie in the $n$-fold self-intersection of  $\delta_0^{\operatorname{ram}}$.

\subsection{Elementary \'etale examples} \label{secEEE} An elementary \'etale example with $2n$ nodes is an admissible cover  where $\widetilde C$ is irreducible, has exactly  $2n$ nodes, and all of the nodes are exchanged pairwise by the involution.    The dual graph has a unique vertex $\tilde v$ and exactly $2n$ edges $\tilde e_1^+,\tilde e_1^-,\ldots,\tilde e_n^+,\tilde e_n^-$, with $\iota\tilde e_i^+=\tilde e_i^-$ ($i=1,\ldots,n$).    The dual graph of $C$ consists of a unique vertex $v$ and $n$ edges $e_1,\ldots,e_n$, with $\tilde e_i^{\pm}$ lying over $e_i$ (the dual graph in the case $n=1$ is given in Figure \ref{Fig:d0'}).
In this case one has $H_1(\widetilde \Gamma,\ZZ)=\ZZ\langle \tilde e_1^+,\tilde e_1^-,\ldots,\tilde e_n^+,\tilde e_n^-\rangle$.  Consequently, we have  $H_1(\widetilde \Gamma,\ZZ)^{[-]}=\ZZ\langle \frac{1}{2}(\tilde e_1^+-\tilde e_1^-),\ldots,\frac{1}{2}(\tilde e_n^+-\tilde e_n^-)\rangle$.   As $H_1(\widetilde \Gamma,\ZZ)^{[-]}$ is dual to $H^1(\widetilde \Gamma,\ZZ)^-$, we can see that if we set $z_i= \frac{1}{2}(\tilde e_i^+-\tilde e_i^-)$ for $i=1,\ldots,n$, then we may take $\ell_{e_i}=z_i^\vee$ ($i=1,\ldots,n$).    The matrix expressing the $\ell_{e_i}$  in terms of this basis is then the identity, so clearly (V) of Theorem \ref{teoprymext}  holds. In addition, we can take the standard positive definite form $z_1^2+\ldots +z_n^2$  to show that (P) and (C) of Theorem \ref{teoprymext} both hold.    In conclusion, the Prym map extends to the   second Voronoi, perfect cone, and central cone compactifications in a neighborhood of an elementary \'etale example.  These examples lie in the  $n$-fold self-intersection of  $\delta_0'$.

\begin{rem}\label{remEtDel}
The monodromy cone is the standard cone  $\sigma_{1 + \cdots + 1}=\mathbb R_{\ge 0}\langle x_1^2,\ldots,x_n^2 \rangle$.  The associated Delaunay decomposition of $\mathbb R^n$ is the tiling by $n$-cubes.  The associated chain of toric varieties has components equal to $(\mathbb P^1)^n$  with ``opposite sides'' glued.
\end{rem}

\subsection{Wirtinger examples ($W_n$)}\label{secWn} \label{secWE}
A Wirtinger example with $2n$ nodes is an admissible double cover where $\widetilde C$ has exactly two irreducible components, which are interchanged by the involution, and $\widetilde C$ has exactly $2n$ nodes, all of which join the two components, and are interchanged pairwise by the involution.
  The dual graph $\widetilde \Gamma$ of $\widetilde C$ has vertices $V(\widetilde \Gamma)=\{\tilde v^+,\tilde v^-\}$ and edges $E(\widetilde \Gamma)=\{\tilde e_1^+,\tilde e_1^-,\ldots,\tilde e_n^+,\tilde e_n^-\}$; we orient $\tilde e_i^+$ from $\tilde v^-$ to $\tilde v^+$, and $\tilde e_i^-$ in the opposite direction.    The involution $\iota$ acts by $\iota(\tilde v^+)=\tilde v^-$  and $\iota(\tilde e_i^+)=\tilde e_i^-$ ($i=1,\ldots,n$).    The dual graph of $C$ consists of a single vertex $v$ and $n$ edges $e_1,\ldots,e_n$ with $\tilde e_i^\pm$ lying over $e_i$.

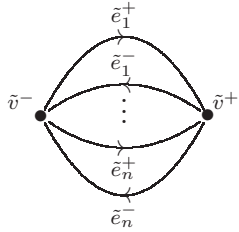
\begin{figure}[htb]
\begin{equation*}
\xymatrix{
 *{\bullet} \ar @{-}@/_1pc/[rr]|-{\SelectTips{cm}{}\object@{>}}_{\tilde e_n^+} \ar @{-} @/_2.5pc/[rr] |-{\SelectTips{cm}{}\object@{<}}_{\tilde e_n^-}
 \ar@{-} @/^1pc/[rr]|-{\SelectTips{cm}{}\object@{<}}^{\tilde e_1^-}  \ar@{-}@/^2.5pc/[rr]|-{\SelectTips{cm}{}\object@{>}}^{\tilde e_1^+}^<{\tilde v^-}^>{\tilde v^+}  &\vdots &*{\bullet}
}
\end{equation*}

\caption{  Dual graph of a Wirtinger example with $2n\ge 2$ nodes ($W_n$).}\label{Fig:dgWE}
\end{figure}
One has
$H_1(\widetilde \Gamma,\ZZ)= \ZZ\langle \tilde e_1^++\tilde e_1^-,\ldots,\tilde e_n^++\tilde e_n^-, \tilde e_1^++\tilde e_2^-,\ldots,\tilde e_{n-1}^++\tilde e_n^-\rangle$
so that in turn
$
H_1(\widetilde \Gamma,\ZZ)^{[-]}\cong  \ZZ\langle \frac{1}{2}(\tilde e_1^+-\tilde e_1^-)-\frac{1}{2}(\tilde e_2^+-\tilde e_2^-),\ldots,\frac{1}{2}(\tilde e_{n-1}^+-\tilde e_{n-1}^-)-\frac{1}{2}(\tilde e_n^+-\tilde e_n^-)\rangle
$.
For brevity, set
$$
z_1=\frac{1}{2}(\tilde e_1^+-\tilde e_1^-)-\frac{1}{2}(\tilde e_2^+-\tilde e_2^-), \ldots, \ z_{n-1}=\frac{1}{2}(\tilde e_{n-1}^+-\tilde e_{n-1}^-)-\frac{1}{2}(\tilde e_n^+-\tilde e_n^-)
$$
so that
$
H_1(\widetilde \Gamma,\ZZ)^{[-]}\cong \ZZ\langle z_1,\ldots,z_{n-1}\rangle.
$
Then $$H^1(\widetilde \Gamma,\ZZ)^-=\left(H_1(\widetilde \Gamma,\ZZ)^{[-]}\right)^\vee\cong \ZZ\langle z_1^\vee,\ldots,z_{n-1}^\vee \rangle .$$
One can check
that for $i=1,\ldots,n$, one may take
$
\ell_{e_i}=
(\tilde e_i^+)^\vee-(\tilde e_i^-)^\vee
$.
The $n\times (n-1)$ matrix expressing the $\ell_{e_i}$ in terms of the basis
$z_1^\vee,\ldots, z_{n-1}^\vee$ is
\begin{equation}\label{eqnWirtinger}
\left(\begin{smallmatrix}
1&&&&&&\\
-1&1&&&&&\\
&-1&1&&&&\\
&&\ddots&\ddots&&&\\
&&&-1&1&&\\
&&&&-1&1&\\
&&&&&-1&1\\
&&&&&&-1\\
\end{smallmatrix}\right)
\end{equation}
Considering the $(n-1)\times (n-1)$ minors, we see that (V) holds.
Now consider the quadratic form
$$
z_1^2+z_1z_2+z_2^2+\ldots+z_{n-2}^2+z_{n-2}z_{n-1}+z_{n-1}^2.
$$
The associated matrix has entries equal $1$ on the diagonal, and $1/2$ above and below the diagonal.  The determinant of such a square matrix of size $m$ is $\frac{m+1}{2^m}$.  Consequently, the quadratic form is positive definite, since all of its leading minors are positive.
It is easy to see by inspection that the rows of the matrix \eqref{eqnWirtinger} are exactly the shortest vectors of this form.
Hence the monodromy cone is a matroidal cone which also satisfies (C). In fact this cone is contained in the $A_{n-1}$-cone, see \eqref{Acone}.
In conclusion, the Prym map extends to the  second Voronoi, perfect cone, and central cone compactifications in a neighborhood of a Wirtinger example.

\begin{rem}\label{remWDel}
The $W_n$ monodromy cone is $\mathbb R_{\ge 0}\langle x_1^2,(x_1-x_2)^2, \ldots, (x_{n-2}-x_{n-1})^2,x_{n-1}^2\rangle$.   This is equivalent to the cone $\mathbb R_{\ge 0}\langle (y_1+\ldots+y_{n-1})^2,y_1^2,\ldots, y_{n-1}^2\rangle$.
For $n=3$, the
Delaunay decomposition of $\mathbb R^2$ is the tiling by $2$-simplices obtained by slicing the standard square into $2$ triangles. The associated chain of toric varieties has components equal to $\mathbb P^2$, glued along coordinate hyperplanes.
We have already seen in Remark \ref{remFS3Del} that for  $n=4$, the   Delaunay decomposition of $\mathbb R^3$ with respect to the cone $\sigma_{C_4}=\mathbb R_{\ge 0}\langle y_1^2,y_2^2,y_3^2,(y_1-y_2-y_3)^2 \rangle$
consists of the tiling by  $1$ octahedron and $2$ tetrahedra (i.e.~the cube with $2$ tetrahedra cut from opposite corners).
\end{rem}

\subsection{Mixed Beauville-\'etale examples}
In these examples we consider admissible double covers where $\tilde C$ has a  dual graph with  a single vertex $\tilde v$, and $2n+m$ edges $$\tilde e_1^+,\tilde e_1^-,\ldots,\tilde e_n^+,\tilde e_n^-,\tilde e_{2n+1},\ldots,\tilde e_{2n+m},$$ where $\iota\tilde e_i^+=\tilde e_i^-$ ($i=1,\ldots,n$) and $\iota \tilde e_j=\tilde e_j$, ($j=2n+1,\ldots,2n+m$).
Then $H_1(\widetilde \Gamma,\ZZ)=\ZZ\langle \tilde e_1^+,\tilde e_1^-,\ldots,\tilde e_n^+,\tilde e_n^-,\tilde e_{2n+1},\ldots,\tilde e_{2n+m} \rangle $, and
$$H_1(\widetilde \Gamma,\ZZ)^{[-]}=\ZZ\langle  \frac{1}{2}(\tilde e_1^+-\tilde e_1^-),\ldots,\frac{1}{2}(\tilde e_n^+-\tilde e_n^-) \rangle .$$  One can then check that $\ell_{e_i}=\tilde (e_i^+)^\vee-(\tilde e_i^-)^\vee$ for $i=1,\ldots,n$ and $\ell_{e_j}=0$ for $j=2n+1,\ldots,2n+m$.  In other words we are exactly back in the situation of the elementary \'etale examples, and consequently the Prym map extends.

\subsection{Mixed Beauville--Friedman--Smith example}
In these examples we consider admissible covers where
 $\tilde C$ has a  dual graph with  two  vertices $\tilde v_1,\tilde v_2$, fixed by the involution, and $2n+m$ edges $$\tilde e_1^+,\tilde e_1^-,\ldots,\tilde e_n^+,\tilde e_n^-,\tilde e_{2n+1},\ldots,\tilde e_{2n+m},$$ where $\iota\tilde e_i^+=\tilde e_i^-$ ($i=1,\ldots,n$) and $\iota \tilde e_j=\tilde e_j$, ($j=2n+1,\ldots,2n+m$).
One finds that $H_1(\widetilde \Gamma,\ZZ)^{[-]}=\ZZ\langle\frac{1}{2}(\tilde e_1^+-\tilde e_1^-,\ldots,\frac{1}{2}(\tilde e_n^+-\tilde e_n^-) \rangle $,  $\ell_{e_i}=\tilde (e_i^+)^\vee-(\tilde e_i^-)^\vee$ for $i=1,\ldots,n$ and $\ell_{e_j}=0$ for $j=2n+1,\ldots,2n+m$.
Consequently, we find ourselves exactly back in the case of the elementary \'etale examples, and the Prym map extends in a neighborhood of these examples.

\section{Degenerations of Friedman--Smith covers}\label{secDegen}

Here we investigate several classes of degenerations of Friedman--Smith covers that arise in describing the resolution of the Prym map in low genus.

\subsection{Vologodsky's degenerations ($DR_n$)} \label{secDR}
In \cite{vologodskyres}, Vologodsky investigates a certain class of degenerations of Friedman--Smith covers, which  he denotes by $DR_n$.  These are covers
with $n$ smooth irreducible components $\widetilde C_1,\ldots,\widetilde C_n$, all preserved by the involution,  with each irreducible component $\widetilde C_i$ meeting $\widetilde C_{i-1}$ and $\widetilde C_{i+1}$ in exactly two nodes each, which are interchanged pairwise by the involution.  Here $\widetilde C_0=\widetilde C_n$ and $\widetilde C_{n+1}=\widetilde C_{1}$.

\begin{figure}[htb]
\begin{equation*}
\xymatrix@C=.5cm{
&*{\bullet} \ar @{-} @/_.5pc/[ld]|-{\SelectTips{cm}{}\object@{>}}_{\tilde e_2^+} \ar @{-} @/^.5pc/[ld]|-{\SelectTips{cm}{}\object@{>}}^{\tilde e_2^-}& &*{\bullet} \ar @{-} @/_.5pc/[ll]|-{\SelectTips{cm}{}\object@{>}}_{\tilde e_1^+}  \ar @{-} @/^.5pc/[ll]|-{\SelectTips{cm}{}\object@{>}}^{\tilde e_1^-} &\\
*{\bullet} \ar @{-} @/_.5pc/[rd]|-{\SelectTips{cm}{}\object@{>}}_{\tilde e_3^+} \ar @{-} @/^.5pc/[rd]|-{\SelectTips{cm}{}\object@{>}}^{\tilde e_3^-}  &&&& *{\bullet}\ar @{-} @/_.5pc/[lu]|-{\SelectTips{cm}{}\object@{>}}_{\tilde e_6^+} \ar @{-} @/^.5pc/[lu]|-{\SelectTips{cm}{}\object@{>}}^{\tilde e_6^-}  \\
&*{\bullet} \ar @{-} @/_.5pc/[rr]|-{\SelectTips{cm}{}\object@{>}}_{\tilde e_4^+} \ar @{-} @/^.5pc/[rr]|-{\SelectTips{cm}{}\object@{>}}^{\tilde e_4^-} &&*{\bullet} \ar @{-} @/_.5pc/[ru]|-{\SelectTips{cm}{}\object@{>}}_{\tilde e_5^+} \ar @{-} @/^.5pc/[ru]|-{\SelectTips{cm}{}\object@{>}}^{\tilde e_5^-} &\\
}
\end{equation*}
\caption{  Dual graph of a $DR_6$ cover.}\label{Fig:dgDR6}
\end{figure}
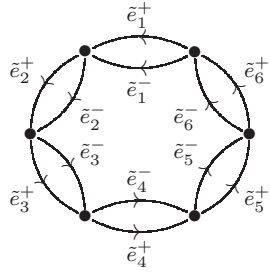

The $DR_n$ dual graph is similar to the one in Figure \ref{Fig:dgDR6}, except with $2n$ edges.
 One can see that $DR_1$ lies in $\delta_0'$, $DR_2=FS_2$,
and that for $n\ge 2$ $DR_n$ is contained in $\overline {FS}_2$ as well as as in the $n$-fold self-intersection of $\delta_0'$.   Moreover, for $n\ge 2$, $DR_{n+1}\subseteq \overline{DR}_n$.   \'Etale locally near a $DR_n$ example with $n\ge 2$, every $\binom{n}{2}$ intersections of $2$ \'etale local components of the  $\delta_0'$ divisors is an \'etale local component of  the $\overline{FS}_2$ locus.

   One can check that
$$
H_1(\widetilde \Gamma,\mathbb Z)=\mathbb Z\langle
\tilde e_1^+-\tilde e_1^-,\ldots,\tilde e_n^+-\tilde e_n^-,\tilde e_1^+\ldots+\tilde e_n^+ \rangle
$$
and
$$
H_1(\widetilde \Gamma,\mathbb Z)^{[-]}=\mathbb Z\langle
\tilde e_1^+-\tilde e_1^-,\ldots,\tilde e_{n-1}^+-\tilde e_{n-1}^-,\frac{1}{2}(\tilde e_1^+-\tilde e_1^-)+\ldots+\frac{1}{2}(\tilde e_n^+-\tilde e_n^-) \rangle.
$$
The monodromy cone is then given by the $n\times n$ matrix:
\begin{equation}\label{eqnDRn}
\begin{tiny}
DR_n \ \ \ \ \left(
\begin{array}{cccc|cc}
2&0&\cdots&&0&1\\
0&2&0&&0&1\\
0&0&2&0&&\\
&&&&0&1\\\hline
&&&&2&1\\
&&&&&1\\
\end{array}\right)=
\left(\begin{array}{c|cc}
2\operatorname{Id}_{n-2}& 0 & 1\\\hline
0& FS_{2}&
\end{array}\right)
\end{tiny}
\end{equation}
One can see immediately that these cones are simplicial.
Vologodsky \cite[Prop.~1.3]{vologodskyres} gives the   second Voronoi decomposition of these cones:
\emph{Viewing the cone as a cone over the standard $n-1$ simplex in $\mathbb R^n$, the   second Voronoi decomposition is the collection of cones defined by the hyperplanes in $\mathbb R^n$ defined by
$
\sum_{j\in J}c_j=\sum_{j\notin J}c_j
$
for every proper subset $J\subset \{1,\ldots,n\}$.}
For $n=2,3$ one obtains the decompositions depicted in Figure \ref{Fig:MCDR2DR3}.
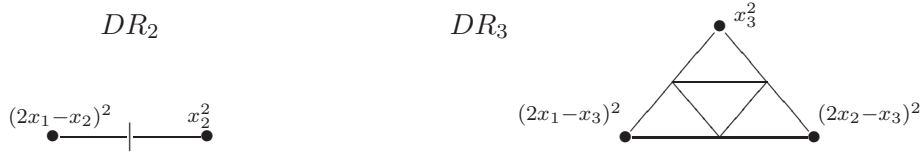
\begin{figure}[htb]
\begin{equation*}
\xymatrix@C=.4cm@R=1cm{
&DR_2&\\
*{\bullet}\ar@{-}[r]|-{\SelectTips{cm}{}\object@{}} ^<{(2x_1-x_2)^2}& *{|} \ar@{-}[r]|-{\SelectTips{cm}{}\object@{}}^>{x_2^2}  &*{\bullet}\\
}
\ \ \ \  \hskip 1 in DR_3
\xymatrix@C=.4cm@R=.5cm{
&&*{\bullet} \ar@{-}[rrdd]|-{\SelectTips{cm}{}\object@{}}^>{(2x_2-x_3)^2} ^<{x_3^2} \ar@{-}[lldd]|-{\SelectTips{cm}{}\object@{}}_>{(2x_1-x_3)^2} &&\\
&*{} \ar@{-}[rd] \ar@{-}[rr]&&*{} \ar@{-}[ld] &\\
*{\bullet}  \ar@{-}[rrrr]&& *{}&&*{\bullet}
}
\end{equation*}
\caption{Decomposition of the monodromy cone for $DR_2$ and $DR_3$ covers.}\label{Fig:MCDR2DR3}
\end{figure}
For dimension reasons, this also gives the   perfect  and central cone decomposition.

We now describe  the perfect cone decomposition of the monodromy cone of $DR_n$ covers for all $n$.    Before doing this, let us introduce some notation.
Let $\Delta^n$ be the standard $n$-simplex in $\mathbb R^{n+1}$;  i.e.~the convex hull on the basis vectors $e_1,\ldots,e_{n+1}$.    Define $\Delta_i^n$ to be the convex hull over the $n+1$ vectors
$$
\frac{e_i+e_1}{2},\ldots,\frac{e_i+e_{n+1}}{2}
$$
In other words $\Delta_i^n$ is the $n$-simplex obtained by cutting off the $i$-th corner of the original simplex at the mid-point of the edges containing $e_i$.

\begin{pro}
Viewing the monodromy cone for a $DR_n$ degeneration as the cone over the $n$-simplex $\Delta^n$, the perfect cone decomposition is given by
$$
\Delta^n=\Delta_1^n\cup \ldots \cup \Delta^n_n \cup \overline{(\Delta^n-\bigcup_{i=1}^n\Delta^n_i)}.
$$
Moreover, the cones over the $\Delta^n_i$ are all contained in matroidal cones of type $A$, and the remaining  cone, the cone over $\overline{(\Delta^n-\bigcup_{i=1}^n\Delta^n_i)}$, is contained in a perfect cone, of type $D$ (see Remark \ref{remlowdim}).  In particular, the    second Voronoi decomposition of a $DR_n$ monodromy cone is a refinement of the perfect cone decomposition.
\end{pro}

\begin{rem}
This also gives the decomposition of the monodromy cone  into cones contained in cones in the CCD.
\end{rem}

\begin{proof}
From \eqref{eqnDRn}, and an elementary change of coordinates, the $DR_n$ monodromy cone can be taken to be  generated by the $n$ quadratic forms $M_1=x_1^2$, $M_i=(2x_i-x_1)^2$,  ($i=2,\ldots,n$); note that in fact with these definitions $M_i=(2x_i-x_1)^2$ for all $i$.     Now define
$$
N_{ab}:=\frac{1}{2}(M_a+M_b) \ \ \ 1\le a\le b\le n,
$$
to be the midpoint of the segment joining $M_a$ to $M_b$.  We claim that the cones
$$
\Delta^n_i:=\mathbb R_{\ge 0}\langle N_{i,1},\ldots, N_{i,n}\rangle
$$
$$
D:=\mathbb R_{\ge 0}\langle N_{ij}\rangle_{i\ne j}
$$
are contained in cones in the perfect cone decomposition, of types $A$ and $D$ respectively.

To see this, observe that
$$
N_{ij}=\frac{1}{2}(2x_i-x_1)^2+(2x_j-x_1)^2)=(x_1-x_i-x_j)^2+(x_i-x_j)^2,
$$
which reduces in the case $i=1$ to
$$
N_{1i}=(x_1-x_i)^2+x_i^2.
$$
Now one can easily check that $\Delta^n_1=\mathbb R_{\ge 0}\langle N_{1,1},\ldots, N_{1,n}\rangle
$ is contained in a cone of type $A$ (see \eqref{Acone}).
The other cones $\Delta^n_2,\ldots, \Delta^n_n$ are contained in type $A$ cones by symmetry.  Similarly, one can check that the cone $D$ is contained in a cone of type $D$ (see \eqref{Acone}).
\end{proof}

\begin{rem}
As mentioned above, the geometric description near a $DR_n$ example is that there are $n$ copies of $\delta_0'$ meeting.  The $\binom{n}{2}$, $2$-fold intersections correspond to $\overline{FS}_2$ loci.  The toric resolution given by the above decomposition of the monodromy cone  is supported along the union of these loci.  For $n=2$ it is the blow-up of the $\overline {FS}_2$ locus.  For $n=3$ it is a more complicated birational modification described in more detail in \S \ref{sectRes}.
\end{rem}

\subsection{Friedman--Smith--Wirtinger degenerations ($FS_n+W_m$)} \label{secFS+W}  In these examples, we degenerate one of the curves in a Friedman--Smith cover ($FS_n$) to a Wirtinger example ($W_m$).     The dual  graph is given in Figure \ref{Fig:FSn+Wn}.

\begin{figure}[htb]
\begin{equation*}
\xymatrix@C=2cm@R=.05cm{
 *{\bullet} \ \ \ar @{-} @/^1.5pc/[dd]|-{\SelectTips{cm}{}\object@{>}}^{\tilde f_m^-}  \ar @{-} @/_1.5pc/[dd]|-{\SelectTips{cm}{}\object@{>}}_{\tilde f_1^-} \ar @{-} @/^3.5pc/[rrd]|-{\SelectTips{cm}{}\object@{>}}^{\tilde e_1^+} ^<{\tilde v_1^+} \ar @{-} @/^1.5pc/[rrd]|-{\SelectTips{cm}{}\object@{>}}^{\tilde e_{2}^+} && \\
\cdots &\vdots&*{\bullet}\\
*{\bullet} \  \   \ar @{-} @/^3.5pc/[uu]|-{\SelectTips{cm}{}\object@{>}}^{\tilde f_1^+}\ar @{-} @/_3.5pc/[uu]|-{\SelectTips{cm}{}\object@{>}}_{\tilde f_m^+} \ar @{-} @/_3.5pc/[rru]|-{\SelectTips{cm}{}\object@{>}}_{\tilde e_{n}^-} _<{\tilde v_1^-}  _>{\tilde v_2}\ar @{-} @/_1.5pc/[rru]|-{\SelectTips{cm}{}\object@{>}}_{\tilde e_{n-1}^-} &&\\
}
\end{equation*}
\caption{  Dual graph of $FS_{n}+W_{m}$  degeneration of a Friedman--Smith cover with $2n\ge 2$ nodes.}\label{Fig:FSn+Wn}
\end{figure}
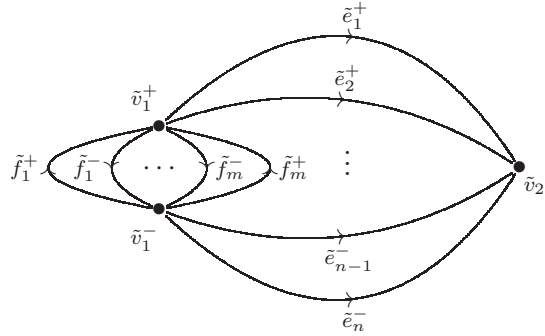

These examples lie in the $(n+m)$-fold self-intersection of $\delta_0'$.  They lie in $\overline{FS}_2\cup \overline{FS}_3$ if and only if $n=2,3$.
One can check that
$$
H_1(\widetilde \Gamma,\mathbb Z)=\mathbb Z\langle (\tilde e_1^+-\tilde e_1^-)+\tilde f_1^+, \ldots, (\tilde e_n^+-\tilde e_n^-)+\tilde f_1^+, (\tilde e_1^+-\tilde e_2^-)+\tilde f_1^+, \ldots, (\tilde e_{n-1}^+-\tilde e_n^-)+\tilde f_1^+,
$$
$$
\tilde f_1^+-\tilde f_1^-,\ldots,\tilde f_m^+-\tilde f_m^-,\tilde f_1^+-\tilde f_2^-,\ldots, \tilde f_{m-1}^+-\tilde f_m^- \rangle.
$$
and
$$
H_1(\widetilde \Gamma,\mathbb Z)^{[-]}=
\mathbb Z\langle (\tilde e_1^+-\tilde e_1^-)+\frac{1}{2}(\tilde f_1^+-\tilde f_1^-),
\frac{1}{2}(\tilde e_1^+-\tilde e_1^-)+\frac{1}{2}(\tilde e_2^+-\tilde e_2^-)+\frac{1}{2}(\tilde f_1^+-\tilde f_1^-),
$$
$$\ldots, \frac{1}{2}(\tilde e_{n-1}^+-\tilde e_{n-1}^-)+\frac{1}{2}(\tilde e_n^+-\tilde e_n^-)+\frac{1}{2}(\tilde f_1^+-\tilde f_1^-),
$$
$$
\frac{1}{2}(\tilde f_1^+-\tilde f_1^-)-\frac{1}{2}(\tilde f_2^+-\tilde f_2^-),\ldots,  \frac{1}{2}(\tilde f_{m-1}^+-\tilde f_{m-1}^-)-\frac{1}{2}(\tilde f_m^+-\tilde f_m^-)\rangle
$$

One can then show that the monodromy cone is generated by the squares of the rows of
the matrix:
\begin{equation}\label{eqnFSWmat}
\begin{tiny}
FS_n+W_m \ \ \  \left(
\begin{array}{ccccc|ccccc}
2&-1&-1&&-1&0&&&0&0\\
&1&0&0&&0&&&0&0\\
&&&\ddots&&& &&\\
FS_{n}&&&&1&0&0&&0&0\\ \hline
1&0&-1&0&*&1&0&&&0\\
0&0&0&&&-1&1&&&0\\
0&0&0&&&&-1&1&&0\\
&&&&&&&\\
&&&&&W_{m}&&&&-1\\
\end{array}
\right)
=
\left(
\begin{array}{c|c}
FS_n& 0\\\hline
v_n&W_m\\
0&\\
\end{array}
\right)
\end{tiny}
\end{equation}
where $*$   is defined to be $-1$ or $0$ depending on if $n$ is odd or even respectively.  Recall that the $W_{m}$ matrix is an $m\times (m-1)$ matrix (see \S \ref{secWn}). In particular, for $m=1$ the two right hand blocks of \eqref{eqnFSWmat} are missing.

\begin{lem}\label{lemmcFSW}  Consider the case $m=1$.   The monodromy cone for a $FS_n+W_1$ degeneration is not contained in a cone in the PCD for $n=2,3$.  The monodromy cone is  contained in a cone in the PCD for $n=1,4,5,6,7$.
\end{lem}

\begin{rem}
In fact the monodromy cone is also contained in a cone in the CCD for $n=1,4,5,7$.
\end{rem}

\begin{proof}
For $n=1$, this is clear.  For $n=2,3$, we have observed already  above   that these examples are degenerations of $\overline{FS}_2$ and $\overline {FS}_3$ examples, and so are not contained in cones in the PCD.  For $n=4$, the quadratic form $Q$ defined in  \eqref{eqnQ} shows that the monodromy cone is contained in a cone in the PCD.  For $n=5,6,7$, the quadratic forms below   \begin{equation}\label{eqnQ56}
Q_5 =
\left(
\begin{smallmatrix}
1 & \frac{1}{2} & \frac{1}{2}& \frac{1}{2} & \frac{1}{2}\\
\frac{1}{2}&1& 0 & 0 & \frac{1}{2}\\
\frac{1}{2}&0&1&0 & 0\\
\frac{1}{2}&0&0&1 &0\\
\frac{1}{2}&\frac{1}{2} & 0 &0&1\\
\end{smallmatrix}
\right),
Q_6=\frac{1}{20}\left(
\begin{smallmatrix}
  23&   7 &   9 &   7 &   9 &   7  \\
    7 &    15&   1 &  -1 &   1 &  -1  \\
  9 &   1 &    15&   1 &  -1 &   1  \\
    7 &  -1 &   1 &    15&   1 &  -1  \\
   9 &   1 &  -1 &   1 &    15&   1  \\
    7 &  -1 &   1 &  -1 &   1 &    15 \\
\end{smallmatrix}
\right),
Q_7=
\left(
\begin{smallmatrix}
1 & \frac{1}{2} & -\frac{1}{2}& \frac{1}{2} & \frac{1}{2} &\frac{1}{2}&\frac{1}{2}\\
\frac{1}{2}&1& 0 & 0 & 0 &0&0 \\
-\frac{1}{2}&0&1&-\frac{1}{2} & -\frac{1}{2} &-\frac{1}{2}&-\frac{1}{2}\\
\frac{1}{2}&0&-\frac{1}{2}&1 &\frac{1}{2} &\frac{1}{2}&0\\
\frac{1}{2}&0 & -\frac{1}{2} &\frac{1}{2}&1&\frac{1}{2}&0\\
\frac{1}{2}&0&-\frac{1}{2}&\frac{1}{2}&\frac{1}{2}&1&0\\
\frac{1}{2}&0&-\frac{1}{2}&0&0&0&1
\end{smallmatrix}
\right)
\end{equation}
show that the monodromy cone is contained in a cone in the PCD.
We thank Mathieu Dutour Sikiri\'c for providing us with the metric $Q_6$, which actually shows that the monodromy  cone \emph{is} a cone in the PCD.
\end{proof}

We now turn our attention to the perfect cone decomposition of the monodromy cone for $n=2,3$, $m=1$; that is to say for the cones:
$$
FS_2+W_1\ \ \ \left(\begin{array}{cc}
2&-1\\
0&1\\ \hline
1&0\\
\end{array}\right)
\ \ \ \text{ and } \ \  FS_3+W_1 \ \
\left(\begin{array}{ccc}
2&-1&-1\\
0&1&0\\
0&0&1\\ \hline
1&0&-1\\
\end{array}\right)
$$
It turns out that the analysis in the proof of Lemma \ref{lemFSPCD} actually provides the decomposition in this case as well.
For $n=2$, the star subdivision associated to one of the edges of the cone provides the decomposition; see Figure \ref{figFS2W1decomp}.
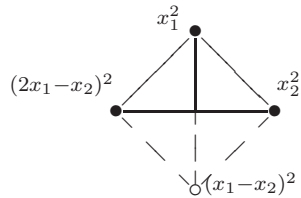
\begin{figure}[htb]
$$
\xymatrix{
&*{\bullet} \ar@{-}[rd]|-{\SelectTips{cm}{}\object@{}}^>{x_2^2} \ar@{-}[ld]|-{\SelectTips{cm}{}\object@{}}_<{x_1^2}  _>{(2x_1-x_2)^2} \ar@{-}[d]|-{\SelectTips{cm}{}\object@{}} & \\
*{\bullet} \ar@{-}[rr]|-{\SelectTips{cm}{}\object@{}}&*{} \ar@{--}[d]|-{\SelectTips{cm}{}\object@{}}^>{(x_1-x_2)^2}&*{\bullet}\\
&*{\circ} \ar@{--}[ru]|-{\SelectTips{cm}{}\object@{}} \ar@{--}[lu]|-{\SelectTips{cm}{}\object@{}}&\\
}
$$
\caption{Perfect cone decomposition of  an $FS_{2}+W_{1}$  monodromy cone.}\label{figFS2W1decomp}
\end{figure}
The dashed lines in the figure show the ambient matroidal cones giving the decomposition of the monodromy cone (which are depicted by  the solid lines).

For $n=3$, the star subdivision associated to one of the $2$-dimensional faces of the cone provides the decomposition; see Figure \ref{Fig:FS3+W1MC}.
\begin{figure}[htb]
$$
\xymatrix@C=.5cm@R=.3cm{
&&*{\bullet}\ar@{-}[rrd]|-{\SelectTips{cm}{}\object@{}} ^<{(x_1-x_3)^2}\ar@{-}[rdddd] \ar@{-}[lldd] \ar@{-}[dd]|-{\SelectTips{cm}{}\object@{}}&&\\
&&&&*{\bullet} \ar@{-}[lld]|-{\SelectTips{cm}{}\object@{}} \\
*{\bullet} \ar@{-}[rrrru]|-{\SelectTips{cm}{}\object@{}}  ^<{(2x_1-x_2-x_3)^2\ \ \ \ \ \ \  }  \ar@{-}[rrrdd] \ar@{-}[rr]|-{\SelectTips{cm}{}\object@{}}&&*{\circ}&&\\
&&&&\\
&&&*{\bullet} \ar@{-}[ruuu]|-{\SelectTips{cm}{}\object@{}} _>{x_3^2}_<{x_2^2}\ar@{-}[luu]|-{\SelectTips{cm}{}\object@{}}&\\
}
$$
\caption{Perfect cone decomposition of a  $FS_{3}+W_{1}$  monodromy cone.}\label{Fig:FS3+W1MC}
\end{figure}
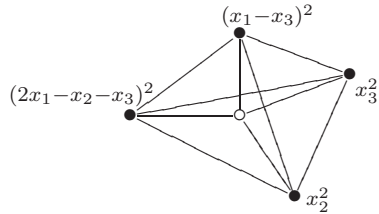
The open bullet in the figure is  the midpoint of the bottom face; i.e.~it corresponds to the quadratic form $\frac{1}{3}((2x_1-x_2-x_3)^2+x_2^2+x_3^2)$.
Note that the face  $C(x_2^2,x_3^2,(x_1-x_3)^2)$ is  contained in the cone
$$
C_1=\mathbb R_{\ge 0}\langle x_1^2,x_2^2,x_3^2, (x_1-x_2)^2,(x_1-x_3)^2,(x_1-x_2-x_3)^2\rangle,
$$
discussed in the proof of Lemma \ref{lemFSPCD}.
We showed there that this cone $C_1$ contains the midpoint of the bottom face (the ray generated by $(2x_1-x_2-x_3)^2+x_2^2+x_3^2$, corresponding to the bullet point in the diagram).  The other sub-cones  can be obtained by change of coordinates.

\begin{rem}
Geometrically, we obtain the following picture of the resolution.  For the case $n=2$, there are $3$ copies of $\delta_0'$ meeting; exactly one pair of intersections gives rise to a single $\overline{FS}_2$ locus.  This locus is blown-up.  For the $n=3$ case, there is a single  $3$-fold intersection of the $\delta_0'$ divisors that gives a $\overline{FS}_3$ locus.  This is blown-up.
\end{rem}

\subsection{Friedman--Smith--Friedman--Smith degenerations \\  ($FS_{n_1+n_2}+FS_m$)} \label{secFS+FS} Here we  consider the case where one of the curves in a Friedman--Smith cover of type $FS_n$ degenerates to a Friedman--Smith cover of type $FS_m$.   It is best to break $n$ down into $n=n_1+n_2$.
The dual  graph is depicted in Figure \ref{Fig:FSn+FSm}.

\begin{figure}[htb]
\begin{equation*}
\xymatrix@C=3cm@R=.8cm{
 *{\bullet} \ \ \ar @{-} @/^1.5pc/[dd]|-{\SelectTips{cm}{}\object@{<}}^{\tilde f_m^-}  \ar @{-} @/_1.5pc/[dd]|-{\SelectTips{cm}{}\object@{<}}_{\tilde f_1^-} \ar @{-} @/^3.5pc/[rrd]|-{\SelectTips{cm}{}\object@{>}}^{\tilde e_1^+} ^<{\tilde v_1} \ar @{-} @/^1.5pc/[rrd]|-{\SelectTips{cm}{}\object@{>}}^{\tilde e_{1}^-}    \ar @{-} @/_.5pc/[rrd]|-{\SelectTips{cm}{}\object@{>}}_{\ \ \ \tilde e_{n_1}^-}^\vdots && \\
\cdots &&*{\bullet}\\
*{\bullet} \  \   \ar @{-} @/^1pc/[rru]|-{\SelectTips{cm}{}\object@{>}}^{\tilde E_{1}^+ \ \ \ \ }_\vdots  \ar @{-} @/^3.5pc/[uu]|-{\SelectTips{cm}{}\object@{>}}^{\tilde f_1^+}\ar @{-} @/_3.5pc/[uu]|-{\SelectTips{cm}{}\object@{>}}_{\tilde f_m^+} \ar @{-} @/_3.5pc/[rru]|-{\SelectTips{cm}{}\object@{>}}_{\tilde E_{n_2}^-} _<{\tilde v_3}  _>{\tilde v_2}\ar @{-} @/_1.5pc/[rru]|-{\SelectTips{cm}{}\object@{>}}_{\tilde E_{n_2}^+} &&\\
}
\end{equation*}

\caption{  Dual graph of $FS_{n_1+n_2}+FS_{m}$  degeneration of a Friedman--Smith example with $2n=2(n_1+n_2)\ge 2$ nodes.}\label{Fig:FSn+FSm}
\end{figure}
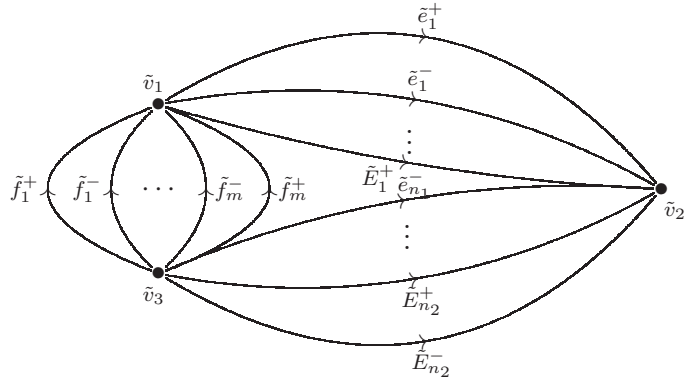

Geometrically, we have $n+m$ copies of $\delta_0'$ meeting.    A $FS_{n_1+n_2}+FS_m$ degeneration is a
 degeneration of a $FS_{2}$ or $FS_3$ example if and only if one (or more) of  $m+n_1$, $m+n_2$ or $n_1+n_2$ is equal to $2$ or $3$.   This is also a $DR_3$ example if $n_1=n_2=m=1$.

The case where $n_1$, $n_2$ or $m$ is $0$  is elementary  (the monodromy matrix can be made block diagonal) so we will ignore this in the analysis that follows.
One can show that
$$
H_1(\widetilde \Gamma,\mathbb Z)=\mathbb Z\langle  (\tilde e_1^+-\tilde e_1^-),\ldots, (\tilde e_{n_1}^+-\tilde e_{n_1}^-),
$$
$$
 (\tilde E_1^+-\tilde E_1^-),\ldots, (\tilde E_{n_2}^+-\tilde E_{n_2}^-),
$$
$$
 (\tilde e_1^+-\tilde e_2^-),\ldots, (\tilde e_{n_1-1}^+-\tilde e_{n_1}^-),
$$
$$
 (\tilde e_{n_1}^+-\tilde E_1^-)+\tilde f_1^+,
$$
$$
 (\tilde E_1^+-\tilde E_2^-),\ldots, (\tilde E_{n_2-1}^+-\tilde E_{n_2}^-),
$$
$$
\tilde f_1^+-\tilde f_1^-,\ldots,\tilde f_m^+-\tilde f_m^-,\tilde f_1^+-\tilde f_2^-,\ldots, \tilde f_{m-1}^+-\tilde f_m^- \rangle
$$
and
$$
H_1(\widetilde \Gamma,\mathbb Z)^{[-]}=\mathbb Z\langle  (\tilde e_1^+-\tilde e_1^-), \frac{1}{2}(\tilde e_1^+-\tilde e_1^-)+\frac{1}{2}(\tilde e_2^+-\tilde e_2^-),\ldots, \frac{1}{2}(\tilde e_{n_1-1}^+-\tilde e_{n_1-1}^-)+\frac{1}{2}(\tilde e_{n_1}^+-\tilde e_{n_1}^-),
$$
$$
\frac{1}{2}(\tilde E_1^+-\tilde E_1^-)+\frac{1}{2}(\tilde E_2^+-\tilde E_2^-)
,\ldots, \frac{1}{2}(\tilde E_{n_2-1}^+-\tilde E_{n_2-1}^-)+\frac{1}{2}(\tilde E_{n_2}^+-\tilde E_{n_2}^-),
$$
$$
(\tilde f_1^+-\tilde f_1^-), \frac{1}{2}(\tilde f_1^+-\tilde f_1^-)+\frac{1}{2}(\tilde f_2^+-\tilde f_2^-),\ldots, \frac{1}{2}(\tilde f_{m-1}^+-\tilde f_{m-1}^-)+\frac{1}{2}(\tilde f_{m}^+-\tilde f_{m}^-)
 \rangle.
$$

The matrix for the monodromy can then be put in the form:
\begin{equation}\label{eqnFSn+FSm}
\begin{tiny}
FS_{n_1+n_2}+FS_m \ \ \  \left(
\begin{array}{ccccc|ccccc}
2&-1&-1&&-1&0&&&0&0\\
&1&0&0&&0&&&0&0\\
&&&\ddots&&& &&\\
FS_{n}&&&&1&0&0&&0&0\\ \hline
0\cdots&0&|-1&\cdots &-1&2&-1&-1&&-1\\
0&0&0&&&&1&&&0\\
0&0&0&&&&&1&&0\\
&&&&&&&\\
&&&&&FS_{m}&&&&1\\
\end{array}
\right)
\end{tiny}
\end{equation}
The matrix is almost block-diagonal, with the $FS_{n}$ and $FS_{m}$ matrices.  The $(n+1)$-st row starts with $n_1$ zeros, then has $n_2$ negative ones, and has for its last $m$ entries the first row  of the $FS_{m}$ matrix.

\begin{lem}\label{lemFS+FS}
The monodromy cone for a $FS_{n_1+n_2}+FS_n$ degeneration is contained in a cone in the PCD if and only if the cover is not a degeneration of an $FS_2$ or $FS_3$ example (i.e.~if and only if $m+n_1, m+n_2, n_1+n_2\ne 2,3$).
\end{lem}

\begin{proof}
This is essentially identical to the proof of Proposition \ref{proFS=PC},  and is left to the reader.
\end{proof}

\begin{rem}
If all of the sums $n_1+n_2,n_1+m,n_2+m$ are divisible by $4$, then the monodromy cone  is  contained in a cone in the central cone decomposition.
\end{rem}

\begin{rem}
For the case $n_1=n_2=m=1$, it was observed above that this is also a $DR_3$ cone.  In particular, we have already worked out the decomposition into perfect (in fact matroidal) cones.   The other cases where  $FS_{n_1+n_2}+FS_m$ degenerations are degenerations of $FS_2$ or $FS_3$ covers  satisfy the condition that one of $n_1,n_2,m=0$ and so the decompositions can be described in terms of Friedman--Smith examples (see \S  \ref{secFSexamples}).
\end{rem}

Specifically, let us consider the cone decomposition for $FS_{2+0}+FS_1$.  The matrix for the monodromy cone can be put in the form
$$
\left(\begin{array}{ccc}
2&-1&0\\
0&1&0\\
0&0&1\\
\end{array}\right)
$$
Then one can easily check that the matrices
$$
\left(\begin{array}{ccc}
2&-1&0\\
0&0&1\\
1&-1&0\\
\end{array}\right) \ \  \
\left(\begin{array}{ccc}
1&-1&0\\
0&1&0\\
0&0&1\\
\end{array}\right)
$$
are matroidal.  In short, the perfect cone decomposition is given in Figure \ref{figFS20FS1decomp}.
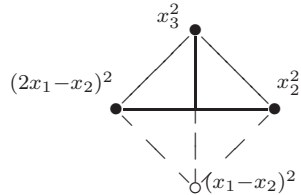
\begin{figure}[htb]
$$
\xymatrix{
&*{\bullet} \ar@{-}[rd]|-{\SelectTips{cm}{}\object@{}}^>{x_2^2} \ar@{-}[ld]|-{\SelectTips{cm}{}\object@{}}_<{x_3^2}  _>{(2x_1-x_2)^2} \ar@{-}[d]|-{\SelectTips{cm}{}\object@{}} & \\
*{\bullet} \ar@{-}[rr]|-{\SelectTips{cm}{}\object@{}}&*{} \ar@{--}[d]|-{\SelectTips{cm}{}\object@{}}^>{(x_1-x_2)^2}&*{\bullet}\\
&*{\circ} \ar@{--}[ru]|-{\SelectTips{cm}{}\object@{}} \ar@{--}[lu]|-{\SelectTips{cm}{}\object@{}}&\\
}
$$
\caption{Perfect cone decomposition of  an $FS_{2+0}+FS_{1}$  monodromy cone.}\label{figFS20FS1decomp}
\end{figure}
The dashed lines in the figure show the ambient matroidal cones giving the decomposition of the monodromy cone (which are depicted by  the solid lines).

\subsection{Friedman--Smith--$\delta_i$ degenerations } \label{secFS+d} Here we consider the case where one of the curves in a Friedman--Smith cover  degenerates to a generic cover in $\delta_i$ (or $\delta_{g+1-i}$).  We call these  $FS_{n_1+n_2}+\delta_i$ degenerations.

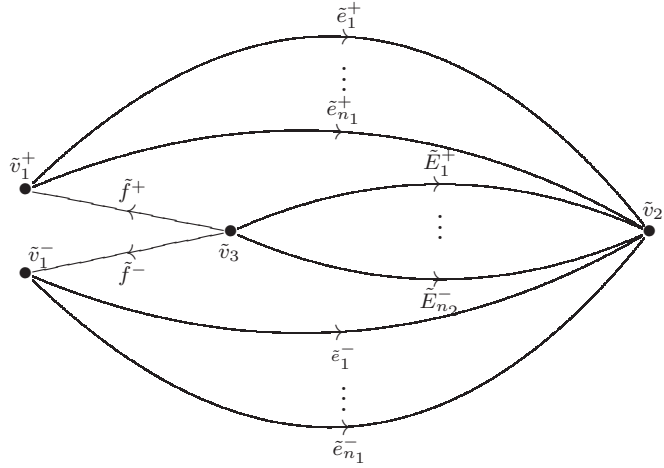
\begin{figure}[htb]
\begin{equation*}
\xymatrix@C=2.5cm@R=.05cm{
*{\bullet} \ar@{-}[rd]|-{\SelectTips{cm}{}\object@{<}}^{\tilde f^+} \ar @{-} @/^5.5pc/[rrrd]|-{\SelectTips{cm}{}\object@{>}}^{\tilde e_1^+} \ar @{-} @/^2.5pc/[rrrd]|-{\SelectTips{cm}{}\object@{>}}  ^>{\tilde v_2} ^<{\tilde v_1^+}^{\stackrel{\vdots}{\tilde e_{n_1}^+}} &&&\\
&*{\bullet}  \ar @{-} @/^1.5pc/[rr]|-{\SelectTips{cm}{}\object@{>}}^{\tilde E_1^+} \ar @{-} @/_1.5pc/[rr]|-{\SelectTips{cm}{}\object@{>}}_{\tilde E_{n_2}^-} _<{\tilde v_3}&\vdots&*{\bullet} \\
*{\bullet} \ar@{-}[ru]|-{\SelectTips{cm}{}\object@{<}}_{\tilde f^-} \ar @{-} @/_5.5pc/[rrru]|-{\SelectTips{cm}{}\object@{>}}_{\tilde e_{n_1}^-} \ar @{-} @/_2.5pc/[rrru]|-{\SelectTips{cm}{}\object@{>}}_{\stackrel{\tilde e_{1}^-}{\vdots}} ^<{\tilde v_1^-}&&&\\
}
\end{equation*}

\caption{  Dual graph of a $FS_{n_1+n_2}+\delta_i$   degeneration of a Friedman--Smith example with $2n\ge 2$ nodes.}\label{Fig:dgFSn+d}
\end{figure}

Geometrically, this is the $(n+1)$-fold intersection of $\delta_0'$ (unless $n_1=0$ in which case it is the $n$-fold intersection of $\delta_0'$ with $\delta_{i}$ (or $\delta_{g+1-i}$)).
This is a degeneration of an $FS_2$ or $FS_3$ cover if and only if $n_1+n_2=2,3$ or $n_2+1=2,3$.
One can show that
$$
H_1(\widetilde \Gamma,\mathbb Z)=\mathbb Z\langle  (\tilde e_1^+-\tilde e_1^-)+(\tilde f^+-\tilde f^-),\ldots, (\tilde e_{n_1}^+-\tilde e_{n_1}^-)+(\tilde f^+-\tilde f^-),
$$
$$
 (\tilde E_1^+-\tilde E_1^-),\ldots, (\tilde E_{n_2}^+-\tilde E_{n_2}^-),
$$
$$
 (\tilde e_1^+-\tilde e_2^-)+(\tilde f^+-\tilde f^-),\ldots, (\tilde e_{n_1-1}^+-\tilde e_{n_1}^-)+(\tilde f^+-\tilde f^-),
$$
$$
 (\tilde e_{n_1}^+-\tilde E_1^-)+\tilde f^+,
$$
$$
 (\tilde E_1^+-\tilde E_2^-),\ldots, (\tilde E_{n_2-1}^+-\tilde E_{n_2}^-) \rangle
$$
and
$$
H_1(\widetilde \Gamma,\mathbb Z)^{[-]}=\mathbb Z\langle  (\tilde e_1^+-\tilde e_1^-)+(\tilde f^+-\tilde f^-), \frac{1}{2}(\tilde e_1^+-\tilde e_1^-)+\frac{1}{2}(\tilde e_2^+-\tilde e_2^-)+(\tilde f^+-\tilde f^-),\ldots,
$$
$$
 \frac{1}{2}(\tilde e_{n_1-1}^+-\tilde e_{n_1-1}^-)+\frac{1}{2}(\tilde e_{n_1}^+-\tilde e_{n_1}^-) +(\tilde f^+-\tilde f^-),
$$
$$
 \frac{1}{2}(\tilde e_{n_1}^+-\tilde e_{n_1}^-)+\frac{1}{2}(\tilde E_{1}^+-\tilde E_{1}^-) +\frac{1}{2}(\tilde f^+-\tilde f^-),
$$
$$
\frac{1}{2}(\tilde E_1^+-\tilde E_1^-)+\frac{1}{2}(\tilde E_2^+-\tilde E_2^-)
,\ldots, \frac{1}{2}(\tilde E_{n_2-1}^+-\tilde E_{n_2-1}^-)+\frac{1}{2}(\tilde E_{n_2}^+-\tilde E_{n_2}^-)
 \rangle.
$$
The monodromy matrix can then be put in the form:
\begin{equation}\label{eqnFSdmat}
\begin{tiny}
FS_{n_1+n_2}+\delta_i  \ \ \  \ \left(\begin{array}{ccccccc}
2&-1&-1&&&&-1\\
0&1&0&&&&0\\
&0&1&0&&&\\
&&&\ddots&&&\\
&&&&1&&\\
&&&&& 1&\\
FS_{n}&&&&&& 1\\ \hline
2&0&-2&\cdots&|-1&\cdots&-1 \\
\end{array}\right)
\end{tiny}
\end{equation}
The bottom row is a string of the form $2$, $0$, $-2$, $0$, $-2, \cdots$ of length $n_1$ followed by a string of the form $-1, -1, -1, \cdots, -1$ of length   $n_2$.

\begin{rem}
For the $n_1=0$ case, we simply have the $FS_{n}$ matrix (the bottom row does not appear).  Note also that for the $n_2=0$ case, the bottom row is divided by $2$ (so that it is primitive).
\end{rem}

\begin{lem}\label{lemFS+d}  For the $FS_{n_1+n_2}+\delta_i$ examples, if  $n=n_1+n_2\le 5$, then the monodromy cone is \emph{not} contained in a cone in the PCD if and only if the example is a degeneration of a $FS_2$, $FS_3$ example.   Moreover:
\begin{enumerate}
\item If $ n_1=0,1$,  the monodromy cone is a $FS_{n}$ cone, and so it is contained in a cone in the PCD if and only if $n\ne 2,3$.

\item If $n_2=0$, the monodromy cone is the same as for Lemma \ref{lemmcFSW}. Thus  if in addition  $n\le 7$, then the monodromy cone  is \emph{not} contained in a cone in the PCD if and only if  $n=2,3$.

\item If $n_2=1,2$ or $n_1+n_2=2,3$,  the monodromy cone is \emph{not} contained in a cone in the PCD (these are exactly the  examples that are   degenerations of a $FS_2,FS_3$).

\end{enumerate}

\end{lem}

\begin{proof}  The only things to show are  the cases $n=4,5$.
 Suppose that $n=4$.  For the cases $n_1=0,1$, we can conclude by (1).  For the case $n_1=2,3$, we have $n_2=2,1$, so we can conclude by (3).   If $n_1=4$, then $n_2=0$, and then we can conclude by (2).

The same analysis works for $n=5$, except when $n_1=2$ and $n_2=3$ (monodromy matrix below left).
We thank  Mathieu Dutour Sikiri\'c for  providing us with the following metric $Q_5'$ (matrix below right), which shows that the cone is contained in (but not equal to)  a cone in the PCD.
$$
\left(\begin{smallmatrix}
2&-1&-1&-1&-1\\
0&1&0&0&0\\
0&0&1&0&0\\
0&0&0&1&0\\
0&0&0&0& 1\\ \hline
2&0&-1&-1&-1 \\
\end{smallmatrix}\right) \ \ \ \  \hskip .6 in
Q_5'=\left(
\begin{smallmatrix}
  \frac{2}{3}&   \frac{1}{6}&   \frac{1}{3}&   \frac{1}{3}&   \frac{1}{3} \\
    \frac{1}{6}&   \frac{2}{3}&    0&    0&    0 \\
    \frac{1}{3}&    0&   \frac{2}{3}&    0&    0 \\
    \frac{1}{3}&    0&    0&   \frac{2}{3}&    0 \\
    \frac{1}{3}&    0&    0&    0&   \frac{2}{3} \\
\end{smallmatrix}
\right)
$$
\end{proof}

\subsection{Friedman--Smith--Beauville degenerations $FS_n+B_m$} \label{secFS+B}

In this section, we consider the case where we replace a vertex of a Friedman--Smith graph in $FS_n$ with a Beauville example with $m$ edges (that is we simply add $m$ fixed loops to the graph at one of the vertices).   We will denote these degenerations by $FS_n+B_m$.

\begin{lem}\label{AlemFSB}
For an $FS_n+B_m$ degeneration, the monodromy cone is the same as for an $FS_n$ cover. Consequently, the monodromy cone is \emph{not} contained in a cone in the PCD if and only if $n=2,3$.
\end{lem}

\begin{proof}
See Lemma \ref{lemMCloop}.
\end{proof}

\subsection{Friedman--Smith--Elementary \'Etale  degenerations $FS_n+EE_m$} \label{secFS+EE}

In this section, we consider the case where we replace a vertex of a Friedman--Smith graph in $FS_n$ with an elementary example with $2m$ edges (that is we simply add $2m$  loops to the graph at one of the vertices, with the loops interchanged pairwise by the involution).   We will denote these degenerations by $FS_n+EE_m$.

\begin{lem}\label{AlemFSEE}
For an $FS_n+EE_m$ degeneration, the monodromy matrix can be put in the form
$$
\left(
\begin{array}{c|c}
FS_n& 0 \\\hline
0& Id_m\\
\end{array}
\right)
$$
 Consequently, the monodromy cone is \emph{not} contained in a cone in the PCD if and only if $n=2,3$.   For $n=2,3$, the decomposition into cones in the PCD (as well as SVD and CCD) is given by the star subdivision associated to the face associated to the $FS_n$ cone.
\end{lem}

\begin{proof}
For the monodromy matrix, see Lemma \ref{AlemD}.   The statements about the decomposition follow directly from the definitions, and the block diagonal form of the monodromy matrix.
\end{proof}

From our computations in these different examples, we are led to the natural question:

\begin{que}
For monodromy cones associated to degenerations of Friedman--Smith covers of type $FS_2$ or $FS_3$, is  the    second Voronoi decomposition  a refinement of the perfect cone decomposition?
\end{que}


\section{Simplifications of monodromy cones}\label{secSimp}

In this section we discuss some morphisms of homology of dual graphs, which are used in the discussion of the Hodge theory.     These morphisms  also lead to  some techniques to simplify the study of the monodromy cones.
While we do not strictly need these for the content of this paper, we do include some lemmas below that are similar in spirit, and which we do use to simplify some computations in the preceding sections.

Let $\Gamma$ be a graph, and let $S$ be a subset of the unoriented edges of $\Gamma$.  Define  $\Gamma-S$  to be the graph obtained by removing the edges in $S$, by which we mean removing both oriented edges lying over an unoriented edge, and $\Gamma/S$ to be the graph obtained by contracting  the edges in $S$.

\begin{rem}
If $\Gamma$ is the dual graph of a stable curve $X$, and $S$ corresponds to nodes $p_1,\ldots,p_n\in X$, then $\Gamma-S$ is the dual graph of the curve obtained from $X$ by normalizing at the nodes $p_1,\ldots,p_n$, and $\Gamma/S$ is the dual graph of the curve obtained from $X$ by smoothing the nodes $p_1,\ldots,p_n$.
\end{rem}

Since $\Gamma-S\subseteq \Gamma$, there is a natural inclusion of complexes
$$
  C_\bullet(\Gamma-S,\ZZ)\hookrightarrow C_\bullet(\Gamma,\ZZ).
$$
There is also a surjective morphism of complexes
$$
  C_\bullet(\Gamma,\ZZ)\twoheadrightarrow C_\bullet(\Gamma/S,\ZZ)
$$
given in the following way.  For an edge $e$, then $e\mapsto 0$ if $e\in S$, and otherwise, $e\mapsto e$.   The vertices of $\Gamma/S$ correspond to equivalence classes of vertices in $\Gamma$ joined by edges in $S$.    The map on vertices is the quotient map.    Consequently, setting $S^c$ to be the set of unoriented edges of $\Gamma$ complementary to those of $S$, there are natural maps
$$
H_i(\Gamma-S,\ZZ)\to H_i(\Gamma,\ZZ)\to H_i(\Gamma/S^c,\ZZ).
$$

\begin{rem}\label{remSESH1}
While the sequence of complexes
$$
C_\bullet(\Gamma-S,\ZZ)\hookrightarrow C_\bullet(\Gamma,\ZZ) \twoheadrightarrow C_\bullet(\Gamma/S^c,\ZZ)
$$
is \emph{not} exact,  the following sequence is (split) exact:
$$
0\to H_1(\Gamma-S,\ZZ)\to H_1(\Gamma,\ZZ)\to H_1(\Gamma/S^c,\ZZ)\to 0.
$$
The analogous  sequence in degree $0$ need not be exact.    From the exact sequence on $H_1$, in many instances one may choose bases to obtain block triangular matrices in monodromy computations for Pryms.  As an application, one can give a short  combinatorial proof  that the monodromy cone of a degeneration of a Friedman--Smith cover with at least $4$ nodes ($FS_n$, $n\ge 2$)  is not matroidal.  This provides an alternate combinatorial  proof of \cite[Prop.~2.1 and p.120]{vologodsky}.
\end{rem}

The observations in the remark lead to some general techniques for simplifying monodromy cones for Prym varieties.  For brevity, we omit these.  Two lemmas with similar statements and proofs, which we do use in earlier arguments,  are stated below.

\begin{lem}\label{AlemD} Let $\widetilde \Gamma$ be a graph with admissible involution $\iota$.  Suppose that $\widetilde \Gamma_1,\widetilde \Gamma_2\subseteq \widetilde \Gamma$ are connected sub-graphs preserved by the involution $\iota \widetilde \Gamma_i=\widetilde \Gamma_i$ ($i=1,2$).
If $\widetilde \Gamma=\widetilde \Gamma_1\cup \widetilde \Gamma_2$ and $\widetilde \Gamma_1\cap \widetilde \Gamma_2=\{v\}\subseteq V(\widetilde \Gamma)$ is a single vertex of the graph, then  the  matrix defining the monodromy cone can be put in the form:
$$
MC(\widetilde \Gamma)=\left(
\begin{array}{c|c}
MC( \widetilde \Gamma_1)& 0\\ \hline
 0 & MC(\widetilde \Gamma_2)\\
\end{array}
\right)
$$
\end{lem}

\begin{proof}
The proof is left to the reader.
\end{proof}

\begin{rem}
In particular, fix an admissible collection of admissible cone decompositions $\Sigma$.   Then in the notation of the lemma, $MC(\widetilde \Gamma)$ is contained in a cone in $\Sigma$ if and only if $MC(\widetilde \Gamma_1)$ and $MC(\widetilde \Gamma_2)$ are contained in cones in $\Sigma$.
\end{rem}

\begin{lem}\label{lemMCloop}
Let $\widetilde \Gamma$ be a graph with admissible involution $\iota$. Suppose that  $\widetilde \Gamma'$ is a graph with admissible involution $\iota'$, which is obtained from $\Gamma$ by adding a single loop (an edge $\tilde e'$ such that $s(\tilde e')=t(\tilde e')$).  Then
$MC(\widetilde \Gamma)=MC(\widetilde \Gamma')$.
\end{lem}

\begin{proof}
Necessarily $\iota'(\tilde e')=\tilde e'$.  Then apply the previous lemma  with  $\widetilde \Gamma_1=\widetilde \Gamma$ and $\widetilde \Gamma_2=(s(\tilde e'),\tilde e')$.
\end{proof}

\bibliography{prymbib}



\vfil\eject

\section{Extension to the central cone decomposition\\by Mathieu Dutour Sikiri\'c} \label{secappMDS}
In this appendix we discuss the extension of the Prym map to the central cone compactification. We shall see that the indeterminacy loci
differ substantially  for the three toroidal compactfications $\bar A_g^V$,  $\bar A_g^P$ and $\bar A_g^C$.

\begin{teo}\label{teoindC}
For the extension of the Prym map $P_g^C:\overline R_{g+1}\dashrightarrow \bar A_g^C$ to the central cone compactification the following holds:
\begin{enumerate}
\item $\overline {FS}_2\cup \overline {FS}_3\subseteq Ind(P_g^C)$, and for $n\geq 4$  the strata 
${FS}_n$
are not contained in $Ind(P_g^C)$.
\item If $g \geq 9$, the indeterminacy locus $Ind(P_g^C)$ contains points that are not contained in $\cup_{n\geq 1}\overline{FS}_n$.  
\end{enumerate}
\end{teo}
\begin{proof}

We shall first prove (1).
In genus $2$ and $3$ the central cone decomposition coincides with the second Voronoi and the perfect cone decomposition. In particular the monodromy cones of $FS_2$ and
$FS_3$ are not contained in a cone in the central cone decomposition (CCD). This shows the inclusion $\overline {FS}_2\cup \overline {FS}_3\subseteq Ind(P_g^C)$.
To complete the proof of (1) it remains to show that the monodromy cones of $FS_n$ are contained in a cone of the  CCD for $n\geq 4$. For this we work with the representation of $FS_n$
given by \eqref{eqnFSMC}, namely
 $$
FS_n \ \ \  \left(\begin{smallmatrix}
2&-1&-1&&&&-1\\
0&1&0&&&&0\\
0&0&1&0&&&\\
&&&\ddots&\ddots&&\\
&&&0&1&0&0\\
&&&&0&1&0\\
&&&&&0&1\\
\end{smallmatrix}\right)
$$
We will apply the criterion of Lemma \ref{lemcc} using the following quadratic form
$$Q_C=
\left(\begin{smallmatrix}
1           &  \frac12 &  \frac12 &  \frac12 &    0     &  \ldots  &  \ldots  & 0\\
 \frac12    &    1     &    0     &    0     &    0     &  \ldots  &  \ldots  & 0\\
 \frac12    &    0     &    1     &    0     &    0     &  \ldots  &  \ldots  & 0\\
 \frac12    &    0     &    0     &    1     & -\frac12 &    0     &  \ldots  & 0\\
0           &    0     &    0  & -\frac12 &    1     & -\frac12 &  \ddots  & 0\\
\vdots      &  \vdots  &  \vdots  &    0     & -\frac12 &    1     &  \ddots  & 0\\
\vdots      &  \ddots  &  \ddots  &  \vdots     &  \ddots     &  \ddots  &  \ddots  & -\frac12\\
0           &  \ldots  &  \ldots  &  \ldots  &  \ldots  &    0     & -\frac12 & 1
\end{smallmatrix}\right)
$$
Clearly this matrix is integer valued, and one computes immediately that for all rows $\ell_i$ of the $FS_n$ matrix  one has $Q_C(\ell_i)=1$. To
prove the claim it remains to show that $Q_C$ is positive definite. To see this we first note that $Q_C$ is equivalent to
$$Q'_C=
{\small
\left(\begin{smallmatrix}
1           & -\frac12 & -\frac12 & -\frac12 &    0     &  \ldots  &  \ldots  & 0\\
-\frac12    &    1     &    0     &    0     &    0     &  \ldots  &  \ldots  & 0\\
-\frac12    &    0     &    1     &    0     &    0     &  \ldots  &  \ldots  & 0\\
-\frac12    &    0     &    0     &    1     & -\frac12 &    0     &  \ldots  & 0\\
0           &    0     &    0  & -\frac12 &    1     & -\frac12 &  \ddots  & 0\\
\vdots      &  \vdots  &  \vdots  &    0     & -\frac12 &    1     &  \ddots  & 0\\
\vdots      &  \ddots  &  \ddots  &  \vdots     &  \ddots     &  \ddots  &  \ddots  & -\frac12\\
0           &  \ldots  &  \ldots  &  \ldots  &  \ldots  &    0     & -\frac12 & 1
\end{smallmatrix}\right)
}
$$
But then taking $e_1,\ldots,e_n$ as a $\mathbb Z$-basis of $\mathbb Z^n$, we have $Q'_{C}(e_i)=1$, $Q_{C'}(e_1,e_i)= -\frac12$ for $i=2,3,4$ and $Q'_C(e_i,e_{i+1})=-\frac12, i=4, \ldots n-1$ with $Q'_C(e_k,e_l)=0$ in all other cases.
This shows that $Q'_C$ is the form $D_n$ (up to scalar) and this shows positive definiteness.

We shall now prove (2). The starting point is that by \cite{ab, AETAL} there are stable curves of genus $g\geq 9$ near which the Torelli map to $\bar A_g^C$ is not a morphism.
We start with two copies of a curve $C$ of genus $g$ with $2$ marked points, say $P,Q$. Then we attach $P$ to $Q$ and vice versa.
On the resulting curve we can define an involution that swaps the two components and the two nodes. The associated Prym variety is then equal 
to the Jacobian of $C$. We can now degenerate this involution to an example which proves our claim. Indeed, since $g \geq 9$ we can degenerate $C$ to a curve $C'$ which lies in the indeterminacy locus of the Torelli map 
to $\bar A_g^C$.
We now choose one component of $C'$, and attach it to a second copy of $C'$ as in the
Wirtinger example with 2 nodes (to the same chosen component on the second copy of $C'$).  Again we let
the involution swap the two copies of $C'$ (and the chosen attaching nodes).
The monodromy cone is then the same as the monodromy cone associated to $C'$, (cf.~the Wirtinger Example \ref{secWE}).
This shows that the map $P_g^C:\overline R_{g+1}\dashrightarrow \bar A_g^C$ is not a morphism near this cover.
It remains to check that this is not contained in $\cup_{n\geq 1} \overline{FS}_n$. But this follows from considering the dual graph: since the vertices of this graph
are interchanged pairwise it cannot be contracted to the graph (with involution) of an $FS_n$ example. 
\end{proof}

\end{document}